\documentclass[11pt]{ip-journal}

\usepackage{amsmath, graphicx, cite, enumerate, comment}
\usepackage{amssymb, amsthm, amscd}
\usepackage{latexsym}
\usepackage{amssymb}
\usepackage[english]{babel}
\usepackage{amsmath,amssymb}
\usepackage{amsfonts}
\usepackage{amscd}
\usepackage{cite}
\usepackage{marvosym}
\usepackage{mathrsfs}
\usepackage{amsthm}
\usepackage{hyperref}
\usepackage{verbatim}
\usepackage{setspace}

\usepackage{tikz} \usetikzlibrary{arrows,shapes,patterns} 
\usepackage{lipsum} 

\newtheorem{thm}{Theorem}[section]
\newtheorem{lem}[thm]{Lemma}
\newtheorem{cor}[thm]{Corollary}
\newtheorem{prop}[thm]{Proposition}
\newtheorem{defi}[thm]{Definition}

\theoremstyle{remark}
\newtheorem{rmk}[thm]{Remark}
\newtheorem{ex}[thm]{Example}

\numberwithin{equation}{section}

\newcommand{\R}{\mathbb{R}}
\newcommand{\p}{\partial}
\newcommand{\e}{\epsilon}
\newcommand{\g}{\gamma}
\newcommand{\de}{\delta}

\newcommand{\s}{\sigma}

\newcommand{\mA}{\mathcal{A}}

\newcommand{\mH}{\mathcal{H}}

\DeclareMathOperator{\Div}{div}
\DeclareMathOperator{\inj}{inj}

\DeclareMathOperator{\diff}{Diff}

\newcommand{\be}{\begin{equation}}
\newcommand{\ee}{\end{equation}}
\newcommand{\bd}{\begin{displaymath}}
\newcommand{\ed}{\end{displaymath}}

     \title[Constrained deformations of PSC metrics]{Constrained Deformations of Positive Scalar curvature metrics}
     \author{Alessandro Carlotto and Chao Li}
     \address{ \noindent Alessandro Carlotto: 
     	\newline ETH D-Math, R\"amistrasse 101, 8092 Z\"urich, Switzerland 
     	\newline IAS, 1 Einstein drive, 08540 Princeton, United States of America
     	\newline
     	 \textit{E-mail address: alessandro.carlotto@math.ethz.ch, alessandro.carlotto@ias.edu} 
     	 \newline \newline \indent Chao Li: 
     	 \newline Princeton University - Department of Mathematics, Fine Hall, 304 Washington Road, 08544 Princeton, United States of America
     	 \newline \textit{E-mail address: chaoli@math.princeton.edu} }
    
\begin{document}
     
   	\begin{abstract}
     	We present a series of results concerning the interplay between the scalar curvature of a manifold and the mean curvature of its boundary.
     	In particular, we give a complete topological characterization of those compact 3-manifolds that support Riemannian metrics of positive scalar curvature and mean-convex boundary and, in any such case, we prove that the associated moduli space of metrics is path-connected.  The methods we employ are flexible enough to allow the construction of continuous paths of positive scalar curvature metrics with minimal boundary, and to derive similar conclusions in that context as well.
     	Our work relies on a combination of earlier fundamental contributions by Gromov-Lawson and Schoen-Yau, on the smoothing procedure designed by Miao, and on the interplay of Perelman's Ricci flow with surgery and conformal deformation techniques introduced by Cod\'a Marques in dealing with the closed case. 
     	\end{abstract}

     	\maketitle      
     
     	\tableofcontents

	\section{Introduction} \label{sec:intro}
	
	Let $X^n$ be a compact $n$-dimensional manifold with boundary $\partial X$. 
	When we endow it with a Riemannian metric $g$, one can consider the following two geometric functions:
	\[
	R_g: X\to\mathbb{\R},  \ \  \ \ H_g:\partial X\to\R,
	\]
	describing the \emph{scalar curvature} of $X$ and the \emph{mean curvature} of its boundary, respectively.

	 We shall be concerned here with certain geometrically significant subsets of Riemannian metrics defined by pointwise inequalities involving the functions $R$ and $H$. As a paradigmatic case, we wish to gain some understanding of the space
	 $\mathcal{M}=\mathcal{M}_{R>0, H>0}$ consisting of those Riemannian metrics on $X$ that have positive scalar curvature and (strictly) mean-convex boundary. Two well-known, fundamental questions one may ask, in this respect, are whether for a given $X$ the set $\mathcal{M}$ is not empty, and, if so, whether one can say something about the topology of this space, and draw at least some partial conclusions about its homotopy type. 
	 
	 These two questions naturally fit in a much larger picture, describing the interactions between interior and boundary curvature conditions. When such conditions are defined by inequalities involving the scalar curvature of the manifold, on the one hand, and the mean curvature of its boundary, on the other hand, such interactions have proven to be rather elusive.
	 We refer the reader to the recent article \cite{Gro18} by M. Gromov for a survey of various results related to this general theme, and for a list of significant open problems. It is also appropriate to point out how some of these questions naturally arise in the study of initial data sets for the Einstein equations, in which context the curvature bounds are motivated by specific physical axioms (see, for instance, the monographs \cite{HE73, Wal84} and references therein).

	 In the two-dimensional case, namely when $n=2$, we know (thanks to the Gauss-Bonnet theorem) that $\mathcal{M}\neq\emptyset$ only when $X$ is a disk and that, by the uniformization theorem, $\mathcal{M}$ is always path-connected. It is then natural to wonder to what extent similar conclusions still hold true in the higher-dimensional scenario, or whether instead new phenomena occur. Our first theorem gives a complete characterization of those compact 3-manifolds that support Riemannian metrics of positive scalar curvature and mean-convex boundary.

	 \begin{thm}\label{thm:A}
	 	Let $X^3$ be a connected, orientable, compact manifold with boundary, such that $\mathcal{M}\neq\emptyset$.
	 	Then there exist integers $A,B,C\ge 0$, such that $X$ is diffeomorphic to a connected sum of the form
	 	\[ P_{\g_1}\#\cdots\# P_{\g_A}\# S^3/\Gamma_1\#\cdots\# S^3/\Gamma_B\# \left(\#_{i=1}^C (S^2\times S^1)\right).\]
	 	Here $P_{\g_i}$, $i\le A$, are genus $\g_i\geq 0$ handlebodies, and $\Gamma_i$, $i\le B$, are finite subgroups of $SO(4)$ acting freely on $S^3$. Viceversa, any such manifold supports Riemannian metrics of positive scalar curvature and mean-convex boundary.
	 \end{thm}
 
 One of the two implications follows quite directly from Remark \ref{rmk:Lawson} below by means of the Gromov-Lawson connected sum construction \cite{GL80b}, while the converse (providing the topological characterization given the curvature conditions) is much more subtle.
 We refer the reader to Section \ref{sec:TopChar} for a series of remarks related to this statement, the discussion of various simple examples and the proof of this result. 
	 
	 Let us then move to the second question we posed. When $n\geq 3$ it is still true that the $n$-dimensional disk (the closure of the $n$-dimensional ball) supports metrics of positive scalar curvature and convex boundary (like those corresponding to round spherical caps) so it makes sense to ask whether $\mathcal{M}$ is path-connected in this case as well. In this article, we provide an affirmative answer to this question when $n=3$, and in fact we present a result that applies to the moduli space of \emph{any} compact 3-manifold with boundary. To state our second main theorem, we let $\mathcal{D}$ be the class of diffeomorphisms of $X$ and $\mathcal{M}/\mathcal{D}$  be the associated moduli space of metrics. Then, the following conclusions hold:
	 
	 	\begin{thm}\label{thm:B}
	 		Let $X^3$ be a connected, orientable, compact manifold with boundary. If $\mathcal{M}$ is not empty, then the moduli space $\mathcal{M}/\mathcal{D}$ is path-connected. 
	 		When $X^3\simeq D^3$ then the space $\mathcal{M}$ is itself path-connected.
	 \end{thm}
 
 Even in the special case of the $n$-dimensional disk, we are not aware of any similar result, in either direction, for any $n\geq 4$; also there seems to be very fragmentary information about higher homotopy groups $\pi_k(\mathcal{M}(D^n))$ whenever $n\neq 2$.
The possibility of employing the first statement to derive the somewhat stronger conclusion presented in the second one relies on a deep result by J. Cerf \cite{Cer68} asserting the contractibility of $\mathcal{D}_+(D^3)$, the class of orientation-preserving diffeomorphisms of $X$.
  
  In fact, the techniques we employ to prove the two theorems above also allow to study the larger moduli spaces that are obtained by weakening one, or both inequalities that define $\mathcal{M}$. Most significantly, and somewhat surprisingly, we can even refine our deformation schemes to obtain continuous paths of positive, or non-negative, scalar curvature metrics with minimal boundary. In particular, Theorem \ref{thm:B} above has a natural counterpart for the space $\mathcal{H}:=\mathcal{H}_{R\geq 0, H=0}$, see Theorem \ref{thm:B'}, which in turn implies the path-connectedness of asymptotically flat Riemannian 3-manifolds with non-negative scalar curvature and minimal boundary (cf. \cite{Sch84, Mar12}). 
	 
In order to contextualize these theorems, it is convenient to first compare them with our current understanding of the corresponding questions in the case of closed manifolds, a setting where much more is known.

\subsection{Connection to the closed case}\label{subs:closed}

For a given compact, orientable manifold $M^n$ without boundary, let $\mathcal{R}$ denote the (possibly empty) space of positive scalar curvature metrics and let $\mathcal{R}/\mathcal{D}$ denote the associated moduli space, for $\mathcal{D}$ the diffeomorphisms of $M$. As above, one wonders whether $\mathcal{R}\neq\emptyset$ and, in that case, what can be said on $\pi_{\ast}(\mathcal{R})$. These questions turn out to be, in general, \emph{highly sensitive to the value of $n$}.
When $n=2$, we obviously know that $\mathcal{R}\neq\emptyset \ \Rightarrow M\simeq S^2$ (a topological sphere) and it is a theorem due to H. Weyl \cite{Wey15} that $\mathcal{R}$ is path-connected (for any metric can be linked, via a conformal path, to a round representative), while much later J. Rosenberg and S. Stolz proved that $\mathcal{R}$ is actually contractible \cite{RS01}. When $n=3$ we know (cf. work by Gromov-Lawson \cite{GL80b}, Schoen-Yau \cite{SY79a, SY79b, SY82}, Perelman \cite{Per02, Per03a, Per03b}) that $\mathcal{R}$ is not empty only if $M$ takes the form of a (finite) connected sum $S^3/\Gamma_1 \# \ldots \# S^3/\Gamma_p \# S^2\times S^1 \#_q S^2\times S^1$; in each of these cases it was proven by F. Cod\'a Marques \cite{Mar12} that the moduli space $\mathcal{R}/\mathcal{D}$ is path-connected. We will say more on this result, which plays a fundamental role in our work, in the next section when providing an outline of the proofs of our main theorems.
When $n\geq 4$ the scenario in front of us is rather different, and to some extent still far from being fully understood. A remarkable characterization result for simply connected compact manifolds, of dimension at least five, admitting a positive scalar curvature metric was proven by S. Stolz in 1992, see \cite{Sto92}. Concerning the path-connectedness issue, we mention a few landmark results:
\begin{itemize}
 \item {in 1974 N. Hitchin \cite{Hit74} proved, via spinorial methods, that $\mathcal{R}(S^{8k})$ and $\mathcal{R}(S^{8k+1})$  are disconnected for all $k\geq 1$;}
 \item {in 1983 M. Gromov and H. Lawson \cite{GL83} proved that $\mathcal{R}(S^{7})$ has infinitely many connected components;}
 \item {in 1988 R. Carr \cite{Car88} proved that the space $\mathcal{R}(S^{4k-1})$ has infinitely many connected components for all $k\geq 2$;}
 \item {in 1993 M. Kreck and S. Stolz \cite{KS93} proved that the moduli space $\mathcal{R}(S^{4k-1})/\mathcal{D}(S^{4k-1})$ has infinitely many connected components for all $k\geq 2$;}
 \item {in 1996 B. Botvinnik and P. Gilkey \cite{BG96} proved that this property of the moduli space $\mathcal{R}(M^n)/\mathcal{D}(M^n)$ holds true whenever $M$ is a nontrivial spherical space quotient of dimension greater or equal than five (thereby providing, in particular, examples that have dimension exactly equal to 5);}
 \item {in 2014 B.	Hanke, T. Schick and W. Steimle \cite{HSS14} constructed, among other things, elements of infinite order in $\pi_k(\mathcal{R}(M^n), g_0)$ for any $k\in\mathbb{N}$ and correspondingly large dimension $n$ (see Theorem 1.1 therein for a precise statement).}
\end{itemize}
For other recent, significant contributions related to these problems the reader may wish to consult e.~g. \cite{BHSW10, BER17, CH16, Wal11, Wal14b} among others.

\subsection{Related results for manifolds with boundary}\label{subs:survey}	  
	  
As it has been anticipated above, much less is known for the corresponding questions in the setting of compact manifolds with boundary. In fact, some of the techniques that lie behind most of the results mentioned in the previous section do not have a straightforward extension to settings where boundary conditions come into play. Besides the $n=2$ case, which can be dealt with `classical' tools, the territory is still partly unexplored. That being said, there are some notable exceptions to this statement.

 The first one we wish to mention is provided by a beautiful paper \cite{AMW16} by A. Ach\'e, D. Maximo and H. Wu, where it is proven that the space $\mathcal{C}$ of metrics of positive Ricci curvature and convex boundary on a compact 3-manifold (necessarily a 3-disk by \cite{MSY82}) is path-connected. From there, they further derive that the corresponding moduli space (i.~e. the quotient modulo orientation-preserving diffeomorphisms of the 3-disk) is actually contractible.

 The strategy behind their proof, which was highly inspirational to us, can be summarized as follows. Given a metric $g_0 \in \mathcal{C}$, one first constructs an isotopy through metrics of positive Ricci curvature and \emph{weakly} convex boundary connecting it to a metric $g_1$ for which the boundary is totally geodesic (in fact a metric that can be smoothly symmetrized so to obtain a closed Riemannian manifold of positive Ricci curvature). Then, as a second step, one employs Hamilton's (normalized) Ricci flow to obtain an isotopy connecting $g_1$ to a round hemispherical metric. One then only needs to slightly perturb the concatenated path to accomodate the (strict) boundary convexity requirement without decreasing the Ricci curvature too much.

Now, as we shall describe in the next section, this conceptual scheme also lies behind the proof of our Theorem \ref{thm:B}, although serious technical obstacles appear both in the first and in the second part of the argument, preventing any attempt of cheap extensions of the results in \cite{AMW16} to the setting we are considering.
In the first part of the argument, we need to design a completely different symmetrization scheme, as the one presented in \cite{AMW16} (which in turn relies on an idea described by G. Perelman in \cite{Per97}, and worked out in detail in the Ph.D. thesis of H.-H. Wang, cf. \cite{Wan99}) does \emph{not} preserve the curvature conditions we deal with. In the second part of the argument, we cannot simply rely on the work by Hamilton, but need to work with the Ricci flow with surgery, and keep track (in a careful fashion) of the various pieces that originate when we cut along necks. This issue needs, to be properly dealt with, an array of specific tools that we will list when describing the general structure of our proof.

A second notable contribution to the subject is instead provided by the article by M. Walsh \cite{Wal14a}, which deals with the same \emph{interior} curvature conditions we are concerned with (positive scalar curvature) but with somewhat different boundary conditions, of \emph{collar} type. In that setting, the author extended to compact manifolds with boundary various foundational results, such as certain surgery results \`a la Gromov-Lawson. Thereby, he was able to derive deep conclusions on various finer aspects of the topology of these spaces of metrics, including information on their homotopy groups.
We refer the reader to Section 1.1 therein for the statements of the main results, see in particular Corollary D. Thirdly, we mention the recent paper \cite{BK18} by B. Botvinnik and D. Kazaras, aimed at investigating psc-bordism problems for compact manifolds without boundary.

 \subsection{Outline of the proofs}\label{subs:outline}
 
 We shall now provide a more detailed, yet non-technical description of the proof of our main results. This description can also be regarded as a short guided tour through the contents of this article. 
 
 The starting point for our project was an enlightening observation by Gromov-Lawson (see Theorem 5.7 in \cite{GL80a}), who first pointed out that if a compact manifold with boundary supports metrics of positive scalar curvature and (strictly) mean-convex boundary then its double can always be endowed with a metric of positive scalar curvature. In particular, this fact allows to reduce the first question we considered to a purely topological matter, i.~e. to the problem of chacterizing those compact manifolds $X$ whose double takes the form of a finite connected sum of handles and spherical space forms (cf. Section \ref{subs:closed}). To our surprise, such a result seems not to be in the literature, so we provided a detailed proof in Section \ref{sec:TopChar}. On the other hand, we wish to mention here two aspects which might alert the reader and indicate, at least in some vague sense, some of the subtleties that arise in aiming at such a characterization. First: in general the map $X\mapsto DX=M$, that associates to a compact manifold with boundary its double (an element of the list provided above) is highly non-injective. For instance, if $M\simeq S^2\times S^1$ then $X$ could either take the form $S^2\times I$ or $D^2\times S^1$, and of course this phenomenon manifests itself in greater complexity when dealing with the general case. Second: we wish to stress something that follows from the statement of Theorem \ref{thm:A}, namely that the condition that $\mathcal{M}\neq\emptyset$ does not place any restriction on the topological type of $\partial X$. This is discussed in the following remark.
 
 \begin{rmk}\label{rmk:Lawson}{}
 	Considering the round metric on $S^3$ we know, by virtue of Lawson's construction in \cite{Law70}, the existence of minimal surfaces of any given genus $\gamma$; any such surface is obviously unstable hence one can consider a positive speed deformation by means of the first eigenfunction of the Jacobi operator so to determine a smooth subdomain of the ambient manifold, having the topology of a genus $\gamma$ handlebody, whose boundary is (strictly) mean-convex. Hence, a Riemannian compact 3-manifold $(X^3,g)$ whose boundary is the disjoint union of closed surfaces of pre-assigned genera can be obtained from those building blocks by means of Gromov-Lawson connected sums. In particular, this observation allows to justify the second assertion in Theorem \ref{thm:A}.
 \end{rmk}
 
 From there, we started to wonder whether this Gromov-Lawson construction could be implemented in such a way that a whole isotopy (i.~e. a continuous path of metrics) could be derived. More concretely, the very first step in the proof of Theorem \ref{thm:B} is to deform a given metric to a `symmetrizable' one, moving through metrics of positive scalar curvature and weakly mean-convex boundary  (from there one can further deform, a posteriori, to gain the strict inequality, cf. Appendix \ref{sec:push-in}). The core statement we prove is as follows:

		\begin{prop}\label{thm:tgdef}
		Let $X^3$ be a connected, orientable, compact manifold with boundary, endowed with a Riemannian metric $g\in \mathcal{M}$. Then there exists a continuous path of smooth metrics on $X^3$, $\mu\in [0,1]\mapsto g_{\mu}\in \mathcal{M}_{R>0, H\geq 0}$, such that $g_0=g$ and $g_1$ has  totally geodesic boundary.  $(X^3, g_1)$ doubles to a smooth Riemannian manifold of positive scalar curvature.
		\end{prop}

	In \cite{GL80a} the double of $X$ endowed with a positive scalar curvature metric was essentially obtained by considering the boundary of the $\e$-neighborhood of $X'\cong X'\times\left\{0\right\}\subset X\times \R$ where $X'\subset X$ is obtained via a small inward deformation of $\partial X$. Such a double does \emph{not} inherit a smooth Riemannian metric in general but, as it was observed, there are various ways of smoothing the double keeping the scalar curvature positive. Since we are interested in obtaining an isotopy it is instead crucial for our scopes to precisely design a special smoothing scheme and to work out the details with great care. 
	Our desingularization crucially exploits some well-known work by P. Miao, related to non-smooth versions of the positive mass theorem \cite{Miao02}. Our situation is somewhat different, and perhaps more favorable, than the one handled there: the two interfaces match at $C^1$-level (rather than $C^0$ with the appropriate inequality for the mean curvature) so we can actually simplify the argument and avoid some delicate issues that arise there. On other hand, the fact that a desingularization scheme can be designed so that not only the scalar curvature is kept positive, but also all the leaves of a suitable foliation do not cease to be mean-convex is quite remarkable (and crucial in the context of our work).
	This part of the proof is presented in Section \ref{sec:elldef}, see also Appendix \ref{sec:Neu} for a conformal deformation lemma we need to employ.
	
	\
	
	The net outcome of this first step is then a Riemannian 3-manifold than can be smoothly symmetrized, along the totally geodesic boundary, to a closed one endowed with a metric of positive scalar curvature:
	Thereby, we are led to work with objects that we shall call \emph{reflexive $n$-manifolds}. We refer the reader to Section \ref{sec:EquivMan} for the basic definitions, that (although rather technical) naturally encode the properties of the geometric objects coming out of the above construction.

		The core of our proof is then aimed at obtaining an equivariant isotopy, which one can then restrict to, say, the `upper-half of the doubled manifold' to ultimately prove Theorem \ref{thm:B}. To formalize and specify this idea, we define the notion of reflexive isotopy (cf. Definition \ref{def:Z2isot}). 
			In those terms, we aim at constructing a reflexive isotopy of classes connecting the reflexive manifold produced above by symmetrization to a `standard' endpoint (say, in the case when $X$ is the 3-disk, the equivalence class of a suitable spherical cap) in $\mathcal{M}/\mathcal{D}$.
	
	The way this goal is achieved resembles, in general terms, the structure of the argument presented in \cite{Mar12} modulo the fact that some \emph{substantial} changes are needed for the plan to be successfully implemented. Let us first outline the argument in the special case of the three-dimensional disk $D^3$ and later we will describe some of the complications that arise when handling a compact 3-manifold that can be presented as indicated by Theorem \ref{thm:A}.
	
	Thus, one considers a triple $(M_0,g_0, f_0)$ as the initial data to run an equivariant version of the Ricci flow with surgery (cf. \cite{DL09}). Since the scalar curvature of $g_0$ is positive, we know that the flow will become extinct in finite time, and after finitely many surgeries. If no surgeries occur we know that the Ricci flow smoothly connects the initial data to some reflexive triple, say $(M_0, g_{t_{\ast}}, f_0)$, with $M_0$ connected by virtue of our initial assumption, that is covered by (equivariant) canonical neighborhoods (as defined by Perelman \cite{Per03b}, cf. Morgan-Tian \cite{MT07} and Kleiner-Lott \cite{KL08}). Thereby we show, by means of a delicate covering argument and dividing into cases depending on the way the set of fixed points $Fix(f_0)$ intersects the $\e$-necks that belong to the covering of the Riemannian manifold in question, and the type of the $(C,\e)$-caps we employ, that $(M_0,g_{t_{\ast}})$ can always be split via equivariant surgery into pieces each of which can be smoothly isotoped to a round sphere. Each piece resulting from the decomposition is isotoped either via the (smooth, non-surgical) Ricci flow, or via a conformal deformation. For the latter case, we need to discuss the equivariance property of Kuiper's developing map, see Section \ref{subs:equivKuip}. This part of the proof is presented in Section \ref{subs:basis}.
	
	Hence, in order to deal with the general case (when surgeries are performed to extend the flow), the proof proceeds via two inductive arguments. First, as the key preparatory step, in Section \ref{subs:conncomp} we deal with the case of reflexive triples that are actually obtained by (equivariant) Gromov-Lawson connected sums of pieces that are known to be (separately) isotopic to the standard round sphere. In order to do that, we first need to list, in our specific context, the possible types of Gromov-Lawson equivariant connected sum operations (which in turn connects to the classification of equivariant $\e$-necks). Thereby, once this inductive step is done, in Section \ref{subs:backw} we employ a certain \emph{reconstruction lemma}, see Lemma \ref{lemma.GL.and.RFsurgery.are.isotopic} for a precise statement, to design a second inductive argument, going \emph{backward in time} in the Ricci flow evolution, to derive from an isotopy at time (say) $t_{i}+\eta$ an isotopy at the pre-surgery time $t_{i}-\eta$, till we finally reduce to the time interval $[t_0, t_1)$ where the evolution is smooth and unaffected by surgeries. 
	To fix the ideas, it is helpful to consider the case when the Ricci flow with surgery occurs with \emph{exactly one surgery time} $0<t_1<\infty$. The reconstruction lemma ensures that the pre-surgical scenario (at time, say, $t_1-\delta$) can be reproduced (in the sense of an isotopy of smooth metrics) by taking Gromov-Lawson connected sums of the pieces arising when performing the surgery procedure. Since $t_2$ is the (finite) extinction time for the flow, we know that each such piece will come, by the inductive basis (as explained above), with an isotopy to a round metric, hence \emph{using the information provided by the whole space-time track of the Ricci flow}, we can merge these pieces together and then invoke the first inductive scheme to get a global isotopy.

	Now, in extending this result from the case of the disk to the case of general 3-manifolds with boundary as in Theorem \ref{thm:A}, the main obstacle is probably to find an appropriate replacement for the round metric, more precisely to associate to our initial manifold $X$ a (connected) subset of model triples of the form $(M,g,f)$, with $M=DX$, that represent the natural endpoints of the isotopy we wish to construct. The procedure to construct such triples is decribed in Section \ref{subs:ModelMet}. Roughly speaking, we start with a reference sphere and then perform a series of operations (that are associated to the various pieces in the decomposition of $X$), which are essentially equivariant connected sums of various types. An interesting feature, though, is that our model metrics also exhibit \emph{iterated necks}, namely (degree two) chains of equivariant necks, built at two different scales, one on top of the other. In fact, an important point of our discussion is then to prove that \emph{no higher-order chains of necks are needed} or, in other words, that higher-order chains of necks can always isotoped to simpler models. The isotopies we need to build are (due to the equivariance requirements) subject to additional requirements compared to \cite{Mar12}, thus these missing degrees of freedoms force the models to be more complicated. The reader might have a look at Figure \ref{fig:InflateSphere} and Figure \ref{fig:Simplify} to get a feeling for the sort of issues we face. The final result we obtain is as follows:

	\begin{prop}\label{pro:Conclusive}
		Let $(M,g,f)$ be a connected reflexive 3-manifold of positive scalar curvature. Then there exists a reflexive isotopy of classes connecting, through metrics of positive scalar curvature, its equivalence class to that of a model triple.
	\end{prop}
		
		\subsection{Other path-connectedness results }\label{subs:other}
		
		 We now mention other similar results the reader will find in this article. First of all, in Section \ref{sec:TopChar} we also prove topological characterization theorems for those moduli spaces of metrics that are obtained by weakening either of our geometric inequalities. More specifically, we obtain the same conclusion as in Theorem \ref{thm:A} for the spaces
			$\mathcal{M}_{R>0, H\geq 0}$ (Riemannian metrics on $X$ that have positive scalar curvature and weakly mean-convex boundary), and
			$\mathcal{M}_{R\geq 0, H> 0}$ (Riemannian metrics on $X$ that have non-negative scalar curvature and strictly mean-convex boundary).
			As can be easily inferred going back to the $n=2$ case, the analysis of the larger space
			$\mathcal{M}_{R\geq 0, H\geq 0}$ (Riemannian metrics on $X$ that have non-negative scalar curvature and weakly mean-convex boundary) is somewhat more delicate. In fact, with that goal in mind, we first prove a rigidity theorem singling out the borderline case (see Proposition \ref{pro:Rigidity}), which may be regarded as the analogue, for manifolds with boundary, of a well-known theorem due to J.-P. Bourguignon (cf. \cite{KW75}). Thereby, we derive Corollary \ref{cor:Equivalence}, that might be an equivalence statement of independent interest, and, as a result, Corollary \ref{cor:GenTopChar}.
			Lastly, a version of Theorem \ref{thm:B} for such moduli spaces is given in Section \ref{subs:pcmet}, see in particular the statement of Theorem \ref{thm:PathConnOther}: the analogy with the treatment of $\mathcal{M}_{R>0, H>0}$ works smoothly unless we have $X^3\simeq S^1\times S^1\times I$, in which case the larger moduli space of $\mathcal{M}_{R\geq 0, H\geq 0}$, consisting of \emph{one point}, is obviously path-connected. 
			The important case when the deal with minimal boundaries is harder, and we present it separately in Section \ref{sec:MinPaths}. Since Proposition \ref{pro:Conclusive} really provides a path through \emph{totally geodesic} (hence minimal) boundaries, the real challenge consists in providing a suitable version of Proposition \ref{thm:tgdef}.

		\subsection{Applications and perspectives}
		Besides their intrinsic significance, our main results (as well as some of the techniques we develop) have, by virtue of a well-known inversion argument via the Green function for the conformal Laplacian (cf. Schoen's solution of the Yamabe problem \cite{Sch84, LP87}) direct implications on the geometry of asymptotically flat spaces. In particular, one can derive path-connectedness results for various classes of time-symmetric (or, more generally, maximal) initial data sets for the Einstein field equations under the so-called \emph{dominant energy condition}. Roughly speaking, in that context the presence of a minimal boundary is related to the description of \emph{black-hole solutions}, namely to the trace of the event horizon on the spacelike slice in question (cf. \cite{HI01, Bra01, Bra02}). Such applications, which in substance are direct corollaries of the main theorem we obtain here in Section \ref{sec:MinPaths}, will be given elsewhere.
			In addition, our path-connectedness statements might, at least in principle, allow to obtain existence theorems (in particular, solvability results for partial differential equations) via the so-called method of continuity (see e. g. \cite{Aub98, SY94} for some remarkable, illustrative instances). This might be most likely when $X^3\simeq D^3$, in which case the endpoint metric (corresponding to a spherical cap) is particularly simple and well-understood.
		
			\subsection{Notation and conventions}
	For the convenience of the reader, we list here some conventions and notational principles we shall adopt throughout this article:
	\begin{itemize}
	\item {the manifolds we deal with are always assumed to be smooth (i.~e. $C^{\infty}$), as are their boundaries (if any);}
	\item {diffemorphisms between manifolds are assumed to be smooth and we shall write $X_1\simeq X_2$ to mean that $X_1$ and $X_2$ are diffeomorphic;}	
	\item {all metrics are Riemannian and smooth, and the space of metrics on a given (compact) manifold, and all subsets thereof, are endowed with the corresponding topology;}
	\item {given a Riemannian metric $g$ we let $d_g$ denote the corresponding distance, and by writing $B_r(p)$ we mean the metric ball of center $p$ and radius $r$;}
	\item {all submanifolds we deal with are embedded, and closed unless explicitly stated otherwise; in our applications they are always two-sided and orientable;}
	\item {we allow both ambient manifolds and submanifolds thereof to be disconnected; in particular this remark applies to reflexive manifolds, as per Definition \ref{def:Z2equiv};}
	\item {concerning the mean curvature, we adopt the following convention: the unit sphere in $\R^3$ has mean curvature equal to $2$; a domain is mean-convex if it has positive mean curvature, thus if an outward deformation with unit speed \emph{increases} area; we shall say that a manifold is \emph{weakly} mean-convex if it has non-negative mean curvature.}
	\end{itemize}	
		
		\

		\noindent \textit{Acknowledgements:} The authors would like to thank Fernando Cod\'a Marques for his interest in this work and for various helpful conversations.  We also wish to express our sincere gratitude to David Gabai, Pengzi Miao, Damin Wu and Boyu Zhang for several clarifications that have been very important at various stages of this project, and to the anonymous referees for carefully reading our manuscript and providing insightful remarks from which we greatly benefited. Furthermore, it is a pleasure for us to thank our advisor Richard Schoen for his constant support and invaluable guidance.
		
			A. C. is partly supported by the National Science Foundation (through grant DMS 1638352) and by the Giorgio and Elena Petronio Fellowship Fund. This project was developed at the Institute for Advanced Study during the special year \emph{Variational Methods in Geometry}: both authors would like to acknowledge the support of the IAS and the excellent working conditions.

		\section{Topological characterization theorems}\label{sec:TopChar}
		
		\subsection{Proof of Theorem \ref{thm:A}}
		
		 By virtue of Theorem 5.7 in \cite{GL80a}, Theorem \ref{thm:A} follows once
		we describe all the diffeomorphism types of manifolds with boundary $X$, such that the doubling $DX=M$ is a closed 3-manifold which supports metrics of positive scalar curvature. 
		We will prove the following topological statement: 
		
		\begin{thm}\label{theorem.topological.classification.of.3-mfld.with.boundary}
			Suppose $X$ is a 3-manifold with boundary, such that:
			\[M\simeq DX\simeq S^3/\Gamma_1 \#\cdots \#S^3/\Gamma_p \# \left(\#_{i=1}^q (S^2\times S^1)\right).\]
			Then there exist integers $a,b,c, d\ge 0$, such that:
			\[X\simeq P_{\g_1}\#\cdots\# P_{\g_a}\# S^3/\Gamma_1\#\cdots\# S^3/\Gamma_b\# \left(\#_{i=1}^c (S^2\times S^1)\right)\setminus (\sqcup_{i=1}^d B_i^3).\]
			Here: $P_{\g_i}$, $i\le a$, are genus $\g_i\geq 1$ handlebodies; $\Gamma_i$, $i\le b$, are finite subgroups of $SO(4)$ acting freely on $S^3$; $B_i^3$, $i\le d$, are disjoint $3$-balls in the interior.
		\end{thm}
		
		Of course, we note that the statement above is equivalent to the one given in the introduction since removing balls in the interior of $X$ corresponds to taking (interior) connected sums with genus zero handlebodies. On the other hand, while seemingly less homogeneous this formulation turns out to work somewhat better for the discussion we are about to present (due to the fact that any positive genus handlebody contributes to an $S^2\times S^1$ in the double, while the genus zero terms do not).
		Before embarking in the proof of Theorem \ref{theorem.topological.classification.of.3-mfld.with.boundary}, we present a collection of remarks illustrating the correspondence between $X$ and $DX$ in terms of the various constants $a,b,c,d$ and $p,q$ appearing in the statement.

		\begin{rmk}
			Following the usual convention in 3-manifold topology, the unit element for connected sum of prime 3-manifolds is $S^3$. As a result, when $a=0$, $b=0$ or $c=0$, we are just taking $S^3$ in place of $P_{\g_1}\#\cdots\# P_{\g_a}$, $S^3/\Gamma_1\#\cdots\# S^3/\Gamma_b$ or $\#_{i=1}^c (S^2\times S^1)$ respectively. Since we assume that $\partial X\ne \emptyset$, at least one of the inequalities $a\ge 1$ and $d\ge 1$ must hold.
		\end{rmk}

		\begin{ex}
			There are two diffeomorphism types of $X$ such that $DX\simeq S^2\times S^1$: $X_1\simeq D^2\times S^1$, or $X_2\simeq S^2\times [0,1]$. In the first case $X_1\simeq P_1$ (a genus one handlebody), which corresponds to $(a,b,c,d)=(1,0,0,0)$ and $\g_1=1$, while in the second case $X_2\simeq S^3\setminus (B_1^3 \sqcup B_2^3)$, which corresponds to $(a,b,c,d)=(0,0,0,2)$.
			This example shows, in particular, that $DX_1\simeq DX_2$ does not imply $X_1\simeq X_2$.
		\end{ex}
		
		\begin{rmk}\label{remark.effect.of.removing.balls.in.connected.sum}
			We now discuss the doubling of a general compact 3-manifold $X$ of the form
			\[X=P_{\g_1}\#\cdots\# P_{\g_a}\# S^3/\Gamma_1\#\cdots\# S^3/\Gamma_b\# \left(\#_{i=1}^c (S^2\times S^1)\right)\setminus (\sqcup_{i=1}^d B_i^3).\]
			The doubling $M=DX$ has a unique connected sum decomposition into prime components. In this decomposition, each spherical space form $S^3/\Gamma_i$, $i=1,\cdots,b$ appears twice, as it is the case for the components $\#_{i=1}^c (S^2\times S^1)$. For each $\g_i>0$, the component $P_{\g_i}$ becomes $\#_{i=1}^{\g_i} (S^2\times S^1)$. If $a>1$, the doubling of $P_{\g_1}\#\cdots \#P_{\g_a}$ is diffeomorphic to the connected sum of $(\#_{i=1}^{\g_1} S^2\times S^1)\#\cdots\# (\#_{i=1}^{\g_a} S^2\times S^1)$, with additional $a-1$ handles attached. Adding these extra handles is equivalent to a connected sum with $\#_{i=1}^{a-1} S^2\times S^1$.
			
			Let us now consider the contribution of removing interior 3-balls in performing the doubling. In general, let $Z$ be a 3-manifold. If $\partial Z= \emptyset$, then the doubling $D(Z\setminus B^3)$ is equal to the connected sum $Z\#Z$. If instead $\partial Z\ne \emptyset$, then the doubling $D(Z\setminus B^3)$ is diffeomorphic to $DZ$ with a handle attached, namely $D(Z\setminus B^3)=DZ\#(S^2\times S^1)$.
	
			Therefore, putting all these facts together we get
			\[M=DX\simeq (S^3/\Gamma_1\#S^3/\Gamma_1)\#\cdots \#(S^3/\Gamma_b\#S^3/\Gamma_b)\#\left(\#_{i=1}^e (S^2\times S^1)\right),\]
			where
			\[e=\begin{cases}\sum_{i=1}^a \g_i+a+2c+d-1 \quad &a> 0,\\ 2c+d-1 \quad &a=0.\end{cases}\]
			We thus conclude that, in either case, 
			\[e=\sum_{i=1}^a \g_i +a+2c+d-1,\] 
			with the convention that if $a=0$, then $\g_i=0$.
		\end{rmk}
	
\begin{proof}[Proof of Theorem \ref{theorem.topological.classification.of.3-mfld.with.boundary}]
	Assume $X$ is a 3-manifold with boundary such that
	\[M=DX\simeq S^3/\Gamma_1 \#\cdots \#S^3/\Gamma_p \# \left(\#_{i=1}^q (S^2\times S^1)\right).\]
	The idea behind the proof is to compress as many $S^1\subset \partial X$ as possible, and keep track of the topological change of $M=DX$ when performing a compression. Let us first recall some classical results in 3-manifolds topology needed for our purpose. For a standard reference the reader may consult the first two chapters of \cite{Jac80}. Assume $\Sigma^2\subset M^3$ is a closed surface, and $\Sigma\not\simeq S^2$. We call $\Sigma$ a compressible surface, if there is a disk $D$ in $M$, such that $\partial D\subset \Sigma$ does not bound a disk inside $\Sigma$. Otherwise, $\Sigma$ is called incompressible. By the Loop Theorem, if $\Sigma\xrightarrow{\iota} M$ is two-sided, and $\Sigma$ is not a sphere, then it is incompressible if and only if $\iota_*:\pi_1(\Sigma)\rightarrow \pi_1(M)$ is injective. Note that if $\Sigma$ is a component of $\partial X$, then $\Sigma\subset DX=M$ is a two-sided surface.
	
	We start by considering the case when every component of $\partial X$ is diffeomorphic to $S^2$. In this case, let $Y$
	be the closed 3-manifold obtained by gluing a copy of $B^3$ onto each component of $\partial X$. By the discussion presented in Remark \ref{remark.effect.of.removing.balls.in.connected.sum}, we have that
	\[M=DX\simeq Y\#Y\#\left(\#_{i=1}^{d-1} (S^2\times S^1)\right),\]
	where $d$ is the number of connected components of $\partial X$. Hence
	\[Y\#Y\#\left(\#_{i=1}^{d-1} (S^2\times S^1)\right)\simeq S^3/\Gamma_1 \#\cdots \#S^3/\Gamma_p \# \left(\#_{i=1}^q (S^2\times S^1)\right).\]
	Since the prime decomposition of any closed 3-manifold is unique (by Milnor \cite{Mil62}), we conclude that $Y$ can only be a connected sum of spherical space forms and $S^2\times S^1$'s. Therefore, we have, for suitable $b,c\ge 0$,
	\[X=Y\setminus \left(\sqcup_{i=1}^d B_i^3\right)\simeq S^3/\Gamma_1\#\cdots\# S^3/\Gamma_b\# \left(\#_{i=1}^c (S^2\times S^1)\right)\setminus (\sqcup_{i=1}^d B_i^3).\]

	Let us now proceed to the case when some connected component $\partial X$ is not an $S^2$. Suppose $\Sigma^2$ is a component of $\partial X$, and $\Sigma\not\simeq S^2$. We first verify that $\Sigma$ is compressible. Suppose otherwise, that $\Sigma\subset X$ is incompressible. Then the embedding $\iota_1: \Sigma\rightarrow X$ induces an injection of the fundamental groups. The embedding $\iota_2: X\rightarrow M$ and the projection $\pi:M\rightarrow X$ compose to the identity map. It follows that their induced homomorphisms
	\[\pi_1(X)\xrightarrow{(\iota_2)_*} \pi_1(M)\xrightarrow{\pi_*} \pi_1(X)\]
	compose to the identity map on $\pi_1(X)$. In particular, the map $(\iota_2)_*$ is injective. It follows that the composition
	\[\pi_1(\Sigma)\xrightarrow{(\iota_1)_*} \pi_1(X) \xrightarrow{(\iota_2)_*} \pi_1(M)\]
	is also injective. We therefore conclude that $\Sigma$ is an incompressible surface in $M$. However, since $M$ can be equipped with a metric of positive scalar curvature, by \cite{SY79a},
	$M$ cannot contain any two-sided incompressible surface that is not diffeomorphic to $S^2$, a contradiction.

	Therefore, let $\Sigma\subset \partial X$ be a compressible surface that is not diffeomorphic to $S^2$. Then, there is a simple closed curve $\Gamma$ on $\Sigma$ that bounds an embedded disk $D$ in $X$. Such a disk $D$ is a locally separating surface in a normal neighborhood $N_\e D\subset X$, and there exists a diffeomorphism $\varphi: D\times [-\e,\e]\rightarrow N_\e D$.

	To proceed, we cut the manifold $X$ along $D$, so that we will have two induced diffeomorphisms $\varphi_{+}: D\times [0,\e]\rightarrow N^{+}_\e D$ and $\varphi_{-}: D\times [0,\e]\rightarrow N^{-}_\e D$ with disjoint image. If we denote the resulting 3-manifold by $Z$, it is readily checked that $DX=DZ\# (S^2\times S^1)$. 
	Hence, from our topological assumption on $DX$, we have that
	\[DZ\simeq S^3/\Gamma_1 \#\cdots \#S^3/\Gamma_p \# \left(\#_{i=1}^{q-1} (S^2\times S^1)\right).\]
	We then proceed to consider the manifold $Z$ instead of $X$, and perform a compression on $Z$ whenever there is a connected component of $\partial Z$ is not diffeomorphic to $S^2$. However, this operation can only be performed finitely many times, as each time the double of the initial manifold loses an $S^2\times S^1$ factor in the prime decomposition. Therefore, after performing finitely many compressions on the boundary, the 3-manifold $Z_0$ only has $S^2$ boundary components. Hence
		\[Z_0\simeq S^3/\Gamma_1\#\cdots\# S^3/\Gamma_b\# \left(\#_{i=1}^c (S^2\times S^1)\right)\setminus (\sqcup_{i=1}^d B_i^3),\]
	for some $b,c\ge 0, d>0$.
	
	Notice that $X$ is obtained from $Z_0$ by taking finitely many \emph{boundary} connected sums of $Z_0$ with $D^2\times [0,1]$. However, for any $\gamma\ge 1$, the boundary connected sum 
	\[(S^3\setminus B^3)\underbrace{\#_\partial (D^2\times [0,1])\#_\partial \cdots\#_\partial (D^2\times [0,1])}_\text{$\gamma$ times}\]
	is diffeomorphic to a genus $\gamma$ handlebody. Thereby, reconstructing $X$ backwards from $Z_0$ we conclude the proof of Theorem \ref{theorem.topological.classification.of.3-mfld.with.boundary}.
\end{proof}
		
	\subsection{Rigidity versus deformability}
	
	 	In adapting the discussion above to derive corresponding results for the larger spaces $\mathcal{M}_{R>0, H\geq 0}, \mathcal{M}_{R\geq 0, H> 0}$ and $\mathcal{M}_{R\geq 0, H\geq 0}$ (the last one with an important \emph{caveat})
	we shall need the following trichotomy result.
		\begin{prop}\label{pro:Rigidity}
		Let $(X^3,g)$ be a connected, orientable, compact Riemannian manifold, such that the scalar curvature of $g$ is everywhere zero, and each boundary component is a minimal surface with respect to $g$. Then:\begin{itemize}
			\item[\emph{i)}] If $Ric_g$ is not identically zero, then there is a small (isotopic) smooth perturbation $g'$ of $g$, such that $g'$ has everywhere positive scalar curvature, and $\partial X$ is minimal with respect to $g'$;
			\item[\emph{ii)}] If $Ric_g$ is everywhere zero, and $\partial X$ is not totally geodesic with respect to $g$, then there is a small (isotopic) smooth perturbation $g'$ of $g$, such that the scalar curvature of $g'$ is everywhere zero and every connected component of $\partial X$ that is not totally geodesic becomes strictly mean-convex with respect to the outward unit normal vector field in $g'$ (while the other ones are kept totally geodesic);
			\item[\emph{iii)}] If $Ric_g$ is everywhere zero, and all the components of $\partial X$ are totally geodesic, then there exists a compact interval $I\subset \R$ such that $(X,g)$ is isometric to $S^1\times S^1\times I$ equipped with a flat metric.
		\end{itemize}
	\end{prop}

	\begin{proof}
		Suppose the assumptions of \emph{i)} are satisfied. Then $Ric_g\ne 0$ in some small geodesic ball $U=B_r(x)$ in the interior of $M$. Take a cutoff function $\chi$ supported in $U$, such that $0\le \chi \le 1$ and $\chi=1$ in $B_{r/2}(x)$. Define, for small $t\in (-\e,\e)$, the metrics
		\[g_t=g+t\chi Ric_g.\]
		Then the scalar curvature $R_{g_t}$ satisfies
		\begin{equation}\label{formula.derivative.for.R}
		\frac{d}{dt}\bigg\vert_{t=0}R_{g_t}=-\Delta_g (\chi R_{g})+ \Div_g \Div_g (\chi Ric_{g})-\chi |Ric_{g}|^2. 
		\end{equation}
		Keeping in mind equation \eqref{eq:changeR} for the conformal change of scalar curvature, specified for $n=3$, we consider the elliptic eigenvalue problems
		\[\lambda_1(t)=\inf\left\{\frac{\int_X  \left(|\nabla_{g_t} \phi|^2+\frac{1}{8}R_{g_t}\phi^2\right) dVol_{g_t}}{\int_X \phi^2 dVol_{g_t}}: \phi\in H^1(X, g_t)\setminus\left\{0\right\}\right\}.\]

		By standard elliptic theory, the value $\lambda_1(t)$ is well-defined, and there exists a smooth first eigenfunction $\phi_t$ solving the Neumann boundary value problem
		\[\begin{cases}
		\Delta_{g_t}\phi_t-\frac{1}{8}R_{g_t}\phi_t=-\lambda_1(t)\phi_t\quad &\text{in} \ X,\\
		\frac{\p \phi_t}{\p \nu}=0 \quad &\text{on} \ \p X.
		\end{cases}\]
		Without loss of generality, assume that $\phi_t>0$ and $\|\phi_t\|_{L^2(X,g_t)}=1$ (this is the standard normalization condition we have stipulated in Appendix \ref{sec:Neu}). Then the map $t\in (-\e,\e)\mapsto \lambda_1(t)$	is smooth, and so is the map $t\in (-\e,\e)\mapsto \phi_t$ as a map into (say) $C^2$, see Lemma \ref{lem:contdepdata}. Since $R_{g_0}=R_{g}= 0$ we notice that $\lambda_1(0)=0$ and $\phi_0$ is a constant (in fact, by our normalization, $\phi^2_0=Vol^{-1}_g(X)$). As a result, one can further observe that
		\begin{equation}
		\begin{split}
		\lambda_1'(0)&=\frac{d}{dt}\bigg\vert_{t=0}\int_X \left(|\nabla_{g_t} \phi_t|^2+\frac{1}{8}R_{g_t}\phi_t^2\right) dVol_{g_t} \\
		&=\frac{1}{Vol_g(X)}\int_X \left(\frac{1}{8}\frac{d}{dt}\bigg\vert_{t=0}R_{g_t}\right) dVol_{g} \\
		&=-\frac{1}{Vol_g(X)}\int_X \frac{1}{8} \chi |Ric_g|^2 dVol_{g}<0
		\end{split}
		\end{equation}
		where in the last step we have used \eqref{formula.derivative.for.R}, the divergence theorem on $X$, and the fact that $\chi$ is compactly supported away from $\partial X$.
		As a result, we conclude that $\lambda_1(t)>0$ for $t\in (-\e,0)$. Then the metric $g_t'=Vol^2_g(X) \hspace{0.5mm} \phi_t^{4}g_t$ is a small smooth perturbation of $g$ such that $R_{g_t'}>0$ everywhere in the interior of $X$, and $\partial X$ is minimal with respect to $g'_t$ (thanks to equation \eqref{eq:changeH} since $g_t=g$ near $\partial X$, and $\p_{\nu} \phi_t=0$). This completes the discussion of part \emph{i)}.
		
		\
		
		We then proceed to prove part \emph{ii)}. Assume that $Ric_g= 0$ in $X$, and that some boundary component $\Sigma\subset \partial X$ is not totally geodesic. 

		Since $\Sigma$ is not totally geodesic, it is an unstable minimal surface and thus we can deform it inward employing the first Jacobi eigenfunction, thereby obtaining subdomains
		$X_t$ that are (strictly) mean-convex, for $t$ small enough;
		also, there is a natural diffeomorphism $\varphi_t: X\rightarrow X_t$. We therefore observe that the metric on $X$ given by $g_t=\varphi_t^*(g)$ is scalar-flat, and that in $(X,g_t)$ the boundary component $\Sigma$ is strictly mean-convex. This concludes the proof of part \emph{ii)}.
		
		\
		
		Finally we study the case \emph{iii)}. Assume that the manifold $(X^3,g)$ is Ricci flat with totally geodesic boundary. Since $X$ has dimension three, the condition $Ric_{g}=0$ implies that all sectional curvatures of $g$ must vanish, hence $g$ is a flat metric. By the Gauss equation, we further conclude that each component of $\partial X$ is intrinsically flat as well, thus is isometric to a flat torus $S^1\times S^1$. Let then $DX$ denote, as usual, the doubling of $X$ across its boundary. It follows from the previous assertions that the doubled metric, also denoted by $g$, is smooth and flat. As a result, $(DX,g)$ is isometric to a flat three-torus, its  universal cover is Euclidean $\R^3$ and we have the (metric) identification
		$DX=\R^3/\Gamma$,
		where $\Gamma$ is the group generated by translations in three linearly independent vectors. Any connected component of the lift of each boundary component of $\partial X$ is a totally geodesic submanifold of the Euclidean space, hence necessarily a plane. Since different components of $\partial X$ are disjoint, their lifts in $\R^3$ are parallel planes. Since any two parallel planes in $\R^3$ bound a connected region, and $X$ is assumed to be connected, we conclude that there can only be two components in $\partial X$. This implies that the lift of $X$ is a slab in $\R^3$, and that $X$ is isometric to a flat product of the form $S^1\times S^1\times I$.
	\end{proof}
	
	\begin{cor}
		\label{cor:Equivalence}

	Let $X^3$ be a connected, orientable, compact manifold with boundary. Then the following three assertions are equivalent:
	\begin{enumerate}
	\item [i)]{$\mathcal{M}_{R>0, H>0}\neq\emptyset$;}
	\item [ii)]{$\mathcal{M}_{R>0, H\geq 0}\neq\emptyset$;}
	\item [iii)]{$\mathcal{M}_{R\geq 0, H>0}\neq\emptyset$.}
	\end{enumerate}	
Furthermore, each of these conditions is equivalent to
	\begin{enumerate}
	\item [iv)]{$\mathcal{M}_{R\geq 0, H\geq 0}\neq\emptyset$,}
\end{enumerate}	
unless $X^3\simeq S^1\times S^1\times I$ (in which case the space $\mathcal{M}_{R\geq 0, H\geq 0}$ only contains flat metrics, making the boundary totally geodesic).
	\end{cor}

\begin{proof} We employ some basic deformation results proved in Appendix \ref{sec:push-in}. More precisely,
the implication $\emph{ii)}\Rightarrow \emph{i)}$ follows from Lemma \ref{lem:pushMet}, while $\emph{iii)}\Rightarrow \emph{i)}$ follows from Lemma \ref{lem:EigenfDef}. Since the converse implications are trivial, this provides the equivalence of the conditions $\emph{i)}, \emph{ii)}, \emph{iii)}$. Similarly, if $X^3\not\simeq S^1\times S^1\times I$ the equivalence with condition \emph{iv)} follows via Proposition \ref{pro:Rigidity}.
Let us then justify the very last assertion, concerning the elements of $\mathcal{M}_{R\geq 0, H\geq 0}$ in the case of $S^1\times S^1\times I$. If this set contained a metric that \emph{either} is not (Ricci) flat \emph{or} does not make the boundary totally geodesic, then we could again invoke Proposition \ref{pro:Rigidity}, possibly in combination with Lemma \ref{lem:pushMet} and Lemma \ref{lem:EigenfDef} to conclude the existence of another metric, say $g$, belonging to the smaller space $\mathcal{M}_{R>0, H>0}$. At that stage, we could construct a stable minimal torus in the interior of $(X^3, g)$, which is not possible due to the fact that the positivity of the scalar curvature $R_g$ forces any stable (two-sided) minimal surface to a sphere. This contradiction proves our claim.
\end{proof}

Hence, we can derive a topological characterization theorem for any of the four spaces of metrics listed above.

\begin{cor}
	\label{cor:GenTopChar}
Let $X^3$ be a connected, orientable, compact manifold with boundary, such that any of the following equivalent conditions i), ii), iii) holds true, or iv) under the additional requirement that $X\not\simeq S^1\times S^1\times I$. 
Then there exist integers $A,B,C\ge 0$, such that $X$ is diffeomorphic to a connected sum of the form
\[ P_{\g_1}\#\cdots\# P_{\g_A}\# S^3/\Gamma_1\#\cdots\# S^3/\Gamma_B\# \left(\#_{i=1}^C (S^2\times S^1)\right).\]
Here $P_{\g_i}$, $i\le A$, are genus $\g_i\geq 0$ handlebodies, and $\Gamma_i$, $i\le B$, are finite subgroups of $SO(4)$ acting freely on $S^3$. Viceversa, any such manifold supports Riemannian metrics of positive scalar curvature and mean-convex boundary.
\end{cor}

	\section{Deforming the boundary to totally geodesic}\label{sec:elldef}

In this section we prove Proposition \ref{thm:tgdef}. The path of metrics we wish to construct is obtained via a suitable \emph{smoothing procedure}, based on \cite{Miao02}, of a doubling construction by Gromov-Lawson (Theorem 5.7 in \cite{GL80a}), which we shall now revisit pointing out a (previously unnoticed) fundamental feature.

\subsection{Revisiting a construction by Gromov-Lawson}\label{subs:GL} 

Given $(X^3,g)$ as in the statement of Proposition \ref{thm:tgdef}, we consider $\e_0>0$ such that there exists a tubular neighborhood of width $5\e_0$  for $\partial X\subset X$ and, in addition, all equidistant intermediate surfaces are themselves (strictly) mean-convex. In particular, we define (for $d_g:X\times X\to\R$ the Riemannian distance associated to the metric $g$)
\[
X':= \left\{x\in X \ : \ d_g(x,\partial X)\geq 2\e_0 \right\}, \ X^{''}:= \left\{x\in X \ : \ d_g(x,\partial X)\geq 4\e_0 \right\}, 
\]
and consider an homotopy of smooth maps (defined via the exponential map restricted to the normal bundle of $\partial X$) 
\[
\Psi: X\times [0,4\e_0]\to X, \text{with} \ \Psi_0 = id_X, 
\]
and such that for each $t\in [0,4\e_0]$ the maps $\Psi_{t}=\Psi(\cdot,t)$ is a diffeomorphism onto $X_{t}$, the set of points of $X$ at distance at least $t$ from the boundary.

The maps $\Psi_{t}$ determine induced actions on metrics. In particular, we can consider the continuous path of metrics (on our manifold $X$) given by $g_t = (\Psi_t)^{\ast}g $. What we are about to do is then to produce a path of metrics on $X$ connecting $g_{2\epsilon_0}$ to a suitable metric with positive scalar curvature and totally geodesic boundary, the whole path happening (except for the endpoint) in the class of metrics with positive scalar curvature and positive boundary mean curvature. 

\

We introduce the smooth manifold $M$ that is obtained as topological double of $X$ (so that we shall henceforth write $M=DX$) endowed with the following differentiable structure. Set $S=\partial X'$, consider on $X_{\ast}=S\times [-\pi\e/2,\pi\e/2]$ the metric $g_{\ast}$ induced on the half-tube $T_{\e}(\partial X')\cap ((X\setminus \text{Int}(X'))\times\R)$ where $T_{\e}(\partial X')$ is the set of points at small distance $\e>0$ from $\partial X'\times\left\{0\right\}$ in the Riemannian product $X\times\R$. Thus, take two copies of $X$ (which we shall denote by $X_+, X_{-}$), their subsets $X'_{+}, X'_{-}$ defined as above and identify, in the union $X'_+\sqcup X_{\ast} \sqcup X'_{-}$, the boundary components $\partial X'_{+}\cong S\times\left\{\pi \e/2\right\}$ and similarly $\partial X'_{-}\cong S\times\left\{-\pi \e/2\right\}$. On such a manifold, we consider the smooth structure $\left\{\text{Int}(X'_+),U_{+},U_{-},\text{Int}(X'_-)\right\}$ where $U_{+}=U^+_{-3\e_0}\cup X^{+}_{\ast}$ and $U^+_{-3\e_0}$ stands for the set of points in $X'_{+}$ at distance less than $3\e_0$ from the boundary $\partial X'_+$ and is identified with $S\times (-3\e_0,0]$ in the obvious fashion using the normal exponential map in metric $g$, while the set $X^{+}_{\ast}=S\times [-\pi\e/2,\pi\e/2)$ is identified with $S\times [0, \pi \e)$ via the normal exponential map in metric $g_{\ast}$. Similarly, we declare analogous identifications for $U_{-}$ in lieu of $U_{+}$.

With slight abuse of notation we shall still write $X'_+, X''_{+}$ (respectively $X'_{-}, X''_{-}$) to denote the images of these sets in $M$. We further let $u_{+}: U_{+} \to \R$ denote the \emph{signed} distance from $\partial X'_{+}$ with the agreement that such a distance is positive on $X_{\ast}$ (where it is determined by the metric $g_{\ast}$), negative on $X'_{+}$ (where it is determined by the restriction of $g$ to that set), and vanishes on their common boundary). We analogously define $u_{-}$, so that in particular for $x\in U_{+}\cap U_{-}$ if $u_{+}(x)=t$ then $u_{-}(x)=\pi\e-t$. For any $u\in [-2\e_0,\pi\e]$ we agree that $S^{+}_{u}$ (respectively $S^{-}_{u}$) is the level set $u^{-1}_{+}(u)$ (respectively $u^{-1}_{-}(u)$), and that $M^{+}_{u}$ (respectively $M^{-}_{u}$) is the connected component of $M$ with boundary $S^{+}_{u}$ (respectively $S^{-}_{u}$) and containing $X''_{+}$ (respectively $X''_{-}$).

We endow this smooth manifold with a (non-smooth) Riemannian metric $g^M=g^M(\e)$ by identifying it (as in Theorem 5.7 of \cite{GL80a}) with the set
\[
T_{\e}(X')=\left\{(x,h)\in X\times\R : d^{X\times \R}((x,h), X'\times\left\{0\right\})=\e \right\},
\]
for $\e\in (0,\e_0)$.

It is straightforward to check that the metric $g^M$ is $C^0$, and clearly smooth away from the interfaces $S^+_0$ and $S^-_0$. We emphasize that the smooth coordinate $u_{+}$ coincides from the signed distance from $S^{+}_{0}$ in $(M,g^{M})$. This follows, via a standard argument, from the fact that each meridian circle of the tube in question is a closed geodesic, hence the tube is completely foliated by such geodesics and the length-minimizing path on $(M,g^M)$ connecting couples of points on any given meridian circle is the shortest arc on such circle.

 \begin{rmk}\label{rem:2ndff}
 	We observe that the interface $S^+_0$ is given the same second fundamental form (with respect to the same unit normal to such interface) when regarded as the boundary of the two domains
 	$X'_{+}$ and $X_{\ast}$. (For instance, we can take the outward-pointing normal to $X'_{+}$, which is inward-pointing with respect to $X_{\ast}$).
 As a result, the metric $g^M$ has a well-defined first order expansion in Fermi coordinates for $S^{+}_0$. Analogously for $S^{-}_0$.
 \end{rmk}

 We shall adapt the method presented by Miao in Section 3 of \cite{Miao02} to deal with the singularity of the interfaces by means of a localized smoothing procedure, followed by a conformal deformation based on the tools presented in Appendix \ref{sec:Neu}. Such modifications will be performed so to preserve the two geometric properties stated in the following lemma.

\

\begin{lem}\label{lem:singpath}
	In the setting above,  there exists $\e_1\in (0,\e_0)$ such that for all $\e\in (0,\e_1]$ the metric $g^M(\e)$ on $M$ satisfies the following two properties:
	\begin{enumerate} 
		\item [\emph{i})]{the scalar curvature is positive on $M\setminus \left\{S^+_0, S^-_0\right\}$;}
		\item[\emph{ii})] {the mean curvature of $S^+_u$ is positive for all $u\in [-2\e_0,\e_0)$ and vanishes on the surface $S^+_{\e_0}$, which is totally geodesic.}
	\end{enumerate}	
\end{lem}

\begin{proof}

Concerning the first claim, since $(X,g)$ is assumed to have positive scalar curvature, we only need to study the $\e$-tube $T_{\e}(\partial X')$ consiting of the set of points $X\times \R$ having distance exactly equal to $\e$ from $\partial X'\times\left\{0\right\}$. We notice that this set is smooth for any $\e>0$ sufficiently small.
By a suitable version of the Gauss Lemma, we have that this set can be parametrized as follows:
\[
T_{\e}(\partial X')=\left\{ exp_{x'}(\e(\cos(\theta)\nu'(x')+\sin(\theta)\xi(x'))) \ : \ x'\in \partial X', \ \theta\in S^1\right\},
\]
where we have denoted by $\nu'(x')$ the outward-pointing unit normal vector to $\partial X'\times\left\{0\right\}\subset X\times \left\{0\right\}$ and by $\xi(x')$ the upward-pointing unit normal vector to $X\times \left\{0\right\}\subset X\times\R$.

The geometry of tubes has been extensively studied, see the monograph \cite{Gr04} and references therein. For our purposes we only need a very simple fact: in a convenient orthonormal frame $\left\{E_1, E_2, E_3\right\}$ diagonalizing the second fundamental form of $T_{\e}(\partial X')$, the principal curvatures $\lambda_1, \lambda_2, \lambda_3$ satisfy, in terms of the principal curvatures $\mu'_1, \mu'_2$ of $\partial X'$, the estimate
\begin{equation}\label{eq:eigenv}
\lambda_i=\begin{cases}
(\mu'_i+O(\e))\cos(\theta) & \text{if} \ i=1,2 \\
1/\e & \text{if} \ i=3.
\end{cases}
\end{equation}

The Gauss equation, applied to the tube in question, gives
\begin{equation}\label{eq:gauss1}
R^{T_{\e}(\partial X')}=R^{X\times \R}-2 Ric^{X\times \R}(E_0, E_0)+2(\lambda_1\lambda_2+\lambda_1\lambda_3+\lambda_2\lambda_3)
\end{equation}
where $R^{T_{\e}(\partial X')}$ (resp. $R^{X\times \R}$) denotes the scalar curvature of $T_{\e}(\partial X')$ (resp. $X\times \R$) and $Ric^{X\times \R}$ the Ricci tensor of the ambient manifold, which is evaluated along the outward-pointing unit normal vector $E_0$. Now, $Ric^{X\times \R}(E_0, E_0)=O(1)\cos(\theta)$, hence if we employ \eqref{eq:eigenv} and in \eqref{eq:gauss1} and rearrange the terms we finally get
\begin{equation}\label{eq:EstScal}
R^{T_{\e}(\partial X')}=R^{X\times\R}+\cos(\theta)\left[\frac{2}{\e}\left(H'+O(\e)\right)+O(1)\right]
\end{equation}
where $H'=\mu'_1+\mu'_2$ is the mean curvature of $\partial X'$.
As a result, $R^{T_{\e}(\partial X')}$ is positive for any sufficiently small $\e>0$ thanks to our assumptions that $g\in\mathcal{M}_{R>0, H>0}$.

\

Let us then justify our second claim, concerning the positivity of the mean curvature of $S^+_{u}$. Once again, it suffices to prove the claim for $u\in [0,\e_0)$ where this is a statement about the tube $T_{\e}(\partial X')\subset X\times\R$.

Thus, we need to study the mean curvature $H^{Z_\theta}$ of the level sets
$Z_{\theta}\subset T_{\e}(\partial X')$ corresponding to a fixed value of $\theta\in [0,\pi/2]$. 
 Clearly, $Z_0$ is totally geodesic, hence $H^{Z_0}=0$ and we will now show that $(H^{Z_\theta})'>0$
 for all $\theta\in (0,\pi/4]$ provided we take $\e>0$ small enough (independently of $\theta$), which will then imply $H^{Z_{\theta}}>0$ for this same range of values of the angle. 
Relying on the second variation formula for the area functional and the Schoen-Yau rearrangement trick, we have
	\begin{equation}\label{eq:FirstVar}
	\frac{1}{\e}(H^{Z_\theta})'=\frac{1}{2}(R^{T_{\e}(\partial X')}-R^{Z_{\theta}}+|A^{Z_\theta}|^2+|H^{Z_\theta}|^2)\geq \frac{1}{2}(R^{T_{\e}(\partial X')}-R^{Z_{\theta}}).
	\end{equation}
If we write down the Gauss equation for $Z_{\theta}\subset X\times \left\{\sin(\theta)\right\}$, 
we find at once that the quantity $R^{Z_{\theta}}$ is uniformly bounded, as we let $\e$ vary in $[0,\e_0]$. As a result, if we plug-in \eqref{eq:EstScal} in \eqref{eq:FirstVar} we derive the claim. To complete the proof for $[\pi/4,\pi/2]$, we observe that the mean curvature vector of the surface $\partial X'\times\left\{0\right\}\subset X\times \R$ is horizontal, outward-pointing, and has length bounded away from zero. The surface $Z_{\theta}$ is a normal graph over $\partial X'\times\left\{0\right\}$ in $X\times \R$, thus for $\e>0$ small enough the conclusion that $H^{Z_\theta}>0$ at each point follows by continuity.
\end{proof}	

\subsection{Smoothing \`a la Miao}\label{subs:Miao}

Before stating the main result of this section, concerning the existence of a family of regularizations of the metrics in question which preserve the two properties stated in Lemma \ref{lem:singpath}, let us introduce some additional notation. 

\

First of all, we convene here that the parameter $\e$ below is fixed, once and for all, based on the statement of Lemma \ref{lem:singpath} (for instance we can take $\e=\e_1$). Throughout this section, we let the constant $C$ to depend on such a choice of $\e$, and we are concerned about uniform estimates in the parameter $\delta$ associated with the convolution we perform. Furthermore, for the sake of notational convenience, we agree to write $u, S_{u}$ in lieu of $u_{+}, S^+_{u}$ respectively. It is tacitly understood that the regularization procedure (which is, as we shall see, local near any given interface) is implemented in the same way, symmetrically for both singular interfaces. 

\

We have that the metric $g^M$ takes (around the interface $S_0$) the form
\[
g=du\otimes du + a(u),
\]
where $a(u)$ denotes, for $u\in [-2\e_0,\pi\e/2]$, a continuous family of smooth Riemannian metrics on $S_0$. Recalling that the quantity
\[
\frac{1}{2}\left[\frac{\partial a}{\partial u}\right]_{u=0}
\]
equals the second fundamental form of $S_0$ in $(M,g^M)$, Remark \ref{rem:2ndff} implies that $a(u)$ is in fact a $C^1$ path of metrics.
As a result, the following two estimates hold for $s_1, s_2\in [-\delta,\delta]$:
\begin{equation}\label{eq:1st}
\|a(s_1)-a(s_2)\|_{C^0}\leq L |s_1-s_2|
\end{equation}
\begin{equation}\label{eq:Lag}
\left\|\frac{\partial a}{\partial u}(s_1)-\frac{\partial a}{\partial u}(s_2)\right\|_{C^0}\leq \omega(|s_1-s_2|)
\end{equation}
where $L:=\sup_{u\in[-\delta,\delta]}\|\partial a/\partial u\|_{C^0}$ and $\omega:[0,2\delta]\to \R$ is a continuous, non-decreasing function vanishing at $s=0$ (the associated modulus of continuity). It is understood that all norms of tensors are measured with respect to a given background metric on $S_0$, which we agree to be $a(0)$.

\

We consider functions $\phi, \psi\in C^{\infty}_c(\R)$ taking values in $[0,1]$ and supported in $(-1,1), (-1/2,1/2)$ respectively, such that $\int_{\R} \phi(s)\, ds=1$ and $\psi(u)=1/\Lambda $ on the interval $[-1/4,1/4]$, for $\Lambda$ a large positive constant that is fixed once and for all throughout the argument.
Given any $0<\delta<\e^2$ we shall further set $\psi_{\delta}(s)=\delta^2\psi(s/\delta)$, and 
\begin{equation}\label{def:1}
a_{\delta}(u)=\int_{\R} a(u-\psi_{\delta}(u)s)\phi(s)\,ds
\end{equation}
so that
\begin{equation}\label{def:2}
a_{\delta}(u)=\begin{cases}
\int_{\R} a(s)\left(\frac{1}{\psi_{\delta}(u)}\phi\left(\frac{u-s}{\psi_{\delta}(u)}\right)\right)\,ds & \ \ \text{if} \ \psi_{\delta}(u)\neq 0 \\
a(u) & \ \ \text{if} \ \psi_{\delta}(u)=0.
\end{cases}
\end{equation}
and, in particular, for $u\in [-\delta/4,\delta/4]$
\begin{equation}\label{def:conv}
a_{\delta}(u)= \int_{\R}a(s)\left(\frac{\Lambda}{\delta^2}\phi\left(\frac{\Lambda(u-s)}{\delta^2}\right)\right)\,ds 
\end{equation}
to be interpreted as a \emph{localized} fiberwise convolution. Distinguishing between $u\notin [-\delta^2,\delta^2]$, in which case we rely on equation \eqref{def:1}, and $u\in [-\delta/4,\delta/4]$, in which case we rely on equation \eqref{def:conv}, one checks at once that
for any $\delta\in (0,\e^2)$ we have that
 $a_{\delta}$ is also a smooth path of smooth metrics on $S_0$.
 
 We further observe that the following two uniform estimates will hold true:
 \begin{equation}\label{eq:est1}
 \|a_{\delta}(u)-a(u)\|_{C^0}\leq L \left(\frac{\delta^2}{\Lambda}\right)
 \end{equation}
 \begin{equation}\label{eq:est2}
 \left\|\frac{\partial a_{\delta}}{\partial u}(u)-\frac{\partial a}{\partial u}(u)\right\|_{C^0}\leq  \omega(\delta^2/\Lambda)+L\delta \|\psi'\|_{C^0} \int_{\R} |s|\phi(s)\,ds.
 \end{equation}
 The first one is justified by simply going back to the definition \eqref{def:1} and exploiting the Minkowski inequality: 
 \begin{multline*}
 \|a_{\delta}(u)-a(u)\|_{C^0}\leq \left\|\int_{\R} (a(u-\psi_{\delta}(u)s)-a(u))\,\phi(s)\,ds \right\|_{C^0} \\
 \leq \int_{\R} \|a(u-\psi_{\delta}(u)s)-a(u)\|_{C^0}\phi(s)\,ds\leq L \left(\frac{\delta^2}{\Lambda}\right)
 \end{multline*} 	
 where the last inequality relies on \eqref{eq:1st}.
 Concerning the second inequality, we use the Minkowski inequality again together with \eqref{eq:Lag}:
 \begin{align*}
 \left\|\frac{\partial a_{\delta}}{\partial u}-\frac{\partial a}{\partial u}\right\|_{C^0}  &=\left\|\int_{\R}\left[(1-\psi'_{\delta}(u)s)\frac{\partial a}{\partial u}(u-\psi_{\delta}(u)s)-\frac{\partial a}{\partial u}(u)\right]\,\phi(s)\,ds\right\|_{C^0} \\
 &\leq \int_{\R}\left\|\frac{\partial a}{\partial u}(u-\psi_{\delta}(u)s)- \frac{\partial a}{\partial u}(u)\right\|_{C^0}\phi(s)\,ds \\
 &+\int_{\R} \|\psi'_{\delta}(u)\|_{C^0}\left\|\frac{\partial a}{\partial u}(u-\psi_{\delta}(u)s)\right\|_{C^0}|s|\phi(s)\,ds  \\
 &\leq \omega(\delta^2/\Lambda)+ L\delta \|\psi'\|_{C^0} \int_{\R} |s|\phi(s)\,ds.
 \end{align*}	
 Furthermore, we can conveniently compute the second derivative in $u$ as follows:
 \begin{multline*}
 \frac{\partial^2 a_{\delta}}{\partial u^2}=\frac{\partial}{\partial u}
 \int_{\left\{u\geq 0\right\}} \frac{\partial a}{\partial u}(u-\psi_{\delta}(u)s)\left(1-\delta \psi'\left(\frac{u}{\delta}\right)s\right)\phi(s)\,ds \\
 +\frac{\partial}{\partial u}\int_{\left\{u\leq 0\right\}}  \frac{\partial a}{\partial u}(u-\psi_{\delta}(u)s)\left(1-\delta \psi'\left(\frac{u}{\delta}\right)s\right)\phi(s)\,ds 
 \end{multline*}
 which implies
 \begin{multline*}
 \frac{\partial^2 a_{\delta}}{\partial u^2}=\int_{\R} \frac{\partial^2 a}{\partial u^2}(u-\psi_{\delta}(u)s)\left(1-\delta \psi'\left(\frac{u}{\delta}\right)s\right)^2\,\phi(s)\,ds\\
 + \int_{\R} \frac{\partial a}{\partial u}(u-\psi_{\delta}(u)s)\left(- \psi''\left(\frac{u}{\delta}\right)s\right)\,\phi(s)\,ds.
 \end{multline*}
 By inspecting the various terms in this formula we can infer a uniform upper bound of the form
 \begin{multline}\label{eq:2ndDerAD}
 \sup_{u\in [-\delta/2,\delta/2]}\left\|\frac{\partial^2 a_{\delta}}{\partial u^2}\right\|_{C^0}
\\ \leq \|\psi''\|_{C^0}\sup_{u\in [-\delta,\delta]}\left\|\frac{\partial a}{\partial u}\right\|_{C^0}+(1+\delta \|\psi'\|_{C^0})^2\sup_{u\in [-\delta,0)\cup (0,\delta]}\left\|\frac{\partial^2 a}{\partial u^2}\right\|_{C^0}
 \end{multline}

 After collecting these estimates, we consider the smooth metrics given by
\[
g^M_{\delta}= du\otimes du + a_{\delta}(u)
\]
which obviously extend (provided we perform the same operation on both interfaces) to a smooth metric on $M$, which we shall also denote by $g^M_{\delta}$. 

\begin{prop}\label{pro:miao}
In the setting above, there exists $\delta_0\in (0,\e^2)$ such that for any $\delta\in (0,\delta_0)$ the smooth metric $g^M_{\delta}$ on $M$ satisfies the following two properties:
\begin{enumerate} 
	\item {the scalar curvature has a positive lower bound away from the two sets defined by $|u_+|\leq \delta/2$ and $|u_-|\leq \delta/2$, and is uniformly bounded from below in such regions;}
	\item {the mean curvature of $S^{+}_{u}$ is positive for all $u\in [-2\e_0,\e_0)$ and vanishes on the surface $S^{+}_{\e_0}$, which is totally geodesic.}
\end{enumerate}	
\end{prop}	

\begin{proof}

Claim (2) in the statement follows at once from the equation
\[
H^{S_u}_{g^M_\delta}=\frac{1}{2}(a_{\delta})^{-1}\frac{\partial a_{\delta}}{\partial u}
\]
 thanks to Lemma \ref{lem:singpath}, by virtue of the equations  \eqref{eq:est1} and \eqref{eq:est2}  which ensure that $H^{S_u}_{g^M_\delta}\to H^{S_u}_{g^M}$ uniformly as one lets $\delta\to 0$.

Let us justify the assertion concerning the scalar curvature. The second variation formula for the area functional, when combined with the Gauss equation for $S_u$ (with respect to the metric $g^M_{\delta}$), gives 
\[
R_{g^M_\delta}=2K^{S_u}_{g^M_\delta}-(|A^{S_u}_{g^M_\delta}|^2+|H^{S_u}_{g^M_\delta}|^2)-2\frac{\partial}{\partial u}H^{S_u}_{g^M_\delta}.
\]
Arguing as above for Claim (2), the first three terms on the right-hand side are uniformly bounded for $\delta\in (0,\e^2)$ (notice that the term involving the Gauss curvature is intrinsic to the given slice, so we just appeal to \eqref{eq:est1}) thus we only need to study the last one, and in the sole strip of points satisfying $|u|\leq \delta/2$. The chain-rule allows to rewrite such a term as
\begin{equation}\label{eq:chainrule}
\frac{\partial (a_{\delta})^{-1}}{\partial u}  \frac{\partial a_{\delta}}{\partial u} +(a_{\delta})^{-1}\frac{\partial^2 a_{\delta}}{\partial u^2}
\end{equation}
which implies Claim (1) thanks to the uniform bound \eqref{eq:2ndDerAD}  for $\frac{\partial^2 a_{\delta}}{\partial u^2}$.
Thereby the proof is complete.
\end{proof}

\subsection{Proof of Proposition \ref{thm:tgdef}}\label{subs:proof1}	

We will now employ conformal deformations (cf. Appendix \ref{sec:Neu} and \ref{sec:ConvexCond}) to obtain a continuous path of smooth metrics having the two properties stated in Proposition \ref{thm:tgdef}.
\begin{proof}
The isotopy $(g_{\mu})$ is obtained as concatenation of the three paths that are defined as follows:
\begin{enumerate}
    \item [\emph{i)}] for $t\in [0,2\epsilon_0+2\e]$, simply the metric $g$, pulled back through the diffeomorphisms $\Psi_t: X\rightarrow X_t$,  defined in Subsection \ref{subs:GL}; thereby we obtain a continuous path of metrics on $X$, all lying in $\mathcal{M}$, interpolating between $g_{0}=g$ and $g_1=(\Psi_{2\epsilon_0+2\e})^{\ast}g$;
    \item [\emph{ii)}] for $\mu\in [0,1]$, the metrics $(1-\mu+\mu(\psi_{2(\epsilon_0+\e)})^4 g_1$, where $\psi_{2\epsilon_0+2\epsilon}$ is the first Neumann eigenfunction of the conformal Laplace operator on $X_{2(\epsilon_0+\e)}$, pulled back through the diffeomorphism $\Psi_{2\epsilon_0+2\e}$; the fact that, for any $\mu\in [0,1]$ these metrics have positive scalar curvature and (strictly) mean-convex boundary follows from the convexity results in Appendix \ref{sec:ConvexCond};
    \item [\emph{iii)}] for $u\in [-2\epsilon, \pi\epsilon/2]$, the pull back of the metrics $(M_u^+,(\psi^{(\delta)}_u)^4 g^M_{\delta})$, where $\psi^{(\delta)}_u$ is the first Neumann eigenfunction of the conformal Laplace operator on $M_u^+$, through a suitable diffeomorphism from $X'_{+}$ to $M_u^+$, and then via $\Psi_{2\e_0}$; here we agree that $\delta$ is chosen small enough that the Proposition \ref{pro:miao} is applicable.
\end{enumerate}

More precisely, for any $u\in [-2\epsilon, \pi\epsilon/2]$ we define a diffeomorphism $X'_{+}\to M^{+}_u$ so that the level sets of the distance function from $\partial X'_{+}$ in $X'_{+}$ correspond bijectively to the level sets of the distance function from $S^{+}_{u}$ in $(M, g^{M}_{\delta})$ near the boundaries, that it equals the identity map well away from it and smoothly interpolates in the transition region in the middle. It follows from Lemma \ref{lem:contdepdata}, Lemma \ref{lem:poseigenv} and Proposition \ref{pro:miao}  that, this definition
gives a path of smooth metrics on $X$ that have positive scalar curvature and (strictly) mean-convex boundary, except for $u=\pi\epsilon/2$ (be $\overline{g}$ the corresponding metric on $X$) in which case the boundary is totally geodesic.

Now, we have that, by the construction presented in Subsection \ref{subs:GL} and Subsection \ref{subs:Miao}, the metric $g^M_{\delta}$ is a smooth metric on $M$, hence smoothness of the (Riemannian) double of $(X,\overline{g})$ follows by simply observing that the even extension to $M$ of the pull-back  of the function $\psi^{(\delta)}_{\pi\epsilon/2}$ to the upper half of $M$, simply denoted by $\psi$, is a $C^1$ function that solves, in a pointwise sense, a linear elliptic equation of the form
\[
\Delta_{g^M_{\delta}} \psi-\frac{1}{8}R_{g^M_{\delta}}\psi = -\lambda_1\psi
\]
hence it is actually a smooth function on $M$ by standard elliptic regularity.
It follows that we can attach two copies of $(X,\overline{g})$, by pointwise identification of the boundary points, to obtain a smooth Riemannian manifold of positive scalar curvature. 
\end{proof}	 

    \section{Ancillary results and tools}\label{sec:ancill}
    
    \subsection{Key definitions and basic facts}\label{sec:EquivMan}
    
    We start by collecting some of the definitions we will refer to throughout the article, and a few elementary facts.
    
    \begin{defi}\label{def:inv}
    Let $M$ be  a compact manifold without boundary. We shall say that a smooth map $f:M\to M$ is an involution if $f^2=id$, and $f$ is not the identity map itself. We further define:
    \begin{itemize}
    \item {$Fix(f)=\left\{x\in M \ : \ f(x)=x\right\}$, the set of fixed points of $f$;}	
    \item {$\mathcal{D}(M,f)=\left\{\varphi\in \diff(M) \ : \ f\circ \varphi=\varphi\circ f \right\}$, the set of $f$-equivariant diffeomorphisms of $M$.}	
    \end{itemize}	
If $g$ is a Riemannian metric on $M$ and $f^{\ast}g=g$ we shall say that $f$ is an isometric involution of $(M,g)$ and, equivalently, that the metric $g$ is $f$-compatible.
    	\end{defi}

\begin{rmk}\label{rmk:examples}
	The definition of involution leaves some considerable freedom for what concerns the structure of $Fix(f)$. For instance, let
	$(S^3, g_{round})$ be the unit round sphere, and consider the following examples of \emph{isometric} involutions:
	\begin{multline*}
	f(x_1, x_2, x_3, x_4)=(-x_1, -x_2, -x_3, -x_4) \ \Rightarrow \\
Fix(f)=\emptyset;
	\end{multline*} 
	\begin{multline*}
	f(x_1, x_2, x_3, x_4)=(x_1, -x_2, -x_3, -x_4) \ \Rightarrow \\ 
	Fix(f)=\left\{x\in\R^4 \ : \ |x|^2=1, \ x_2=x_3=x_4=0\right\};
	\end{multline*} 
	\begin{multline*}
	f(x_1, x_2, x_3, x_4)=(x_1, x_2, -x_3, -x_4) \ \Rightarrow \\ Fix(f)=\left\{x\in\R^4 \ : \ |x|^2=1, \ x_3=x_4=0\right\}
\end{multline*} 
	\begin{multline*}
	f(x_1, x_2, x_3, x_4)=(x_1, x_2, x_3, -x_4) \ \Rightarrow \\ Fix(f)=\left\{x\in\R^4 \ : \ |x|^2=1, \ x_4=0\right\}.
	\end{multline*}
\end{rmk}
    
    Our next definition singles out the behaviour corresponding to the last of the four maps described above, which precisely encodes the structure of the manifolds produced as output of Section \ref{sec:elldef}.

    \begin{defi}\label{def:Z2equiv}
    	We shall define a reflexive $n$-manifold to be a triple $(M,g,f)$ such that:
    	\begin{enumerate}
    		\item [i)]{$M$ is a compact, orientable, smooth manifold of dimension $n\geq 3$ without boundary;}
    		\item[ii)] {$g$ is a smooth Riemannian metric on $M$;}
    		\item[iii)] {$f\in C^{\infty}(M,M)$ is an isometric involution of $(M,g)$ and 
    			\begin{itemize}
    			\item{if $M$ is connected, then $Fix(f)$ is a smooth, embedded, closed separating hypersurface;}
    			\item{if $M$ is disconnected, then there exists $2d$ connected components 
    				\[
    				\left\{M^{\ast}_1, M^{\ast\ast}_1,\ldots, M^{\ast}_d, M^{\ast\ast}_d \right\}
    				\] such that $\left\{M^{\ast}_i, M^{\ast\ast}_i\right\}$ are isometric under $f$ for any $i=1,\ldots, d$, and any other component $M^{\bullet}$ in $M$ is such that the triple $(M^{\bullet}, g^{\bullet}, f^{\bullet})$ is itself a connected reflexive triple (with $g^{\bullet}, f^{\bullet}$ denoting the restriction of the metric $g$ and the involution $f$ to $M^{\bullet}$, respectively).}	
    		\end{itemize}
    	}
    	\end{enumerate}	
    \end{defi}

\begin{rmk}
	\label{rmk:separating}
If $M$ is a connected manifold (or a connected component of a disconnected one) we say that a (possibly disconnected) smooth, closed, embedded hypersurface $\Sigma$ is separating if $M\setminus{\Sigma}$ consists of exactly two connected components. For instance, we agree that if $X\simeq S^2\times I$ and $M=DX$ then the disconnected surface $\Sigma= S^2\times\partial I$ (that consists of two spheres) is separating in $M$.
\end{rmk}

Now, we add a couple of elementary lemmata providing some basic information about \emph{connected} reflexive manifolds.

\begin{lem}\label{lem:mix1}
	Let $(X^n,g)$ be a connected manifold with boundary, equipped with an isometry $f$, such that $\partial X\subset Fix(f)$. Then $f=id$ on $X$.
\end{lem}
\begin{proof}
	Induction on $n$. When $n=1$ we can assume, without loss of generality, that $X^1$ is $([-1,1],du^2)$. Then $f:[-1,1]\rightarrow [-1,1]$ is a diffeomorphism, and $f(\pm 1)=\pm 1$. Therefore $f$ is monotone increasing. Since $f$ is an isometry, $(f'(u))^2=1$. Therefore $f'(u)=1$ everywhere, and hence $f=id$.
	
	Suppose the statement is true for dimension $n-1$. Take $\e<r_0$, where $r_0$ is the injectivity radius of $X$. Take $p\in \partial X$. Denote the set $\Gamma=S_{\e}(p)$ the geodesic sphere of radius $\e$ centered at $p$, equipped with the induced metric. Since $f$ is an isometry, it maps $\Gamma$ to itself. Moreover, $\Gamma$ is a Riemannian manifold with boundary of dimension $n-1$. Restricted to $\Gamma$, $f$ is an isometry such that $\partial \Gamma\subset Fix(f)$. By induction, $f$ is the identity map when restricted to $\Gamma$. In fact, this argument shows that $f=id$ in a tubular neighborhood of $\Gamma$, hence in a neighborhood of $p$, and the full conclusion comes by suitably repeating this construction.
\end{proof}

\begin{lem}\label{lem:mix2}
	Suppose we have $(M,g,f)$ a connected reflexive manifold.
	Then for any point $p\notin Fix(f)$ and for any simple curve $\s$ connecting $p$ to $f(p)$ the intersection $\s \cap Fix(f)$ is not empty.
\end{lem}

\begin{proof}
	Let $\Omega_1,\Omega_2$ be the two (open) connected components of $M\setminus Fix(f)$. Since $f$ is a diffeomorphism, by connectedness we shall have \emph{either} $f(\Omega_1)\subset \Omega_1$ \emph{or} $f(\Omega_1)\subset \Omega_2$. However, we cannot have $f(\Omega_1)\subset \Omega_1$ (equivalently: $f(\Omega_1)=\Omega_1$), for otherwise $f$ would be an isometry on the Riemannian manifold with boundary $M_1=\overline{\Omega}_1$, such that $\partial M_1\subset Fix(f)$. By Lemma \ref{lem:mix1}, this would imply that $f=id$ on $M_1$, hence on $M$. Contradiction. Therefore we must have that $f(\Omega_1)=\Omega_2$, and hence $f(\Omega_2)=\Omega_1$. 
	
	If $\s$ is disjoint from $Fix(f)$, then it is entirely contained in one of $\Omega_1$ and $\Omega_2$, but this is impossible since $p$ and $f(p)$ are not both in $\Omega_1$ or $\Omega_2$.
\end{proof}

\begin{lem}\label{lem:LocalDesc}
	Let $(M,g,f)$ be a connected reflexive manifold, let $U$ be a tubular neighborhood of $Fix(f)$ and let $(x,t)$ for $x\in \Sigma, t\in (-\e,\e)$ be the associated coordinates on $U$. Then $f(x,t)=(x, -t)$ for any $x\in \Sigma, \ t\in (-\e,\e)$.
\end{lem}

\begin{proof}
	With some (convenient) abuse of language we shall notationally identify points in $U$ with their coordinates in the normal bundle of $\Sigma:=Fix(f)$ in $M$, associated with the tubular neighborhood in question.	Said $\nu$ a unit normal to $\Sigma$ at $(x,0)$, we notice that the linear map $df: T_{(x,0)}M\to T_{(x,0)}M$ acts as a linear isometry and since it restricts to the identity on the subspace $T_{(x,0)}\Sigma$ it will be either $df(\nu)=\nu$ or $df(\nu)=-\nu$. However, the former alternative is easily ruled out thanks to Lemma \ref{lem:mix2}, thus the latter will hold. Now, let $\gamma(s)=exp_{(x,0)}(s\nu)$ and $\gamma^{f}(s)=f(exp_{(x,0)}(-s\nu))$. It follows from our information on $df$ that the two curves in question will emanate at $s=0$ from the same point and with the same velocity. Thus, by the uniqueness part of the Cauchy-Lipschitz theorem we must conclude that they coincide, which in particular implies that for any point $(x,t)$ in the tubular neighborhood we will have $f(x,t)=(x,-t)$, as desired.
\end{proof}	

\begin{cor}\label{cor:UniqueF}
	Let $M$ be connected and assume that both $(M,g,f)$ and $(M,g,\tilde{f})$ are reflexive manifolds. If $Fix(f)=Fix(\tilde{f})$ then $f=\tilde{f}$.
\end{cor}

\begin{proof}
Let $\Omega_1, \Omega_2$ be the connected components of $M\setminus Fix(f)$, and let $M_1, M_2$ be their closures, respectively. It follows from Lemma \ref{lem:mix2} that both $f$ and $\tilde{f}$ interchange $M_1, M_2$ namely $f: M_1\to M_2$ is a diffeomorphism (of manifolds with boundary), and the same for $\tilde{f}$. Thus we can consider the restriction of the composition $\tilde{f}\circ f$ to $M_1$: Lemma \ref{lem:mix1} ensures that such composition equals the identity, hence $f=\tilde{f}$.	
\end{proof}	

Thanks to the local description provided in Lemma \ref{lem:LocalDesc} a simple interpolation argument allows to prove the following statement:

\begin{cor}
	\label{cor:RedStandard}
Let $(M,g,f)$ and $(M',g',f')$ be connected reflexive triples. Assume that there exists a diffeomorphism $\varphi: M_1\to M'_1$ where $M_1$ (resp. $M'_1$) is the closure of one of the connected components of $M\setminus Fix(f)$ (resp. $M'\setminus Fix(f')$). Then there exists a diffeomorphism $\overline{\varphi}: M\to M'$ such that $\overline{\varphi}\circ f= f'\circ\overline{\varphi}$. As a result  $(M, (\overline{\varphi})^{\ast}g',f)$ is also a connected reflexive triple.
\end{cor}

For instance, this corollary ensures that whenever we have a reflexive triple of the form $(M^3\simeq S^3, g, f)$ we can always assume that the involution $f$ equals the standard reflection on $S^3\subset \R^4$ provided we replace $g$ with its pull-back by means of a suitable diffeomorphism. Since Theorem \ref{thm:B} is a statement about the moduli space of metrics this replacement can always be made without losing any generality.

\

We conclude this section by introducing a suitable notion of isotopy in the category of objects we will be dealing with.

\begin{defi}\label{def:Z2isot}
	\begin{itemize}
		\item {A reflexive isotopy is a (continuous) path $\mu\in[0,1]\mapsto (M,g_\mu,f)$ such that for any $\mu\in [0,1]$ the triple $(M,g_{\mu},f)$ is a reflexive $n$-manifold.}
		\item {Given a reflexive $n$-manifold $(M,g,f)$ let $\mathcal{D}(M,f)$ (resp. $\mathcal{D}_+(M,f)$) denote the class of diffeomorphisms (resp. orientation-preserving) diffeomorphisms $\varphi$ of $M$ that are $f$-invariant in the sense that $\varphi\circ f=f\circ\varphi$.
			If $g'$ is a Riemannian metric such that $(M,g',f)$ is a reflexive $n$-manifold we let
			$[g']$ denote the corresponding equivalence class modulo $\mathcal{D}(M,f)$. A reflexive isotopy of classes is a (continuous) path $\mu\in[0,1]\mapsto (M,[g_\mu],f)$ such that for any $\mu\in [0,1]$ the triple $(M,g_{\mu},f)$ is a reflexive $n$-manifold.}
	\end{itemize}
\end{defi}

It is straightforward to check that the last definition, concerning isotopies of classes, is well-posed i.~e. it does not depend on the choice of the representatives in the moduli space.

    \subsection{Equivariant developing maps}\label{subs:equivKuip}
    
    \begin{prop}\label{pro:ConfDevMap}
    	Let $(M\simeq S^3, g, f)$ be a locally conformally flat, reflexive manifold for $f=\rho$ the reflection of $S^3\subset \R^4$ defined by $\rho(x', t)=(x',-t)$. Then there exists a conformal diffeomorphism $\varphi: (M,g)\to (M, g_{round})$ (the round unit sphere) that satisfies $\varphi\circ f=f\circ \varphi$ (that is to say $\varphi\in\mathcal{D}(M,f)$).
    	Furthermore, there exists a reflexive isotopy $(g_\mu)$ connecting $(M, g, f)$ to $(M, \varphi^{*}g_{round}, f)$. If $g\in\mathcal{R}$ then $g_\mu\in\mathcal{R}$ for any $\mu\in [0,1]$.
    \end{prop}

    \begin{proof}
    	It follows from Kuiper's work (cf. \cite{Kui49}, see also \cite{Kui50} for later extensions) that one can define a conformal diffeomorphism $\varphi: (M,g)\to (M, g_{round})$, the so-called \emph{developing map}. This map is obtained by patching together (thanks to Liouville's theorem classifying conformal automorphisms of spheres) the local definitions that we are granted since we are assuming that $(M,g)$ is locally conformally flat.

    	Now, since $f^{\ast}g=g$, it follows that the two (a priori distinct) developing maps $\varphi$ and $\varphi\circ f$ 
    	must coincide on the sphere $Fix(f)$. On the other hand we also know, by the uniqueness part of Kuiper's theorem, that they must differ by a M\"obius transformation, say $\vartheta$, of $S^3\equiv \overline{\mathbb{R}^3}$, i.~e. one shall have $\varphi\circ f=\vartheta\circ\varphi$.
    	
    	An elementary argument shows that these two things together are only possible if $\varphi(Fix(f))$ is a totally umbilic sphere in round $S^3$ and thus, possibly by composing $\varphi$ with a M\"obius transformation and renaming it accordingly, if $\varphi(Fix(f))$ is an equatorial sphere and also $\vartheta=\rho(=f)$. Therefore, $\vartheta\circ\varphi= \varphi\circ f$ is the same as  $f\circ \varphi=\varphi\circ f$, and this equation precisely encodes the desired equivariance property of $\varphi$. 
    	\
    	
    	Once we know that  $\varphi\in\mathcal{D}(M,f)$, the reflexive isotopy is simply obtained by considering a linear interpolation of the conformal factors: we set
    	\[
    	g_{\mu}=\phi^4_{\mu}g, \ \text{for} \ \phi_{\mu}=(1-\mu)+\mu \phi
    	\]	
    	where, $\phi$ is the conformal factor associated to the pulled-back round metric $\varphi^{\ast}g_{round}$.
    	The fact that the corresponding isotopy of metrics actually occurs inside $\mathcal{R}$ if $g$ belongs to such a set follows directly, as a very special case, from the convexity results provided in Appendix \ref{sec:ConvexCond} (cf. Remark \ref{rem:ConvClosedCase}).
    \end{proof}	

\subsection{Equivariant $\e$-necks}\label{subs:Necks}

\begin{defi}
	\label{def:EquivNeck}
	Let $(M^3,g,f)$ be a reflexive manifold (in the sense of Definition \ref{def:Z2equiv}). Given $0<\e<1$, we say that a reflexive $\e$-neck is the datum of:
	\begin{enumerate}
	\item {a point $z\in Fix(f)$;}
	 \item {an $f$-invariant set $N\subset M$;}
	 \item {a diffeomorphism $\psi: N\to S^2\times (-1/\e,1/\e)$ such that:
	 	\begin{itemize}
	 		\item {the metric $R_g(z)(\psi^{-1})^{\ast}g$ is $\e$-close in the $C^{[1/\e]}-$topology to the cylindrical metric $g_{cyl}$ where the first factor is scaled so to have scalar curvature one;}
	 		\item {the pulled-back action of $f$ is an isometry of $(S^2\times (-1/\e,1/\e), g_{cyl})$.}
	 	\end{itemize}	}
	 \end{enumerate}
 We shall say that $z$ is the center of the neck, and that $h:=R_g(z)^{-1/2}$ is its scale.
\end{defi}

The above definition is essentially taken from the work of Dinkelbach-Leeb (cf. Section 3.3 in \cite{DL09}), that is devoted to the study of equivariant Ricci flows to the scope of classifying smooth actions on certain classes of geometric manifolds. 

  \begin{rmk}\label{rmk:local modelsNeck}
	If $(M^3,g,f)$ is a reflexive triple, $N\subset M$ is a reflexive neck (with slight abuse of language), and $\psi: N\to S^2\times (-1/\e,1/\e)$ is the corresponding neck chart, then $\Sigma_N=\psi(Fix(f))$ is a totally geodesic surface in $S^2\times (-1/\e,1/\e)$. Notice that such a surface is not empty because it will contain at least the center of the neck, which is (by definition) a point of $Fix(f)$.
	Hence, it follows (by explicit classification of totally geodesic surfaces in a round cylinder) that only two local models are possible: 
	\begin{enumerate}
	\item[] {\emph{Type T:} the surface $\Sigma_N$ is a product of the form $S^1\times (-1/\e,1/\e)\subset S^2\times (-1/\e,1/\e) $, where it is understood that $S^1$ sits inside round $S^2$ as an equatorial circle;}
	\item[]	{\emph{Type C:} the surface $\Sigma_N$ is the central sphere of the neck, i.~e. we have that $\Sigma_N=S^2\times\left\{0\right\}\subset S^2\times (-1/\e,1/\e)$.}
	\end{enumerate}

	Correspondingly, the classification of isometric actions of $\mathbb{Z}/2\mathbb{Z}$ on a round cylinder allows to explicitly determine the pulled-back action of $f$ on the cylinder.  In the case of type $T$ the action is given by the map $\sigma:S^2\times (-1/\e,1/\e)\to S^2\times (-1/\e,1/\e)$ acting at the level of the single (spherical) fibers as the standard reflection of $S^2\subset \R^3$ with respect to the plane defined by vanishing of the last coordinate, while in the case of type $C$ the action is a symmetry with respect to the central sphere of the neck (which we shall denote by $\kappa$).
	Notice that, in case of a neck of type $T$, we can always pre-compose the map $\psi$ with an isometry of the cylinder so that $\psi(Fix(f))$ equals (in the corresponding chart) the set $Fix(\sigma)$.
	\emph{From now onwards we shall convene that all reflexive necks we deal with satisfy such an additional condition.} 
	
	\begin{figure}[htbp]
		\centering
		\includegraphics[width=\textwidth]{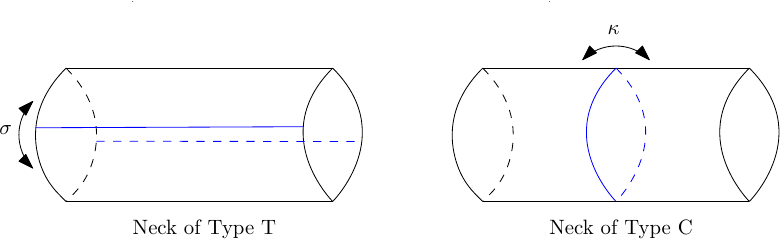}
		\caption{Equivariant necks.}
		\label{fig:NecksType}
	\end{figure}
	\end{rmk}

We complete this section recasting, in our equivariant setting, the definition of \emph{structured chain of $\e$-necks} and providing a simple lemma that ensures that the elements of any such chain can be combined into a unique neck structure, with the same equivariance property.

\begin{defi}
	\label{def:EquivChain}
	Given a reflexive manifold $(M^3,g,f)$, we say that an ordered $k$-tuple of reflexive $\e$-necks $(N_1, N_2,\ldots, N_k)$ in $M$ is a 	reflexive chain of $\e$-necks if the following three conditions are satisfied:
	\begin{enumerate}
		\item {said $z_i\in Fix(f)$ the center of $N_i$ one has $s_{N_i}(z_{i+1})=0.9/\e$ for any $i=1,2,\ldots, k-1$;}
		\item {$g(\partial/\partial s_{N_i}, \partial/\partial s_{N_{i+1}})>0$ for any $i=1,2,\ldots, k-1$;}
		\item {$s_{N_1}^{-1}((-1/\e, -0.5/\e))\cap N_{i}=\emptyset$ for any $i=2,\ldots, k$.}
	\end{enumerate}	
\end{defi}

Notice that in this definition we have employed the standard notation that $s_{N_i}: N_i\to (-1/\e, 1/\e)$ is the projection onto the second factor associated to $\e$-neck structure $(N_i, \psi_i)$. This notation will be systematically adopted in the sequel of this article.

We remark that the above definition is somewhat more restrictive here than it is in the non-equivariant setting. In particular, we note the following fact:

\begin{lem}
	\label{lem:TypeChain}
	Let $(M^3,g,f)$ be a reflexive manifold and let $(N_1, N_2,\ldots, N_k)$ be a reflexive structured chain of $\e$-necks therein. Then each $\e$-neck is of type $T$ (cf. Remark \ref{rmk:local modelsNeck}) and on any overlapping region $N_i\cap N_{i+1}$ (for $i=1,\ldots, k-1$) we have that the transition map $\varphi=\psi_{i+1}\circ\psi^{-1}_i$ satisfies the equation
	$\varphi\circ\sigma=\sigma\circ\varphi$.
	In particular, on the overlapping region $Fix(f)$ is mapped by $\psi_i, \psi_{i+1}$ to the same set modulo a translation of the second coordinate.
\end{lem}

\begin{proof}
	To prove the first statement, there are in principle two different scenarios to be ruled out: first an overlapping of two (consecutive) $\e$-necks of type $C$, second an overlapping of an $\e$-neck of type $C$ with an adjacent neck of type $T$. However, they can both be handled with the same argument: it is enough to notice that, said (as above) $z_{i+1}\in Fix(f)$ the center of $N_{i+1}$ we should have $s_{N_i}(z_{i+1})=0$ (because $Fix(f)$ is mapped to the central spherical leaf of $N_i$ by definition of type $C$), but on the other hand $s_{N_i}(z_{i+1})=0.9/\e$ by condition (1) above, a contradiction.
	
	Concerning the second statement, we know (since we are dealing with two equivariant necks of type $T$, and we are adopting the convention given in Remark \ref{rmk:local modelsNeck}) that $f$-related points in $M$ correspond to $\sigma$-related points in the $\e$-neck chart (this assertion being true for each of the two charts in question). The conclusion follows at once.
\end{proof}

\begin{lem}\label{lem:Combo1}  (cf. Lemma 4.1 in \cite{Mar12})
	There exists $\e_1\in (0,1)$ such that the following assertion is true.	
	Let $(M^3,g,f)$ be a reflexive manifold and let
	$(N_1,\psi_1), (N_2,\psi_2)$ be two reflexive $\e$-necks structures of type $T$, having scales $h_1, h_2$ respectively, such that
	$\Lambda:=s^{-1}_{N_1}((-0.95/\e, 0.95/\e))\cap s^{-1}_{N_2}((-0.95/\e, 0.95/\e))\neq\emptyset$ and $g(\partial/\partial s_{N_1},\partial/\partial s_{N_2})>0$.
	Suppose further that any embedded two-dimensional sphere separates $M$. If $0<\e\leq \e_1$, for any $z\in\Lambda$
	then there exists a set $N\subset N_1\cup N_2 $ and a diffeomorphism $\psi: N\to S^2\times (-1/\e, \beta)$ for $\beta=1/\e+s_{N_1}(z)-s_{N_2}(z)$ such that:
	\begin{enumerate}
		\item {the set $N$ is $f$-invariant and the pulled-back action is an isometry of $S^2\times (-1/\e, \beta)$ with the cylindrical metric $g_{cyl}$, where the round $S^2$-factor is normalized so to have scalar curvature one;}
		\item {$\psi^{-1}(\theta, t)=\psi_1^{-1}(\theta, t)$ for all $(\theta, t)\in S^2\times (-1/\e, s_{N_1}(z)-0.0025/\e)$;}
		\item {$\psi^{-1}(\theta, t)=\psi_2^{-1}(\theta, t+1/\e-\beta)$ for all $(\theta, t)\in S^2\times (\beta+s_{N_2}(z)-0.975/\e, \beta)$;}
		\item {there exists a continuous path of metrics $(g_{\mu})$ on $S^2\times (-1/\e, \beta)$, all of positive scalar curvature, with $g_0=(\psi^{-1})^{\ast}(g)$, $g_1$ rotationally symmetric, and restricting to the lineary isotopy $g_{\mu}=(1-\mu)(\psi_1^{-1})^{\ast}(g)+\mu h^2_1g_{cyl}$ on $S^2\times (-1/\e, s_{N_1}(z)-0.0025/\e)$ and $g_{\mu}=(1-\mu)(\psi_2^{-1})^{\ast}(g)+\mu h^2_2g_{cyl}$ on $S^2\times (\beta+s_{N_2}(z)-0.975/\e, \beta)$.}
	\end{enumerate}	
\end{lem}	

\begin{proof}
	One can follow \emph{verbatim} the proof of Lemma 4.1 in \cite{Mar12}, the only issue being verifying that the definition of the map $\psi:N \to S^2\times (-1/\e,\beta)$, provided at page 829 therein, satisfies the equivariance property given in item (1) above. This fact is justified by observing that both the transition map $\varphi=\psi_2\circ \psi^{-1}_1$ and its approximation $\tilde{\varphi}$ are commuting with $\sigma$.
	
	The first assertion is indeed a consequence of Lemma \ref{lem:TypeChain} above, while the second one relies on the convention that we have stipulated, in Remark \ref{rmk:local modelsNeck} for type $T$ equivariant necks, ensuring that the isometry $A$ therein (which appears when changing the neck chart) must be the identity. 
\end{proof}	

This observation allows, in particular, to combine the elements of a structured chain of necks:

\begin{lem}\label{lem:Combo2}
	Let $(M^3,g,f)$ be a reflexive manifold and let
	$(N_1, N_2,\ldots, N_{a+1})$ be a reflexive structured chain of $\e$-necks for some positive integer $a$. 
		If $0<\e\leq \e_1$ (with $\e_1$ as prescribed in the previous statement) then there exists a diffeomorphism $\psi: \cup_{i=1}^{a+1} N_i\to S^2\times (-1/\e, (1+a(0.9))/\e)$ such that:
	\begin{enumerate}
		\item {$\psi^{-1}(\theta, t)=\psi_1^{-1}(\theta, t)$ for all $(\theta, t)\in S^2\times (-1/\e,0.25/\e)$;}
		\item {$\psi^{-1}(\theta, t)=\psi_{a+1}^{-1}(\theta, t)$ for all $(\theta, t)\in S^2\times ((-0.25+a(0.9))/\e,(1+a(0.9))/\e)$;}
		\item {there exists a continuous path of metrics $(g_{\mu})$ of positive scalar curvature on $S^2\times (-1/\e, (1+a(0.9))/\e)$ such that $g_0=(\psi^{-1})^{\ast}g$, $g_1$ is rotationally symmetric, and restricting to the lineary isotopy $g_{\mu}=(1-\mu)(\psi_1^{-1})^{\ast}(g)+\mu h^2_1g_{cyl}$ on $S^2\times (-1/\e, 0.25/\e)$ and $g_{\mu}=(1-\mu)(\psi_{a+1}^{-1})^{\ast}(g)+\mu h^2_{a+1} g_{cyl}$ on $S^2\times ((-0.25+a(0.9))/\e,(1+a(0.9))/\e)$.}	
	\end{enumerate}	
\end{lem}	

\begin{proof}
	This follows at once by applying Lemma \ref{lem:Combo1} above exactly $a$ times, picking $z=z_{i+1}$ at the $i$-th application.
\end{proof}

\subsection{Equivariant surgery}\label{subs:Surgery}

We will mostly stick to the notation of the monograph by Morgan-Tian \cite{MT07} (see in part. Chapter 13 for the part on Perelman's notion of surgery).

We let $g_{std}$ denote the \emph{standard initial metric on $\R^3$} with tip at the origin. The metric in question is rotationally symmetric, has non-negative sectional curvature and we recall the existence of $A_0>0$ such that $\psi: (\R^3\setminus B_{A_0}(0), g_{std})\to (S^2\times (-\infty, 4], g_{cyl})$ given by $\psi(x)=(x/|x|, |x|)$ is an isometry.

\

We let, for purely notational convenience, $K_{std}$ denote the smooth manifold obtained as \[
((S^2\times (-4,4))\cup B_{A_0+4}(0))/((x,t)\sim \psi^{-1}(x,t)), \ (x,t)\in S^2\times (-4,4)
\]
that is clearly diffeomorphic to $\R^3$. There is a well-defined function $s:K_{std}\to \R$ that extends the projection onto the second factor associated to $\psi$; such a function takes values in $(-\infty, 4+A_0]$). We further consider the involution $\tau$ on $K_{std}$ defined by letting $\tau=\sigma$ on  $S^2\times (-4,4)$ and $\tau(x_1, x_2, x_3)=\tau(x_1, x_2, -x_3)$ on $B_{A_0+4}(0)$ (regarded as an open ball in $\R^3$). 

Now, let $h$ be a Riemannian metric on $(S^2\times (-4,4))$ that is $\sigma$-invariant (meaning that $\sigma$ is an isometry of $(S^2\times (-4,4),h)$). If we further assume, based on the notion of $\e$-neck, that it be $\e$-close in $C^{[1/\e]}$ to $g_{cyl}$ then the whole discussion presented in Section 5 of \cite{Mar12} applies. In particular, we collect for later reference the following two statements.

\begin{lem}\label{lem:Surg}
There exists $\e_2>0$ such that if $h$ is a Riemannian metric on $(S^2\times (-4,4))$ that is $\sigma$-invariant, and $\e$-close in $C^{[1/\e]}$ to $g_{cyl}$ for some $0<\e\leq\e_2$ one can define a metric $h_{surg,\e}$ on $K_{std}$ that satisfies the following properties:
\begin{enumerate}
	\item {$\tau: K_{std}\to K_{std}$ is an isometry of $h_{surg,\e}$;}
	\item {$h_{surg,\e}=h$ on $s^{-1}((-4,0])$;}
	\item {the restriction of $h_{surg,\e}$ to $s^{-1} [1,4+A_0]$ has positive sectional curvature;}
	\item {the scalar curvature of $h_{surg,\e}$ satisfies $R_{h_{surg,\e}}\geq R_{h}$ on $s^{-1}((-4,1])$; in particular the metric $h_{surg,\e}$ has positive scalar curvature (if $\e$ is small enough);}
	\item {the smallest eigenvalue of the curvature operator of $h_{surg,\e}$ (seen as an endomorphism in the standard fashion) is greater or equal that the smallest eigenvalue of the curvature operator of $h$ on $s^{-1}((-4,1])$; in particular the metric $h_{surg,\e}$ has positive sectional curvature if $h$ does.}
\end{enumerate}	
\end{lem}

This same \emph{extension} procedure can also be applied to families of metrics, in particular to paths corresponding to isotopies interpolating between an almost-cylindrical metric $h$ and $g_{cyl}$.

\begin{lem}	\label{lem:Ext}
	Let $h$ be a Riemannian metric on $(S^2\times (-4,4))$ that is $\sigma$-invariant, and $\e$-close in $C^{[1/\e]}$ to $g_{cyl}$ for some $0<\e\leq\e_2$. Then there exists a continuous path of positive scalar curvature metrics on $K_{std}$ starting at $h_{surg,\e}$, ending at a rotationally symmetric metric and restricting to a linear isotopy on $s^{-1}(-4,0]$ for all $\mu\in[0,1]$. Furthermore, the map $\tau$ is an isometry for any $\mu\in [0,1]$.
	\end{lem}

Indeed, if one lets $h_{\mu}=(1-\mu)h+\mu g_{cyl}$ (for which we notice that $\sigma$ is an isometry for any $\mu\in [0,1]$), it suffices to take
$ h'_{\mu}=(h_{\mu})_{surg,\e}. $

We can now explain what we mean, in our context, by equivariant surgery.

\begin{defi}\label{def:EquivSurg}
Given two reflexive triples $(M,g,f)$ and $(\tilde{M},\tilde{g},\tilde{f})$ and $\e>0$ we will say that $(\tilde{M},\tilde{g},\tilde{f})$ is obtained from $(M,g,f)$ by performing a \emph{simple} equivariant surgery if one of the following three statements applies:
\begin{enumerate}
\item {there exists a reflexive $\e$-neck $N\subset M$ of type $T$ and $(\tilde{M},\tilde{g},\tilde{f})$  is obtained from $(M,g,f)$ by cutting along the central sphere of $N$ and gluing caps as decribed above (cf. Lemma \ref{lem:Surg});}
\item {there exists a reflexive $\e$-neck $N\subset M$ of type $C$ and $(\tilde{M},\tilde{g},\tilde{f})$ is obtained from $(M,g,f)$ by cutting along the central sphere of $N$ and gluing caps on both sides (Section 5 of \cite{Mar12}) in the same fashion;}
\item {there exist two disjoint $\e$-necks $N_1, N_2\subset M$ that are $f$-isometric and such that $(N_1\cup N_2)\cap Fix(f)=\emptyset$, and $(\tilde{M},\tilde{g},\tilde{f})$ is obtained from $(M,g,f)$ by cutting along the central spheres of $N_1, N_2$ and gluing caps on both sides to each of them in the same fashion.}	
\end{enumerate}	
In particular, we require (in any of the three cases) that $\tilde{f}=f$ on $\text{Int}(M\cap \tilde{M})$ (the set of interior points of the manifold with boundary $M\cap \tilde{M}$). 
More generally, we will say that $(\tilde{M},\tilde{g},\tilde{f})$ is obtained from $(M,g,f)$ by performing equivariant surgery if it can be realized by finitely many simple equivariant surgeries (one of the three moves above).
\end{defi}

We need to make sure that the previous notion is indeed well-posed, in the following sense: 

\begin{lem}\label{lem:WellPosed}
	If $(M,g,f)$ is a 3-reflexive manifold and equivariant surgery is performed, then one obtains a triple $(\tilde{M},\tilde{g},\tilde{f})$ that is a reflexive 3-manifold.
\end{lem}

\begin{proof}
	Clearly, it is enough to check that we can partition the connected components of $\tilde{M}$ as described in item \emph{iii)} of Definition \ref{def:Z2equiv}. Also, without loss of generality we can deal with the case of \emph{simple} equivariant surgeries.
	
	By the way surgeries are defined, we need to prove that \emph{if}  $\tilde{M}^*$ is a connected component of $\tilde{M}$ such that $\tilde{f}(\tilde{M}^*)\subset \tilde{M}^*$, then $Fix(\tilde{f})\cap \tilde{M}^*\neq\emptyset$. 
	Recall that, in any of the three cases of simple surgery above, we can exploit the equation
		\[
		Fix(\tilde{f})=Fix(f)\cap \text{Int}(\tilde{M}\cap M)
		\]
		to get
		\[
		Fix(\tilde{f})\cap \tilde{M}^* = Fix(f)\cap \text{Int}(\tilde{M}^*\cap M)
		\]
		and thus it suffices to check that if $\tilde{f}(\tilde{M}^*)\subset \tilde{M}^*$  then the right-hand side is not empty. But this is simple.
	First, one just needs to observe that $\tilde{M}^*\cap M$ is connected (roughly speaking: intersecting $\tilde{M}^*$ with $M$ corresponds to removing at most two disjoint open balls from a connected manifold).
	Second, one can pick any point $p\in  \text{Int}(\tilde{M}^*\cap M)\setminus Fix(f)$ and join it to $\tilde{f}(p)$ by a curve $\sigma$ that is entirely contained in the interior of $\tilde{M}^*\cap M\subset M$. By Lemma \ref{lem:mix1} and Lemma \ref{lem:mix2} that same curve $\sigma$, regarded as a subset of $M$, must intersect $Fix(f)$.
		\end{proof}

\begin{rmk}\label{rmk:spheres}	
It follows by the way the surgery procedure is defined that if $Fix(f)$ is diffeomorphic to $S^2$, then either $Fix(\tilde{f})=\emptyset$ or  the same conclusion holds true for \emph{each connected component} of $Fix(\tilde{f})$: the set of fixed points loses one connected component in case (2), and is modified in case (1) but not in case (3).
\end{rmk}
    	
    \subsection{Equivariant Ricci flow with surgery}\label{subs:RFS}
    
    Given a compact Riemannian manifold $(M^3,g)$, without boundary, it is well-known that Hamilton's Ricci flow (see \cite{Ham82})
    \begin{equation}\label{eq:RF}
    \begin{cases}
    \partial_t g(t)=-2Ric_{g(t)} \\
    g(0)= g
    \end{cases}
    \end{equation}
    has a unique, locally defined smooth solution up to a maximal time $t^{\ast}>0$; as a result of the uniqueness statement the flow happens equivariantly with respect to any given group of isometries of the initial metric $g$. It was observed in \cite{DL09}, Section 5.1, that the Ricci flow with surgery presented by Perelman can also be constructed equivariantly: since the regions of high curvature, the \emph{horns} and the \emph{continuing regions} are, by their very definitions, invariant under the group action in question, one can perform surgery (i.~e. remove necks, add spherical caps and extend the metric) in an equivariant fashion. We shall now present the results that will be needed in the proof of our main theorems, specified to our setting. The reader may wish to consult Section 7 of \cite{Mar12} for the basic definitions we employ, and the monographs \cite{MT07} and \cite{KL08} for a comprehensive introduction to the Ricci flow with surgery. 
    
    \begin{defi}\label{def:CanNeigh}
    	Let $(M^3,g,f)$ be a reflexive triple as per Definition \ref{def:Z2equiv} and let $g(t), t\in [a,b)$ evolve by Ricci flow with initial condition $g(0)=g$. Given positive constants $C, \e$ we shall say that $(M,g,f)$ satisfies the $(C,\e)$ reflexive canonical neighborhood property with parameter $r>0$ if every point $x\in M$ such that for some $t\in [a,b)$ the inequality $R_{g(t)}\geq r^{-1}$ is satisfied at $x$ has a canonical neighborhood (in the sense of Perelman) and, in addition:
    	\begin{itemize}
    	\item {if $z\in Fix(f)$ is the center of an $\e$-neck, then it is also the center of a reflexive $\e$-neck (Definition \ref{def:EquivNeck});}
    	\item {if $z\in Fix(f)$ is contained in the core of a $(C,\e)$-cap, then there exists an $f$-invariant $(C,\e)$-cap
    		$K=Y\cup N$ whose core $Y$ contains $z$, and such that the pull-back action restricts to an isometry of the reflexive $\e$-neck $N$; furthermore in the latter case the boundary of the core $\partial Y$ is the central sphere of a reflexive $\e$-neck.}
    	\end{itemize}	
    \end{defi}

\begin{rmk}\label{rmk:EquivComp}
	It follows from Section \ref{sec:EquivMan} that if instead, in the setting above, $U$ is a $C$-component or an $\e$-round component, then one of the following two assertions applies:
	\begin{itemize}
	\item {\emph{either} $(U,g^{\bullet}, f^{\bullet})$ is a reflexive triple (which happens if $Fix(f)\cap U\neq\emptyset$);}
	\item {\emph{or} there exists another component of the same type of $U$ that is $f$-isometric to $U$ (which happen if instead $Fix(f)\cap U=\emptyset$).}	
	\end{itemize}	
\end{rmk}

\begin{rmk}\label{rmk:localmodelsCaps}
	In the same setting, let now $K\subset M$ be a $(C,\e)$-cap that is reflexive in the sense explained in Definition \ref{def:CanNeigh}. In this case, since $Fix(f)\cap Y\neq\emptyset$ (by assumption), a standard connectedness argument shows that $N$ must be a neck of type $T$.
\end{rmk}

    \begin{thm}\label{thm:RFS}(cf. Section 5.1 in \cite{DL09},  Theorem 7.1 in \cite{Mar12})
    	Let $(M^3, g,f)$ be a connected reflexive triple, having positive scalar curvature. Then there exist:
    	\begin{enumerate}
    		\item [a)]{positive constants $C, \epsilon$;}
    		\item [b)] {canonical neighborhood parameters $\textbf{r}= \left\{r_i\right\}^{j}_{i=0}$ such that
    			\[
    			r_0=\epsilon\geq r_1\geq r_2\geq \ldots\geq r_j>0;
    			\]
    		}	
    		\item[c)]{surgery control parameters $\Delta= \left\{\delta_i\right\}^{j}_{i=0}$ such that $\delta_0\leq \epsilon/6$ and
    			\[
    			\delta_0\geq \delta_1\geq \delta_2\geq \ldots\geq \delta_{j}>0;
    			\]}	
    	\end{enumerate}	
    	and a Ricci flow with surgery $(M_i, g_i(t)_{t\in[t_i,t_{i+1})}, f_i)$ with $0\leq i\leq j$ such that
    	\begin{enumerate}
    		\item {$M_0=M, \ \ g_0(0)=g, \ \ f_0=f$;}	
    		\item {for every $i=0,\ldots, j$ and $t\in[t_i, t_{i+1})$ the triple $(M_i, g_i, f_i)$ is a reflexive manifold;}
    		\item {the flow becomes extinct at finite time $t_{j+1}<\infty$;}
    			\item {the scalar curvature of $g_i(t)$ is bounded from below by $\inf_{x\in M} R_g(x)$ for all $i=0,\ldots, j$ and $t\in [t_i, t_{i+1})$;}
    		\item {for every $i=0,\ldots, j$ 
    			 the Ricci flow $(M_i,g_i(t),f_i), t\in[t_i,t_{i+1})$ satisfies the $(C,\epsilon)$ reflexive canonical neighborhood property with parameter $r_i$, for all $0\leq i\leq j$;}
    		\item {for every $i=0,\ldots, j$ the Riemannian manifold $(M_{i+1},g_{i+1}(t_{i+1}))$ is obtained from the Ricci flow $(M_{i},g_{i}(t))_{t\in [t_i, t_{i+1})}$ by performing equivariant surgery at time $t_{i+1}$ with parameters $r_i$ and $\delta_i$,
    			and $f_{i+1}=f_i$ on $\text{Int}(M_{i+1} \cap M_i)$.}
    	\end{enumerate}	
    \end{thm}

    \begin{rmk}\label{rmk:equivNeighb}
    	The fact that condition (5) above can be accomodated relies on the discussion presented in Section 3.3 of \cite{DL09}, see in particular Lemma 3.8 therein (applied for $G=H=\mathbb{Z}/2\mathbb{Z}$, the group with two elements). 
    \end{rmk}

    \begin{rmk}\label{rem:cases} 
    	 Concerning the position of any given connected component $\Sigma$ of the set $Fix(f_i)$ with respect to the region affected by surgery at time $t_{i+1}$, we shall distinguish the following cases:
    	\begin{enumerate}
    		\item [\emph{i})]{$\Sigma\subset C_{t_{i+1}}$, i.~e. $\Sigma$ is contained in the continuing region of the Ricci flow with surgery, and is therefore unaffected by the surgery procedure at time $t_{i+1}$;}
    		\item [\emph{ii})]{$\Sigma\subset C^{c}_{t_{i+1}}$, i.~e. $\Sigma$ is contained in the complement of the continuing region of the Ricci flow with surgery, thus (by the way this region is defined) the scalar curvature at any of its points is larger than $\rho^{-2}_i$ and so larger than $r^{-2}_i$ as well (for $\rho_i=\delta_i r_i$ and $\delta_i\in (0,1)$, by definition): thus $\Sigma$ is contained in the union of finitely many $(C,\epsilon)$-canonical neighborhoods, where the parameters $C$ and $\epsilon$ are those mentioned in the statement of Theorem \ref{thm:RFS}. Hence, there are the following subcases:
    			\begin{enumerate}
    				\item[\emph{ii.a})] {$\Sigma$ is contained in a $C$-component, or in an $\epsilon$-round component; in both cases the ambient component has positive sectional curvature and will become extinct in finite time;}
    				\item[\emph{ii.b})]{$\Sigma$ is entirely contained in an $\epsilon$-neck: in this case there exists a possibly different reflexive $\epsilon$-neck of type $C$ for which $\Sigma$ is the central sphere.}
    			\end{enumerate}
    		}
    		\item [\emph{iii})]{if none of the two applies, then (since the continuing regions are chosen to be equivariant with respect to $f_i$) the closed surface $\Sigma$ will meet $\partial C_{t_{i+1}}$ transversely, in fact orthogonally (in the neck charts) along either one or two closed curves.}	
    	\end{enumerate}	
    \end{rmk}

	\subsection{Equivariant Gromov-Lawson connected sums}\label{subs:equivGL}
	Let $(M^n,g)$ be a compact Riemannian manifold without boundary, of positive scalar curvature. Given $p\in M$, $\{e_i\}\subset T_p M$ an orthonormal basis, and $r_0>0$, the Gromov-Lawson construction defines a positive scalar curvature metric $g'$ on $B_{r_0}(p)\setminus \{p\}$ that coincides with $g$ near the boundary $\partial B_{r_0}(p)$, and such that $(B_{r_2}(p)\setminus \{p\},g')$ is isometric to a half cylinder for some $r_2\in (0,r_0)$.
	
	We refer the reader to Section 1 of \cite{GL80b} for the details of the construction, see also Section 6 of \cite{Mar12}. For our purposes, we just state the result we need. We identify the neighborhood $B_{r_0}(p)$ with $B_{r_0}(0)\subset \R^n$ via the choice of an orthonormal frame $\{e_k\}$ and the exponential normal coordinates. Then given $r_0<\frac{1}{2}\min\{\inj_M(p),1\}$, there exists a planar curve $\gamma\subset \R^2$, depending on the lower bound of scalar curvature of $g$, and the upper bound of the $C^2$ norm of $g$, that satisfies the following four properties:
	
	\begin{enumerate}
		\item the image of $\gamma$ is contained in the region $\{(r,t):r\ge 0,t\ge 0\}$;
		\item the image of $\gamma$ contains the half-line $r\ge r_1,t=0$ for some $r_1\in (0,r_0)$;
		\item the image of $\gamma$ contains the half-line $r=r_2,t\ge t_2$ for some $r_2\in (0,r_1)$ and $t_2>0$;
		\item the induced metric on $M'=\{(x,t):(|x|,t)\in \gamma\}$ as a submanifold of the Riemannian product $B_{r_0}(p)\times \R$ has positive scalar curvature.
	\end{enumerate}
	
	By choosing $r_2$ sufficiently small, the induced metric on the tubular piece $r=r_2,t\ge t_2$ is a perturbation of the standard cylindrical metric on $S_{r_2}^{n-1}(0)\times \R$. Then with a cutoff function, we can slightly modify the induced metric and obtain a positive scalar curvature metric $g'$ of the form $g_{ij}(x,t)dx_i dx_j+dt^2$, which coincides with the original metric near $t_2$, and is isometric to $S_{r_2}^{n-1}(0)\times [t_3,\infty)$ for some $t_3>t_2$.
	
	\begin{rmk}\label{remark.GL.connected.sum.locally.conformall.flat}
		If the metric $g$ has constant positive sectional curvature on $B_{r_0}(p)$, then the metric $g'$ is rotationally symmetric on $B_{r_0}(p)\setminus \{0\}$. As a result, any connected sum of round spheres (or isometric quotients thereof) is locally conformally flat.
	\end{rmk}
	
	The above construction can be done in an equivariant way. Precisely, let $(M,g,f)$ be a \emph{connected} reflexive $n$-manifold. Assume $p\in \Sigma:=Fix(f)$. We may choose the local orthonormal basis at $p$, such that $\{e_i\}_{i=1}^{n-1}$ is a local orthonormal basis for $T_p\Sigma$, and $e_n$ is normal to $\Sigma$. Based on the local description of $f$ near its fixed locus, see Lemma \ref{lem:LocalDesc}, the choice of the planar curve $\gamma$ can be made so that the induced metric on $M'=\{(x,t):(|x|,t)\in \gamma\}$ is invariant under the coordinate change $(x_1,x_2,\cdots,x_n,t)\mapsto (x_1,x_2,\cdots,-x_n,t)$. We may then proceed to consider a modification $g'=g'_{ij}(x,t)dx_idx_j+dt^2$, such that $g'$ is also invariant under the same coordinate change, in addition to all the previous properties. As a result, the manifold $(M\setminus \{p\},g',f)$ has positive scalar curvature, and is itself a reflexive manifold.
	
	\
	
	Let now $(M_1,g_1)$ and $(M_2,g_2)$ be \emph{connected} compact manifolds, without boundary, of positive scalar curvature. Assume further that $(M_1\sqcup M_2,g_1\sqcup g_2, f)$ is a reflexive manifold for some smooth $f$.  We distinguish two possible types of equivariant Gromov-Lawson connected sums:
	
	\textbf{Type T}: assume that $f(M_1)=M_1$, thus also $f(M_2)=M_2$. In this case, $M_1,M_2$ are both reflexive components. Thus, $f|_{M_j}$, $j=1,2$, has a fix point set $\Sigma_j$, where $\Sigma_j$ is a totally geodesic separating hypersurface of $M_j$. Given $p_j\in \Sigma_j$, take orthonormal bases $\{e_k\}\subset T_{p_1}M_1$, $\{\overline{e}_k\}\subset T_{p_2} M_2$, such that $e_n$, $\overline{e}_n$ are normal to $\Sigma_1$, $\Sigma_2$, respectively. There exist $0<r_0<\frac{1}{2}\min \{\inj_{M_1},\inj_{M_2},1\}$, so that the previous construction with $\gamma$ applies to both manifolds. As a result, we can glue the manifolds $(B_{r_0}(p_1)\setminus \{p_1\},g_1')$ and $(B_{r_0}(p_2)\setminus \{p_2\}, g_2')$ along the spheres where $t=t_3+1$ with reverse orientations. The result is a positive scalar curvature, reflexive manifold ($M_1\#M_2, g_1\# g_2, f')$, such that the fix point set of $f'$ is diffeomorphic to $\Sigma_1\#\Sigma_2$. The metric $g_1\# g_2$ depends only on $g_1,g_2$, and the choice of parameters.
	
		\textbf{Type C}: assume that $f(M_1)=M_2$, thus also $f(M_2)=M_1$ and $Fix(f)=\emptyset$. Given $p_1\in M_1$, take $p_2=f(p_1)\in M_2$. Take an orthonormal basis $\{e_k\}\subset T_{p_1}M_1$. Then the push forward $\{\overline{e}_k=(df)_* (e_k)\}\subset T_{p_2}M_2$ is also an orthonormal basis. Applying the non-equivariant Gromov-Lawson construction with the same curve $\gamma$  to both manifolds, we may then glue $(B_{r_0}(p_1)\setminus \{p_1\},g_1')$ and $(B_{r_0}(p_2)\setminus \{p_2\}, g_2')$ along the spheres where $t=t_3+1$ with reverse orientations. As a result, we construct a positive scalar curvature metric on $M_1\#M_2$, equipped with an isometry $f'$. Here $f'=f$ on $M_j\setminus B_{r_0}(p_j)$ and, after gluing, $f'$ has a connected fix point set: it is the sphere $S_{r_2}^{n-1}(0)\times \{t_3+1\}$ in the above construction. This construction also depends only on $g_1,g_2$, and the choice of parameters.
	
\
	
	We further consider a third type of operation, which (with some abuse of language) we shall call \emph{type D double connected sum}. In this case we start with two \emph{possibly disconnected} compact manifolds, without boundary, of positive scalar curvature $(M_1,g_1)$ and $(M_2,g_2)$; we suppose that $(M_1\sqcup M_2,g_1\sqcup g_2, f)$ is a reflexive manifold for some smooth $f$.

	\textbf{Type D}: in this case we assume that both pieces $(M_j, g_j, f|_{M_j}), j=1,2$ be reflexive, but we contemplate three possible scenarios, depending on whether they are both connected, both disconnected, or $(M_1, g_1, f|_{M_1})$ is connected while $(M_2, g_2, f|_{M_2})$ is not.
	The manifold $(M,g,f)$ is obtained by joining two distinct points, located away from $Fix(f)$, $p_1\in M_1, p_2\in M_2$ via a non-equivariant neck (namely: through a standard Gromov-Lawson connected sum), and then performing the very same construction on the couple of $f$-related points namely picking $q_1=f(p_1)\in M_1, q_2=f(p_2)\in M_2$, and correspondly $f$-associated bases. That is to say: 
	we take orthonormal bases $\{e_k\}\subset T_{p_1}M_1$, $\{\overline{e}_k\}\subset T_{p_2} M_2$, and we consider $\left\{(df)_* (e_k) \right\}\subset T_{q_1}M_1$ and $\left\{(df)_* (\overline{e}_k) \right\}\subset T_{q_2} M_2$; we then
	glue the manifolds $(B_{r_0}(p_1)\setminus \{p_1\},g_1')$ and $(B_{r_0}(p_2)\setminus \{p_2\}, g_2')$ along the spheres where $t=t_3+1$ with reverse orientations, and analogously we 
	glue $(B_{r_0}(q_1)\setminus \{q_1\},g_1')$ and $(B_{r_0}(q_2)\setminus \{q_2\}, g_2')$ along the spheres where $t=t_3+1$ with reverse orientations. The result is a positive scalar curvature, reflexive manifold ($M_1\#M_2, g_1\# g_2, f')$, such that the fix point set of $f'$ is equal to the disjoint union $\Sigma_1\sqcup\Sigma_2$.

	\begin{figure}[htbp]
		\centering
		\includegraphics[width=\textwidth]{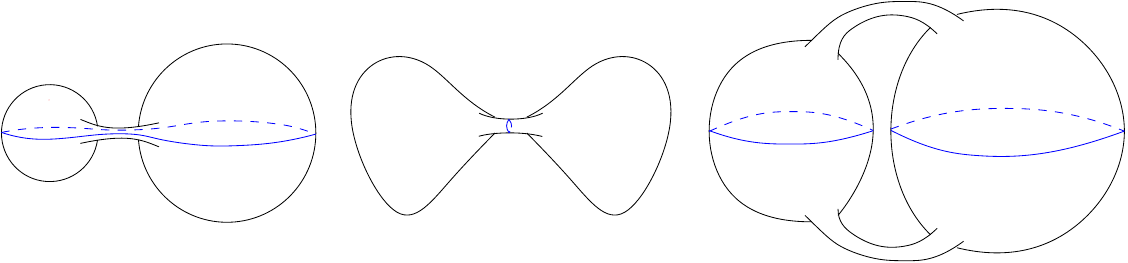}
		\caption{A schematic illustration of the three possible types of  Gromov-Lawson equivariant connected sums.}
		\label{fig:EquivGLsums}
	\end{figure}

	Notice that in any of the three cases we have $f'=f$ on $M_1\sqcup M_2$ minus the basepoints for the connected sum, namely $f'=f$ on $M_1\sqcup M_2\setminus\left\{p_1, p_2\right\}$ in the first and second cases, while $f'=f$ on $M_1\sqcup M_2\setminus\left\{p_1, p_2, q_1, q_2\right\}$ in the third one.
	
	\
	
		\begin{rmk}
		\label{rmk:Deg}
		For any of the three types of equivariant connected sum described above, we also allow $(M_1, g_1, f_1)=(M_2, g_2, f_2)$ provided $(M_1, g_1, f_1)$ is a reflexive manifold (that is additionally required to be connected for type $T$ or $C$ connected sum). More precisely, it is here understood that we have \emph{one} input manifold, not two copies of the same manifold, and we are essentially attaching a handle to a reflexive manifold (type T or type C, where in both cases $f_1(M_1)\subset M_1$ hence $Fix(f_1)\neq\emptyset$, but the basepoints lie on $Fix(f_1)$ in the first case but not on the second) or two handles instead (type D, where the basepoints are two pairs lying in the same connected component of $M_1\setminus Fix(f_1)$ and the handles do not contribute to the symmetry locus of the resulting manifold). The results stated in this section also apply to this (important) special case.
	\end{rmk}
	
	By choosing the parameters continuously, the above construction applies to families of reflexive metrics.  Consider a continuous family of reflexive manifolds $(M_1\sqcup M_2, g_{1,\mu}\sqcup g_{2,\mu},f)$, $\mu\in [0,1]$, points $p_{i,\mu}\in M_1\sqcup M_2$ and orthonormal bases $\{e_k^{(i)}(\mu)\}$ of $T_{p_{i,\mu}}(M_1\sqcup M_2)$, $i\in I$.  We call $(p_{i,\mu},\{e_k^{(i)}(\mu)\})$, $i\in I$ continuous $f$-equivariant choices of data, if one of the following assertions (corresponding to the three possible operations above) hold true:
	\begin{itemize}
		\item {$M_1,M_2$ are reflexive components, $I=\left\{1,2\right\}$, $p_{i,\mu}\in Fix(f)$, and $e_n^{(i)}(\mu)\perp Fix(f)$ for all $\mu\in [0,1]$, $i=1,2$;}
		\item {$M_1,M_2$ are non-reflexive components, $I=\left\{1,2\right\}$, $p_{2,\mu}=f(p_{1,\mu})$, and $e_k^{(2)}(\mu)=(df)_* (e_k^{(1)}(\mu))$ for all $\mu\in [0,1]$, $k=1,\ldots, n$.}
			\item {$M_1,M_2$ are (possibly disconnected) reflexive components, $I=\left\{1,2, 3, 4\right\}$, $p_{1,\mu}, p_{3,\mu}\in M_1\setminus Fix(f)$, $p_{2,\mu}, p_{4,\mu}\in M_2\setminus Fix(f)$  and for all $\mu\in [0,1]$ we have
			\[
			p_{3,\mu}=f(p_{1,\mu}), \ p_{4,\mu}=f(p_{2,\mu}), 
			\]
			\[ e_k^{(3)}(\mu)=(df)_* (e_k^{(1)}(\mu)), \ e_k^{(4)}(\mu)=(df)_* (e_k^{(2)}(\mu)) \ \ k=1,\ldots, n;
			\] 
		}
	\end{itemize}

\begin{figure}[htbp]
	\centering
	\includegraphics[scale=0.86]{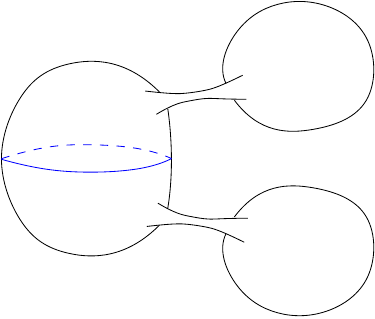}
	\caption{A type $D$ double connected sum involving a connected reflexive manifold $(M_1, g_1, f_1)$ and a disconnected reflexive pair $(M_2, g_2, f_2)$.}
	\label{fig:EquivGLsumsD}
\end{figure}

	\begin{lem}\label{lemma.equivariant.GL.construction}
		Let $\mu:[0,1]\mapsto g_{i,\mu}$ be continuous paths of positive scalar curvature metrics on the compact manifolds without boundary $M_1^n, M_2^n$. Assume that $(M_1\sqcup M_2,g_{1,\mu}\sqcup g_{2,\mu}, f)$ is a reflexive $n$-manifold for all $\mu\in [0,1]$ for some smooth map $f$. Given continuous $f$-equivariant choices of data $(p_{i,\mu},\{e_{k}^{(i)}(\mu)\})$, there exists $r_0>0$ such that the Gromov-Lawson connected sums of $(M_1, g_{1,\mu})$ and$(M_2,g_{2,\mu})$ constructed with $\gamma=\gamma(r_0)$ and $\{e_{k}^{(i)}(\mu)\}$ at $p_{i,\mu}$, form a family of reflexive manifolds $(M_1\#M_2, g_{1,\mu}\# g_{2,\mu}, f')$ whose metrics vary continuously as one lets $\mu\in [0,1]$.
	Furthermore $g_{1,\mu}\# g_{2,\mu}=g_{1,\mu}$, $f'=f$ on $M_1\setminus B_{r_0}(p_{1,\mu})$; $g_{1,\mu}\# g_{2,\mu}=g_{2,\mu}$, $f'=f$ on $M_2\setminus B_{r_0}(p_{2,\mu})$, for all $\mu\in [0,1]$. 
	\end{lem}
	
	\begin{rmk}\label{remark:GL.connected.sum.and.equivariant.Kuiper}
		Recall that if $(M\simeq S^3 ,g,f)$ is a \emph{locally conformally flat} reflexive manifold with $f=\rho$ (which, by Corollary \ref{cor:RedStandard} one can always assume), then there exists a conformal diffeomorphism $\phi:(M,g,f)\rightarrow (M,g_{round},f)$, such that $\phi^* g_{round}= u^{4}g$, and $f^* u=u$. Therefore, if $(M_1, g_1),(M_2, g_2)$ are both round spheres and one peforms a type $T$ or type $C$ Gromov-Lawson connected sum, then under the assumption of Lemma \ref{lemma.equivariant.GL.construction}, $(M_1\#M_2, g_{1}\#g_{2},f)$ is isotopic, through reflexive metrics, to the standard round sphere.
	\end{rmk}
	
	We can employ the equivariant Gromov-Lawson construction to \emph{reconstruct} those necks that are cut by surgery. Precisely, let $(S^2\times (-4,4),h,f)$ be an equivariant neck, with $f=\sigma$ (the reflection of type $T$ necks) and assume $h$ is within $\e$ of the cylindrical metric $g_{cyl}$ in the $C^{[1/\e]}$-topology. We then apply surgery along the central sphere $S^2\times \{0\}$ on this neck, glue standard caps, and obtain two manifolds $(S^{-},h_{surg,\e}^{-})$ and  $(S^{+},h_{surg,\e}^{+})$ as explained in Section \ref{subs:Surgery}. Notice that the (disconnected) manifold $(S^{-}\sqcup S^{+},h_{surg,\e}^{-}\sqcup h_{surg,\e}^{+})$ also comes equipped with an isometric involution $f'$ that makes it a reflexive manifold.
	We can take the tips $p^{-},p^{+}$ of the spherical caps, and apply the equivariant Gromov-Lawson construction to $(S^{-}\sqcup S^{+},h_{surg,\e}^{-}\sqcup h_{surg,\e}^{+},f')$ at $p^{-},p^{+}$, obtaining a connected sum $(S^{-}\#S^{+}, h_{surg,\e}^{-}\# h_{surg,\e}^{+})$ of positive scalar curvature. Notice that the resulting manifold $(S^{-}\#S^{+}, h_{surg,\e}^{-}\# h_{surg,\e}^{+})$ is also equipped with the isometric involution $f$ (that makes the triple $(S^{-}\#S^{+}, h_{surg,\e}^{-}\# h_{surg,\e}^{+},f)$ a connected reflexive manifold).
The case when instead we start with a neck to type $C$, and perform (after the surgery cutting along the central sphere) a type $C$ connected sum is similar.
	
	\
	
	The same proofs as in Lemma 6.2 and Lemma 6.3 in \cite{Mar12} apply in the equivariant setting, and we conclude the following important result:
	
	\begin{lem}\label{lemma.GL.and.RFsurgery.are.isotopic}
		Fix $\e>0$ sufficiently small. Let $(S^2\times (-4,4),h,f)$ be a reflexive neck, where either $f=\sigma$ (type $T$) or $f=\kappa$ (type $C$). Suppose $h$ is within $\e$ of the cylindrical metric $g_{cyl}$ on $S^2\times (-4,4)$ in the $C^{[1/\e]}$ topology, then $(S^2\times (-4,4),h,f)$ can be continuously deformed to $(S^{-}\#S^{+}, h_{surg,\e}^{-}\# h_{surg,\e}^{+},f)$, through a reflexive isotopy of positive scalar curvature metrics which all coincide with $h$ near the ends of $S^2\times (-4,4)$.
	\end{lem}
	
	\begin{rmk}
		\label{rmk:typeDreconstruction}
	When we have a reflexive triple $(M,g,f)$ containing two $f$-isometric $\e$-necks (disjoint from $Fix(f)$) one can also obtain a suitable \emph{reconstruction lemma} by simply invoking Lemma 6.3 in \cite{Mar12} for a type $D$ Gromov-Lawson connected sum.
	\end{rmk}

	\section{Proof of Theorem \ref{thm:B}}\label{sec:path}
	
	Following the general scheme presented in the introduction, we shall now proceed to the proof of Theorem \ref{thm:B} by first discussing the 
	necessary steps in the special case when $X^3\simeq D^3$ (these steps are the object of the next three sections, see below for the key statements, provided at the beginning of each section). We will then describe the general model triples we employ (Section \ref{subs:ModelMet}) and present, in Section \ref{subs:PathGen}, the necessary modifications needed to handle the case when $X^3$ is not necessarily a disk (but has a topology compatible with what is stated in Theorem \ref{thm:A}). We wrap up everything in Section \ref{subs:pcmet}, where we also present the analogous path-connetedness results for $\mathcal{M}_{R>0, H\geq 0}, \mathcal{M}_{R\geq 0, H>0}$ and $\mathcal{M}_{R\geq 0, H\geq 0}$.

	\subsection{Basis of the induction: a manifold covered by canonical neighborhoods}\label{subs:basis}
	
	It is convenient to introduce a simple definition:
	
	\begin{defi}
		\label{def:SphMod}
	A connected reflexive triple $(M,g,f)$ is called a (spherical) model triple if $M=S^3$, $g$ is the standard (unit) round metric, and $f$ is the standard reflection $\rho$. More generally, we shall say that a reflexive triple $(M,g,f)$ is a (spherical) model triple if each connected component is endowed, when restricting $g$, with a standard (unit) round metric. 
	\end{defi}

\begin{rmk}
	\label{rmk:Lift}
Let us consider the statement of Theorem \ref{thm:B} in the special case when $M\simeq S^3$ and $Ric_g>0$. In this scenario, we know (by \cite{Ham82}) the existence of a reflexive isotopy connecting $(M,g,f)$ to an endpoint triple of the form $(M,\tilde{g},f)$ where $\tilde{g}$ is a round metric (which we can always normalize to have, say, scalar curvature equal to 6). Note that $(M,\tilde{g},f)$ will not, at least in general, be a model triple. Yet, there exists a diffeomorphism $\varphi\in\mathcal{D}$ such that $\varphi^{\ast}\tilde{g}$ is indeed the standard round metric, and thus (by Corollary \ref{cor:UniqueF}) we also know that $\varphi^{\ast}f=\rho$. As a result, possibly modifying the given isotopy via the action of a diffeomorphism we will indeed obtain a continuous path ending at the (spherical) model triple. All arguments we are about to present, which concern isotopies of classes (i.~.e. paths happening in the moduli space $\mathcal{M}/\mathcal{D}$) should always be read in relation to this \emph{lifting argument}, which applies to the most general situation discussed in Section \ref{subs:PathGen} and \ref{subs:pcmet}.
\end{rmk}
	
	The key result we prove here can be stated as follows:
	
	\begin{prop}\label{pro:can}
	Let $(M\simeq S^3,g,f)$ be a \emph{connected} reflexive 3-manifold having positive scalar curvature, and being covered by $(C,\e)$ reflexive canonical neighborhoods. Then there exists a reflexive isotopy of classes connecting it to a (spherical) model triple.
	\end{prop}

\begin{rmk}
	\label{rem:Consistency}
	Consistently with Definition \ref{def:CanNeigh}, when we say that $(M\simeq S^3,g,f)$ is covered by $(C,\e)$ reflexive canonical neighborhoods we mean that any point $x\in M$ has a canonical neighborhood, and if $x\in Fix(f)$ then such neighborhood can always be taken to be equivariant in the sense explained above (that is to say: the neighborhood in question is a reflexive $\e$-neck, a reflexive $(C,\e)$-cap, or a reflexive component as per Remark \ref{rmk:EquivComp}).
\end{rmk}	

\begin{rmk}
	\label{rmk:TopSigma}
	A simple but important remark, to be kept in mind throughout the proof we are about to present, is that if $(M\simeq S^3,g,f)$ is a \emph{connected} reflexive 3-manifold, then necessarily $Fix(f)\simeq S^2$. Indeed, $Fix(f)$ is assumed to be separating (cf. Remark \ref{rmk:separating}) thus it must on the one end be connected, and on the other be two-sided thus necessarily (being $M$ orientable) a genus $\gamma$ surface for some $\gamma\geq 0$. If we then denote by $X$ (the closure of) any of the two fundamental domains of $f$ (i.~e. either of the two compact manifolds with boundary determined by $Fix(f)\subset M$) we have that $X$ must be a genus $\gamma$ handlebody.
	If we double $(X,g)$, in the standard sense of pointwise identification of boundary points, then it is easily checked that the double (namely, in this case, $M$) is homeomorphic to the connected sum of exactly $\gamma$ copies of $S^2\times S^1$ unless $\gamma=0$, in which case we obviously obtain $M\simeq S^3$).
\end{rmk}

\begin{rmk}\label{rmk:disjointCC}
	In view of the discussion to be presented in Section \ref{subs:conncomp}, let us add some comments about the applicability of the statement above. Thus, let $(M_i, g_i, f_i)$ be a Ricci flow with surgery (cf. Theorem \ref{thm:RFS}): we know that short before the extinction time $t_{j+1}<\infty$, say at time $t_{j+1}-\eta$ the (possibly disconnected) manifold $M_j$ will be covered by canonical neighborhoods, that are indeed of reflexive type when centered at points of $Fix(f_j)$. Now, assuming we started with $M_0\simeq S^3$, all connected components of $M_j$ will also be diffeomorphic to $S^3$, and can be divided into two families: the reflexive ones (for which we shall apply Proposition \ref{pro:can} above), and the non-reflexive ones, that come in pairs (for which we shall instead apply, without any modification, Proposition 8.1 in \cite{Mar12}). Of course, when dealing with a pair of non-reflexive components, say $M^{*}_j, M^{**}_j$ if $\mu\in [0,1]\mapsto [g^*_\mu]$ is an isotopy of classes starting at $(M^{*}_j, g(t_{j+1}-\eta))$ and ending at a round metric (as provided Proposition 8.1 in \cite{Mar12}), we agree to consider for $M^{**}_j$ the isotopy given by $\mu\in [0,1]\mapsto [f^{\ast}_j(g^*_\mu)]$.
	
	As a result, once Proposition \ref{pro:can} is proven we shall have an isotopy at the level of the disjoint union of connected components (of course, possibly reparametrizing, we can and we shall assume that all isotopies in question are defined over a common time interval, which we typically convene to be $[0,1]$).
\end{rmk}

		\begin{proof}
First of all, if $(M,g)$ is a $C$-component or an $\e$-round component, the desired equivariant isotopy is provided by Hamilton's normalized Ricci flow (since in both cases the manifold in question would have positive sectional curvature). This case being ruled out, we have that \emph{each point $z$ of $\Sigma:=Fix(f)$} has a canonical neighborhood that is either a reflexive $\e$-neck (centered at $z$) or a reflexive $(C,\e)$-cap (whose core contains $z$). 

\

\emph{Case 1: for all $z\in\Sigma$ \textbf{only} the latter alternative holds, i.~e. $z$ belongs to the core of a reflexive $(C,\e)$-cap $K=K(z)$.}

\

In this case, take any $z\in\Sigma$ and let $K$ be the $(C,\e)$-cap in question (for which we write $K=Y\cup N$ with, as it is standard notation, $Y$ the core and $N$ the neck of the cap). The two-dimensional sphere $\partial Y$ does not intersect $\Sigma$, (otherwise, by definition of $(C,\e)-$cap, one would find $z'\in\partial Y\cap\Sigma$ center of a reflexive $\e$-neck) thus it must enclose it (due to the fact that $z\in Y\cap \Sigma$ by assumption), so that clearly $M=Y\cup f(Y)$. If both caps are of type $A$ (in the sense of \cite{Mar12}, page 837) then $M$ has positive sectional curvature and the isotopy is, once again, provided by flowing via normalized Ricci flow. Else, if both caps are of type $B$, then one shall perform two surgeries on the corresponding necks (well away from $\Sigma$), thereby obtaining three pieces (a reflexive pair, and a central connected reflexive manifold) with the property that each of them can be isotoped to a rotationally symmetric Riemannian metric. Then, this scenario can be handled exactly as in the last paragraph of Case 2a below. 

\

\emph{Case 2: there exists  $z\in\Sigma$ that is the center of a reflexive $\e$-neck $N=N(z)$.}

\

This can be further divided into two subcases.

\

\emph{Case 2a: there exists $z\in \Sigma$ that is the center of a reflexive $\e-$neck $N$ containing $\Sigma$.}

\

Let $\varphi: N\to S^2\times (-1/\e,1/\e)$ be the parametrization of the neck in question (notice that the neck is necessarily of type $C$, and in particular $\Sigma=s_N^{-1}(0)$). We can perform surgery along the sphere $s_N^{-1}(8)$ as described in Section 5 of \cite{Mar12}: we glue two spherical caps (endowed with the \emph{standard initial metric}) on either side of such a slice, interpolating the metric $g$ with the $g_{std}$ on the domains $s_N^{-1}([4,8])$ and $s_N^{-1}([8,12])$.
Then we repeat the same procedure, in a perfectly symmetric fashion, along the sphere $s_N^{-1}(-8)$.

Thereby we get a new reflexive manifold $(\tilde{M},\tilde{g}, \tilde{f})$ that consists of three connected components: a  pair $(M_L, g_L), (M_R, g_R)$ (two isometric, non-reflexive pieces), and the reflexive component $(M_C, g_C)$, to be thought as the `central piece' of the manifold.

By Corollary 5.2 in \cite{Mar12} we know that $(M_L, g_L)$ has positive scalar curvature, hence by Corollary 1.1 therein we know the existence of an isotopy $((g_L)_{\mu})$ of metrics connecting it to a unit round sphere. We can then consider the isotopy defined by letting $(g_R)_\mu:=\tilde{f}^{\ast}(g_L)_{\mu}$ on $M_R$ (where it is understood that we are here considering the restricted map $\tilde{f}:M_R\to M_L$).

On the other hand, using the fact that $g$ is $\e$-close to being cylindrical in $N$, one can construct a reflexive isotopy $(g_C)_{\mu}$, relying on Lemma 5.3 in \cite{Mar12}, to connect $(M_C, g_C)$ to a rotationally symmetric metric $(g_C)_1$ on $M_C$.

At this stage, we can pick two $(g_C)_1$-antipodal points  $x_C, y_C=f(x_C)$ in $M_C$ (corresponding to the tips of the caps in the surgery process), together with the other tips $x_L\in M_L$ and $M_R\ni y_R=f(x_L)$ and consider (for any $\mu\in [0,1]$) the type D Gromov-Lawson connected sum $(M_L, (g_L)_{\mu})\# (M_C, (g_C)_{\mu})\# (M_R, (g_R)_{\mu})$; we convene that the balls where the metrics are modified are centered at $x_L, x_C$ and $y_C, y_R$ respectively. Thanks to Remark \ref{rmk:typeDreconstruction}, we know that $(M,g,f)$ can be connected, via a reflexive isotopy to $(M_L\# M_C\# M_R, g_L\#g_C\# g_R, f)$, while by Lemma \ref{lemma.equivariant.GL.construction}  we get that $(M_L\# M_C\# M_R, g_L\#g_C\# g_R, f)$ can be connected via a reflexive isotopy to $(M_L\# M_C\# M_R, (g_L)_1\# (g_C)_1\# (g_R)_1, f)$. On the other hand, this Riemannian manifold is locally conformally flat (see Remark \ref{remark.GL.connected.sum.locally.conformall.flat} and \ref{remark:GL.connected.sum.and.equivariant.Kuiper}), hence, thanks to Proposition \ref{pro:ConfDevMap}, isotopic (through reflexive triples) to a (spherical) model triple, which completes the proof.

So, to summarize, in this case the presence of a neck $N$ containing $\Sigma$ allows to `decouple' the problem, with the consequence that one can conclude the argument using the tools that come up in the non-equivariant setting.

\

\emph{Case 2b: for any $z\in \Sigma$ and any reflexive $\e-$neck $N=N(z)$ centered at $z$ one has $\partial N\cap \Sigma\neq\emptyset$  (thus this intersection shall necessarily consist, by Remark \ref{rmk:local modelsNeck} of two circles, the intersection being orthogonal in the neck chart).}

\

\emph{Claim: there exists $z\in\Sigma$ that is contained in the core $Y$ of a reflexive $(C,\e)$-cap $K$}.

We prove this assertion arguing by contradiction: we will now show that if that were not the case then $M$ would be diffeomorphic to $S^2\times S^1$, contrary to our assumption (we are working with a sphere). 
In fact, if the claim above were false then we would prove that any $x\in M$ is the center of an $\e$-neck (which is known, by Appendix A of \cite{MT07}, to imply that $M\simeq S^2\times S^1$).

Let us see why this is the case, so assume the claim in question does not hold. First, it must indeed be $M=\cup_{z\in \Sigma}N(z)$. If not, one could find $x_{\ast}\in M$ that is not contained in any reflexive $\e$-neck centered at some point of $\Sigma$. Let then $z_{\ast}\in\Sigma$ denote a point of $\Sigma$ of least distance from $x_{\ast}$ (this might not be unique, but this is irrelevant for the argument), and let $\Lambda$ be a minimizing gedesic segment from $x_{\ast}$ to $z_{\ast}$.

 Let $N_{\ast}:=N(z_{\ast})$ be a reflexive $\e$-neck centered at the point $z_{\ast}$, so that (in particular) $x_{\ast}\notin N_{\ast}$. Since each of the two boundary spheres of $N_{\ast}$ is separating (for $M\simeq S^3$) the curve $\Lambda$ must intersect (at least) one of the two spheres, say $\Gamma$, and let $y_{\ast}$ be the first intersection point (when moving from $x_{\ast}$ to $z_{\ast}$). If we then let $z'_{\ast}\in\Sigma\cap \Gamma$ be a point of least distance from $y_{\ast}$, a 
 standard distance-comparison argument shows that 
 \[
 d_g(x_{\ast},z'_{\ast})\leq d_g(x_{\ast}, y_{\ast})+d_g(y_{\ast}, z'_{\ast})<d_g(x_{\ast}, y_{\ast})+d_g(y_{\ast}, z_{\ast})=d_g(x_{\ast}, z_{\ast})
 \]
 for any $\e$ small enough (which we are always assuming). Such an inequality obviously contradicts the minimizing property of $z_{\ast}$.
 
 Thus any point of $M$ is covered by an $\e$-neck. But we need something even stronger, namely that any point of $M$ is actually the center of an $\e$-neck. Given $x_{\ast}\in M$ and said $z_{\ast}\in \Sigma$ a point of least distance, the very definition of $\e$-neck implies that in fact $|s_{N_{\ast}}(x_{\ast})-s_{N_{\ast}}(z_{\ast})|<\zeta\e$ for some constant $\zeta>0$ only depending on the ambient manifold.
 (Indeed: if $(\varphi^{-1})^{\ast}g$ were exactly cylindrical the point $\varphi(x_{\ast})$ would belong to the central sphere of the cylinder, and in general $(\varphi^{-1})^{\ast}g$ shall be $\e$-close to the cylindrical metric). Let us say, for notational simplicity, that $\zeta=1$. Then on the pre-image set $\tilde{N}:=\varphi^{-1}(S^2\times (-1/\e, 1/\e -2\e))$ one can define an $\e$-neck structure $\tilde{\varphi}: \tilde{N}\to S^2\times (-1/\e,1/\e)$ (the neck being now centered at $x_{\ast}$) by letting, in the coordinates $(x,t)$ cooresponding to $(N_{\ast},\varphi)$,
 \[
 \tilde{\varphi}(x,t)=(x, \ell(t)), \ \ \ell(t)=(1-\e^2)\left(t-\frac{\e}{1-\e^2} \right).
 \]
 Thereby our claim is justified.
 
\

So let $z\in\Sigma$ be contained in the core $Y$ of a reflexive $(C,\e)$-cap $K$.
We can proceed with the argument assuming that $\partial Y\cap\Sigma\neq\emptyset$ (else one argues as in Case 1 above). Now, there exists a reflexive $(C,\e)$-cap renamed $K$ such that no point on the sphere $\partial K^{0.9}=s^{-1}_N(0.9/\e)$ is contained in the core of a reflexive $(C,\e)$-cap that contains $K$. This is proven exactly as Claim 1 as at page 843 of \cite{Mar12} (the `competitors' being only the equivariant caps), so we omit the details.

One can then further obtain a structured chain of reflexive $\e$-necks $(N_1=N, N_2,\ldots, N_a)$ and a reflexive $(C,\e)$-cap $\tilde{K}$ such that the following properties are satisfied:
\begin{enumerate}
\item [\emph{a)}] {$(s^{-1}_{N_a}(0.9/\e)\cap\Sigma)\cap \tilde{Y}\neq\emptyset$;}	
\item [\emph{b)}] {$s^{-1}_{\tilde{N}}(0)\subset \overline{Y}\cup s_{N_1}^{-1}(-1/\e,0.9/\e)\cup\ldots\cup s_{N_a}^{-1}(-1/\e,0.9/\e) $}
\item [\emph{c)}]{$M=K\cup N_2\cup\ldots\cup N_a\cup \tilde{K}$.}
\end{enumerate}	 
Here we have adopted the usual notation for caps: $\tilde{K}=\tilde{Y}\cup\tilde{N}$, with $\tilde{Y}$ the core and $\tilde{N}$ the neck. We notice that the case $a=1$ is also contemplated, in which case $M=K\cup \tilde{K}$ is the union of two caps.
The chain is constructed iteratively: by definition of $K$, we know that every point in $s^{-1}_{N}(0.9/\e)$, hence in particular any point in $s^{-1}_{N}(0.9/\e)\cap\Sigma$ (we are intersecting with $Fix(f)$), is contained in either a reflexive $\e$-neck or in a reflexive $(C,\e)$-cap that does \emph{not} contain $K$. If there is some $s^{-1}_{K}(0.9/\e)\cap\Sigma$ for which the former alternative holds we take a reflexive $\e$-neck $N_2$ centered at that point, and check that this cannot intersect the left-most quarter of $N=N_1$ (the neck part of $K$) by distance-comparison. Observe that, as a result, $N_2\cap Y=\emptyset$ since $s_N^{-1}(-0.25/\e)$ is separating (for we are assuming $M\simeq S^3$).

 We then repeat the process considering $s^{-1}_{N_2}(0.9/\e)\cap\Sigma$ and so on. In the structured chain of necks one has that $s^{-1}_{N_i}((-0.25/\e,0.25/\e))\cap s^{-1}_{N_{i+1}}((-0.25/\e,0.25/\e))=\emptyset$ for any index $i$, so by volume comparison the process must terminate after finitely many steps. We thereby get a reflexive structured chain of $\e$-necks $(N_1=N, N_2,\ldots, N_a)$ such that any point $z\in s^{-1}_{N_a}(0.9/\e)\cap\Sigma$ is contained in the core of a reflexive $(C,\e)$-cap. 

So let us pick one point $\tilde{z}\in s^{-1}_{N_a}(0.9/\e)\cap\Sigma$, and let $\tilde{K}$ be a reflexive $(C,\e)$-cap to whose core it belongs. We claim that \emph{for such choice of $\tilde{K}$} properties b) and c) hold as well.
Indeed, set $Q=\overline{Y}\cup s_{N_1}^{-1}(-1/\e,0.9/\e)\cup\ldots\cup s_{N_a}^{-1}(-1/\e,0.9/\e)$ one has that:
\begin{itemize}
\item {by distance comparison $s^{-1}_{\tilde{N}}(0)\cap s^{-1}_{N_a}(0.9/\e)=\emptyset$;}
\item {by the previous step, and by the way we have defined $K$ one has that $s^{-1}_{\tilde{N}}(0)\subset Q$ (property b));}
\item {by the previous two steps, if we set $\tilde{K}^0$ to be the connected component of $\tilde{K}\setminus s^{-1}_{\tilde{N}}(0)$ containing $\tilde{Y}$ one has that $Q\cup \tilde{K}^0$ is open and closed in $M$ thus, by connectedness of $M$, we conclude that $M=Q\cup \tilde{K}^0\subset K\cup N_2\cup\ldots\cup N_a\cup \tilde{K}$, thus property c).}	
\end{itemize}	
Once such a distiguished cover of $M$ is obtained, one can follow the procedure presented in the non-equivariant case (pages 845-849 of \cite{Mar12}) with the following modifications:
\begin{itemize}
\item {reflexive necks (of type $T$) are combined into a single neck structure using Lemma \ref{lem:Combo2}, provided in Section \ref{subs:Necks};}
\item {reflexive surgery is performed as described in Section \ref{subs:Surgery};}
\item {locally conformally flat reflexive components are isotoped employing Proposition \ref{pro:ConfDevMap}.}
\end{itemize}	
We notice that in the case we are dealing with (\emph{Case 2b} in the ramification above) all connected components one obtains by means of surgery are reflexive. If $(M_A, g_A, f_A)$ and $(M_B, g_B, f_B)$ are two reflexive manifolds resulting from performing one surgery on $(M,g,f)$ and if $(g^A_{\mu})$ and $(g^{B}_{\mu})$ are reflexive isotopies connecting, respectively, such reflexive triples to (spherical) model triples, then one can invoke on the one hand Lemma \ref{lemma.GL.and.RFsurgery.are.isotopic} to assert that $(M_A \# M_B, g_A \# g_B, f)$ is reflexively isotopic to $(M,g,f)$, and on the other hand, Lemma \ref{lemma.equivariant.GL.construction} and Proposition \ref{pro:ConfDevMap} allow to conclude that $(M_A \# M_B, g_A \# g_B, f)$ is connected, via a reflexive isotopy, to a (spherical) model triple. Concatenating these two isotopies allows to conclude the proof. The case when more than one surgery is performed is handled following the very same scheme modulo notational changes.

	\end{proof}	
	
	\subsection{The induction argument, I: connected components}\label{subs:conncomp}
	Let $(M_i^3,g_i(t),f_i)$ be a reflexive Ricci flow with surgery as  in Theorem \ref{thm:RFS}, where the flow becomes extinct at a finite time $t_{j+1}$. For sufficiently small $\eta$, the manifold $(M_j,g_j(t_{j+1}-\eta),f_j)$ satisfies the canonical neighborhood assumption and hence it can be deformed, via an isotopy of classes going through reflexive metrics, to some, possibly disconnected, (spherical) model triple. By backwards induction on the surgery time $t_i$, we will prove that $(M_0,g_0,f_0)$ is isotopic, in the moduli space, to a (spherical) model triple.
	
	We start with the following lemma.
	
	\begin{lem}\label{lemma.equivariant.connected.sums.canonical}
		Assume $(M^3\simeq S^3,g,f)$ is a connected reflexive manifold. Suppose also that $(M^3,g,f)$ is obtained, by performing equivariant Gromov-Lawson connected sums, from $(S_1\sqcup \cdots\sqcup S_N, g_1\sqcup\cdots\sqcup g_N,f')$. If there is a reflexive isotopy of classes connecting $(S_1\sqcup \cdots\sqcup S_N, g_1\sqcup\cdots\sqcup g_N,f')$ to $(S_1\sqcup \cdots\sqcup S_N, g_{1,1}\sqcup\cdots\sqcup g_{N,1},f')$ where the latter is a (spherical) model triple,
	then there exists a reflexive isotopy of classes connecting $(M^3,g,f)$ to a (spherical) model triple.
	\end{lem}

Recall that asserting the existence of an isotopy of classes to a model triple is the same as stating the existence of an isotopy of metrics, ending at a model triple, provided we allow modifying the data $(M,g,f)$ via a diffeomorphism $\varphi\in\mathcal{D}$.
	
	\begin{proof}
		We argue by induction on $N$.
		
		For the base case, we verify the assertion for $N=1,2$. We need to discuss $N=1$ and $N=2$ for the base cases, as the induction step might possibly reduce the case $N$ to the case $N-2$. When $N=1$, there is nothing to prove. When $N=2$, let $(S_1,g_1)$, $(S_2,g_2)$ be the two connected components. We face an alternative: either both components in question are non-reflexive or they are both reflexive. In the former case, we invoke Lemma \ref{lemma.equivariant.GL.construction}, to perform a type $C$ equivariant Gromov-Lawson construction to the family of equivariant metrics $(S_1\#S_2, g_{1,\mu}\#g_{2,\mu}, f)$, and conclude that $(M,g,f)$, which we are assuming to be $(S_1\#S_2, g_1\#g_2,f)$, is isotopic to $(S_1\#S_2, g_{1,1}\#g_{2,1},f)$, where $g_{1,1}$ and $g_{2,1}$ are round metrics. By Remark \ref{remark.GL.connected.sum.locally.conformall.flat} and \ref{remark:GL.connected.sum.and.equivariant.Kuiper}, the connected sum $(S_1\#S_2, g_{1,1}\#g_{2,1},f)$ is locally conformally flat, hence there exists a reflexive isotopy to a (spherical) model triple. Thus, the conclusion comes by virtue of the discussion presented in Remark \ref{rmk:Lift}. In the second case, we also apply Lemma \ref{lemma.equivariant.GL.construction} to perform a type $T$ equivariant Gromov-Lawson construction. With the very same argument one can conclude the base case.
		
		Assume the assertion is true up to a positive integer $N-1$, where $N>2$. Since $M^3$ is diffeomorphic to $S^3$, and is obtained by taking equivariant connected sums of $S_1,\cdots,S_N$, there must exists a component, which without loss of generality may assumed to be $S_1$, that is only connected to one neck. In fact, if all $S_j$, $j=1,\cdots,N$, are connected with at least two necks, then there exists a loop $S_{j_1},\cdots,S_{j_k}$, such that each pair $(S_{j_i},S_{j_{i+1}})$, and $(S_{j_k},S_{j_1})$ are connected via a neck. This produces a non-contractible $S^1$ in $M$, contradicting the fact that $M\simeq S^3$. By assumption, one of the following two possibilities must happen to the component $S_1$.
		
		\emph{Case 1}: $S_1$ is non-reflexive in the disjoint union $(S_1\sqcup \cdots\sqcup S_N, g_1\sqcup\cdots\sqcup g_N,f')$. Without loss of generality, assume $f'(S_1)=S_2$. Since there is only one neck connecting $S_1$ in $(M,g,f)$, there can also be only one neck connecting $S_2$. Notice that since $N>2$, $S_1$ and $S_2$ are not connected via this neck. As a result, the same equivariant Gromov-Lawson construction, applied to $(S_3\sqcup \cdots \sqcup S_N, g_3\sqcup\cdots \sqcup g_N,f')$, produces a connected reflexive manifold $(M',g',f|_{M'})$. By induction, there is an isotopy $(M',g'_\mu, f|_{M'})$, $\mu\in [0,1]$, with $g'_0=g'$, $g'_1$ is the round metric. Now we consider $M$ obtained via a type D connected sum, and invoke Lemma \ref{lemma.equivariant.GL.construction} to conclude that $(M,g,f)$ is isotopic to $(S_1\#S_2\#M',g_{1,1}\#g_{2,1}\#g'_1,f)$, where $g_{1,1},g_{2,1},g'_1$ are all the round metrics. Once again, by Remark \ref{remark.GL.connected.sum.locally.conformall.flat} and \ref{remark:GL.connected.sum.and.equivariant.Kuiper}, we conclude relying on the fact that the metric $g_{1,1}\#g_{2,1}\#g'_1$ is locally conformally flat. 
		
		\emph{Case 2}: $S_1$ is reflexive in the disjoint union $(S_1\sqcup \cdots\sqcup S_N, g_1\sqcup\cdots\sqcup g_N,f')$. Since there is only one neck connecting $S_1$, there exists a connected reflexive triple $(M',g',f|_{M'})$ that can be realized by performing Gromov-Lawson connected sums from the disjoint union $(S_2\sqcup \cdots \sqcup S_N, g_2\sqcup\cdots \sqcup g_N,f')$. Thus, by induction, there exists a reflexive isotopy connecting such a triple to a round one.
		 At that stage, we employ Lemma \ref{lemma.equivariant.GL.construction} to perform a type $T$ Gromov-Lawson connected sum of $(M',g',f|_{M'})$ and $(S_1, g_1, f'|_{S_1})$: since both parts can be isotoped to the (spherical) model triple, by Remark \ref{remark.GL.connected.sum.locally.conformall.flat} and \ref{remark:GL.connected.sum.and.equivariant.Kuiper}, the resulting connected sum can be connected, via a reflexive isotopy of classes, to the (spherical) model triple. Possibly appealing, once again, to Remark \ref{rmk:Lift}, the assertion is proved.
	\end{proof}

	\subsection{The induction argument, II: backward in time}\label{subs:backw}
	Let us consider the reflexive Ricci flow with surgery. $(M_i,g_i(t),f_i)$, $0\le i\le j$, and let $\mA_i$ be the assertion:
	
	\vspace{1.5mm}
	
	\begin{center} 
		
		\emph{there is a reflexive isotopy of classes connecting
				$(M_i,g_i(t),f_i)$ to $(M_i, g_i, f_i)$, a (spherical) model triple.}
	\end{center}

		We also assume that:
	
	\begin{prop}
		$\mA_j$ holds.
	\end{prop}

Indeed, notice that such an assertion is just rephrasing Proposition \ref{pro:can}, which we proved in Section \ref{subs:basis}.
	Here we carry out the backward induction argument, and later conclude our main theorem.
	
	\begin{prop}\label{proposition.backward}
		If $i<j$ and $\mA_{i+1}$ holds, so does $\mA_i$.
	\end{prop}
	
	\begin{proof}
		To start, let us recall some basic notions involved in the equivariant Ricci flow with surgery. Let, as above, $\{r_i\}_{i=0}^j$ be the canonical neighborhood parameters, $\{\de_i\}_{i=0}^j$ be the surgery control parameters and $\rho_i=\delta_i r_i$. We set $\Omega(t_{i+1}):=\{x\in M_i: \lim\inf_{t\rightarrow t_{i+1}} R_{g(t)}(x)<+\infty\}$, $\Omega_{\rho_i} (t_{i+1}):=\{x\in \Omega(t_{i+1}): R_{g_i(t_{i+1})}(x)\leq \rho_i^{-2}\}$, and define $\Omega^{big}(t_{i+1})$ as the union of the finitely many components of $\Omega(t_{i+1})$ that intersect $\Omega_{\rho_i}(t_{i+1})$. Also recall that $\Omega^{\text{big}}(t_{i+1})$ contains a finite disjoint collection of open sets, called $2\e$-horns, each diffeomorphic to $S^2\times [0,1)$, denoted by $\mH_1,\cdots,\mH_l$. The $\de_i$-neck in each $\mH_k$ is denoted by $N_k$, and centered at $y_k$ with $R_{g_i(t_{i+1})}(y_k)=h^{-2}(r_i,\delta_i)$. We denote the central sphere through $y_k$ by $S_k$. Let $\mH_k^{+}$ the component, separated by $S_k$, that contains the end of the horn. We also denote $N_k^+$ the halves of $N_k$ that are contained $\mH_k^{+}$, and $N_k^{-}$ the other halves. Finally, denote $C_{t_{i+1}}$ the continuing region, defined as the complement of $\sqcup_{k=1}^l \mH_k^+$ in $\Omega^{big}(t_{i+1})$. In particular, we do not change the region $C_{t_{i+1}}$ in the surgery.
		
		Let us now briefly describe the surgery process in the form proposed by Perelman to continue the Ricci flow beyond singularities. The region $C_{t_{i+1}}$ has boundary $\sqcup_{k=1}^l S_k$, and $g_i(t_{i+1})$ is a sufficiently small perturbation of the standard cylindrical metric in $N_k$. We perform surgery along $S_k$, remove $N_k^{+}$, add a spherical cap $B_k$, and obtain $M_{i+1}=C_{t_{i+1}}\cup (\sqcup_{k=1}^l B_k)$. By the topological classification of singularities, it is known that the complement of $C_{t_{i+1}}$ in $M_i$ is a disjoint union of regions diffeomorphic to $S^2\times (0,1)$. On these regions, we perform the surgery along $(M_i\setminus C_{t_{i+1}})\cup (\sqcup_{k=1}^l N_k^{-})$, remove the neck regions $N_k^{-}$, add spherical caps, and obtain finitely many closed three spheres $(P_k, g_{P_k})$, $k=1,\cdots,m$. As a result, we can recover the manifold $M_i$ by performing the connected sum of $M_{i+1}$ with $\{P_k\}_{k=1}^m$.

		Consider a time $t'\in (t_i,t_{i+1})$ sufficiently close to $t_{i+1}$, such that $R_{g_i}(x,t')>\rho_{i}^{-2}/2$ for any $x\in M_i\setminus C_{i+1}$. The isometric involution $f_i$ on $(M_i,g_i(t'))$ induces the involution $f_{i+1}$ on $M_{i+1}$, and an involution $f'_{i+1}$ on $(\sqcup P_k,\sqcup g_{P_k})$, such that $Fix(f_i)$, $Fix(f_{i+1})$ and $Fix(f')$ are all equal, in their common regions of definition. Let $\tilde{g}_{t'}$ be the metric on $M_{i+1}$, obtained after gluing spherical caps to $(C_{t_{i+1}},g_i(t'))$. Then $f_{i+1}$ is also an isometric involution on $(M_{i+1},\tilde{g}_{t'})$. By Lemma \ref{lem:Ext} (allowing to extend isotopies from necks to caps) and its non-equivariant counterpart, $(M_{i+1},\tilde{g}_{t'},f_{i+1})$ and $(M_{i+1},g_{i+1}(t_{i+1}),f_{i+1})$ are isotopic. By the induction hypothesis $\mA_{i+1}$, it then follows that $(M_{i+1},\tilde{g}_{t'},f_{i+1})$ can be connected, through a reflexive isotopy (through positive scalar curvature metrics) to a (spherical) model triple.
		
		By the choice of $\Omega_{\rho_i}(t_{i+1})$, each point $x\in M_i\setminus C_{i+1}$ satisfies $R_{g_i(t')}(x)>\rho_i^{-2}/2>r_i^{-2}$. As a result, every such $x$ has a $(C,\e)$ canonical neighborhood at time $t'$. In particular, the regions $P_k$, $k=1,\cdots,m$, are covered by canonical neighborhoods. By Proposition \ref{pro:can} (possibly in combination with Remark \ref{rmk:Lift}), the manifold $(\sqcup P_k,g_{P_k},f'_{i+1})$ can also be connected, through a reflexive isotopy (through positive scalar curvature metrics) to a (spherical) model triple.

		By Lemma \ref{lemma.equivariant.connected.sums.canonical}, the equivariant Gromov-Lawson connected sum, obtained from \[\left(M_{i+1}\sqcup(\sqcup_{k=1}^m P_k),\tilde{g}_{t'}\sqcup(\sqcup_{k=1}^m g_{P_k}) ,f_{i+1}\sqcup f'_{i+1}\right)\]
		by suitably reconstructing the necks where surgery has been performed, can be connected
		through a reflexive isotopy (of positive scalar curvature metrics) to a (spherical) model triple.
		 On the other hand, note that $\left(M_{i+1}\sqcup(\sqcup_{k=1}^m P_k),\tilde{g}_{t'}\sqcup(\sqcup_{k=1}^m g_{P_k}) ,f_{i+1}\sqcup f_i'\right)$ is also obtained from $(M_i,g_i(t'),f_i)$ by performing surgery. On each neck surgery region parameterized by $S^2\times (-4,4)$, the surgery metric
		\[(S^{-}\#S^{+},(g_i(t'))_{surg,\de}^{-}\#(g_i(t'))_{surg,\de}^+,f_i)\]
		is isotopic to the corresponding equivariant Gromov-Lawson construction, by Lemma \ref{lemma.GL.and.RFsurgery.are.isotopic}. We therefore conclude that $(M_i,g_i(t'),f_i)$ is isotopic, through reflexive triples of positive scalare curvature, to a (spherical) model triple.
		Finally, we may isotope $(M_i,g_i(t_i),f_i)$ to $(M_i,g_i(t'),f_i)$ through smooth equivariant Ricci flow. Thereby, the assertion $\mA_i$ is proved.
	\end{proof}

    \subsection{Description of the model metrics}\label{subs:ModelMet}
    We now wish to extend the previous discussion to general 3-manifolds $X^3$ for which $\mathcal{M}\neq\emptyset$, see Theorem \ref{thm:A}). To that scope, the first (fundamental) step is to define the model triples that will play the role of natural endpoints for the isotopies we construct.
    
    Let $(M_1,g_1)$ and $(M_2, g_2)$ be \emph{possibly disconnected} 3-manifolds and assume that their disjoint union, denoted by $(M, g, f)$ is a reflexive 3-manifold. Notice that we also want to consider the degenerate case when there is only one reflexive triple that serves as input for the construction, as explained above.
  Recall that there are three different types of Gromov-Lawson operations, between $M_1$ and $M_2$, we need to consider:
    \begin{enumerate}
    \item [\emph{i)}] {a type $T$ connected sum, that is obtained by joining two distinct points on $Fix(f)$, $p_1\in M_1, p_2\in M_2$ by means of an equivariant neck (whose central portion is) of type $T$;}
    \item [\emph{ii)}]  {a type $C$ connected sum, that is obtained by connecting two distinct $f$-related points, $p_1\in M_1, p_2=f(p_1)\in M_2$ via an equivariant neck (whose central portion is) of type $C$;}
    \item [\emph{iii)}]  {a type $D$ double connected sum, that is obtained by joining two distinct points, away from $Fix(f)$, $p_1\in M_1, p_2\in M_2$ via a non-equivariant neck, and then performing the very same construction on the couple of $f$-related points namely $f(p_1)\in M_1, f(p_2)\in M_2$.}
     \end{enumerate}	
 We notice that not all of the three operations may in general be allowed, for instance  a type $T$ connected sum is only possible if both $(M_1,g_1,f_1)$ and $(M_2, g_2, f_2)$ have at least one reflexive connected component.
 
 Now, given $X^3$ such that $\mathcal{M}\neq\emptyset$, and using the same notation as in Theorem \ref{theorem.topological.classification.of.3-mfld.with.boundary}, we can describe the family of \emph{model metrics} on $M=DX$.
 
 Let us first discuss the simpler case when $a=0$, namely when
 \[
 X\simeq S^3\# S^3/\Gamma_1\#\cdots\# S^3/\Gamma_b\# \left(\#_{i=1}^c (S^2\times S^1)\right)\setminus (\sqcup_{i=1}^d B_i^3)
 \]
 In this case, a model metric is any metric that is obtained as follows:
 \begin{enumerate}
 	\item {we consider a standard round reflexive 3-manifold, which we will name \emph{reference sphere} (we agree that the reference sphere is normalized so to have unit radius);}
 	\item {we perform exactly $d-1$ Gromov-Lawson connected sums of type $C$, each connecting a pair of $f$-related points in the reference sphere;}
 	\item {we perform, for any $i=1,\ldots, b$, a type $D$ Gromov-Lawson connected sum of the reference sphere with a reflexive pair consisting of the disjoint union of two copies of $S^3/\Gamma_i$ picking a basepoint in the upper (open) hemisphere of the reference sphere, and then the corresponding point in the lower (open) hemisphere;}
 	\item {we perform exactly $c$ Gromov-Lawson type $D$ connected sums of the reference sphere with itself picking $c$ pairs of basepoints in the upper (open) hemisphere, and then $c$ pairs of corresponding points in the lower (open) hemisphere (thereby attaching a total of $2c$ handles to the reference sphere).}
 \end{enumerate}

 In general, let
  \[X\simeq P_{\g_1}\#\cdots\# P_{\g_a}\# S^3/\Gamma_1\#\cdots\# S^3/\Gamma_b\# \left(\#_{i=1}^c (S^2\times S^1)\right)\setminus (\sqcup_{i=1}^d B_i^3).\]
  
  \begin{rmk}\label{rem:Ordering}
     Without loss of generality, we can (and we shall) assume that in the equation above $\gamma_1\leq\gamma_2\leq\ldots\leq \gamma_a$. This convention, although unnecessary, is helpful in avoiding any possible ambiguity in the construction described below; if we worked with unordered tuples we would end up with possibly different, yet anyway isotopic model metrics.
  \end{rmk}
  
  We first need to describe what we mean by \emph{genus $\g$ handle of mixed type $CT$}: if we are given a connected reflexive 3-manifold $(M,g,f)$ we consider a type $C$ Gromov-Lawson connected sum and then we further perform, along the central sphere of the neck (in particular: in the region where the metric is exactly cylindrical) exactly $\g\geq 1$ type $T$ connected sums, with endpoints on such central sphere. Of course, one needs to work at two different scales (which are fixed once and for all) in order for this operation to be well-defined.

  \begin{figure}[htbp]
  	\centering
  	\includegraphics[scale=0.95]{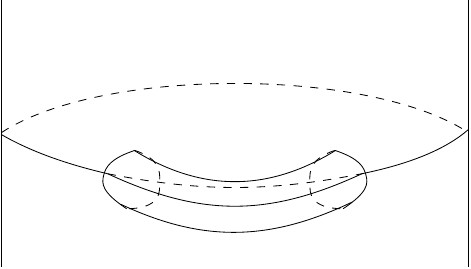}
  	\caption{An equivariant neck of mixed type $CT$.}
  	\label{fig:NeckTypeCT}
  \end{figure}
  
  That being said, a model metric is any metric on $M$ that is obtained by following this procedure: 
  \begin{enumerate}
  	\item {we consider a standard round reflexive 3-manifold, which we will refer to as \emph{reference sphere};}
  	\item {we perform exactly $\g_1$ Gromov-Lawson connected sums of type $T$, connecting pairs of points on the equator of the reference sphere;}
  	\item {we perform exactly $d$ Gromov-Lawson connected sums of type $C$, each connecting a pair of $f$-related points in the reference sphere;}
  	\item {we perform exactly $a-1$ Gromov-Lawson connected sums of mixed type $CT$, each connecting a pair of $f$-related points in the reference sphere, and having genera $\g_2, \g_3, \ldots, \g_a$ respectively;}
  	\item {we perform, for any $i=1,\ldots, b$, a type $D$ Gromov-Lawson connected sum of the reference sphere with a reflexive pair consisting of the disjoint union of two copies of $S^3/\Gamma_i$ picking a basepoint in the upper (open) hemisphere of the reference sphere, and then the corresponding point in the lower (open) hemisphere;}
  	\item {we perform exactly $c$ Gromov-Lawson type $D$ connected sums of the reference sphere with itself picking $c$ pairs of basepoints in the upper (open) hemisphere, and then $c$ pairs of corresponding points in the lower (open) hemisphere (thereby attaching a total of $2c$ handles to the reference sphere).}
  \end{enumerate}

\begin{figure}[htbp]
	\centering
	\includegraphics[width=\textwidth]{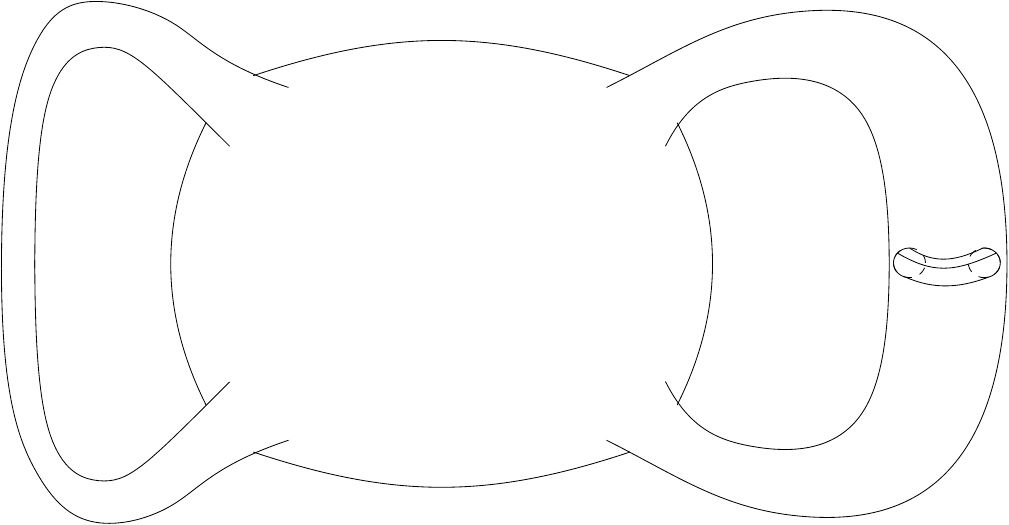}
	\caption{A model metric.}
	\label{fig:ModMetric}
\end{figure}

 \begin{defi}
 	\label{def:ModTriple}
 We shall say that a reflexive triple $(M,g,f)$ is a model triple if each of its connected component is endowed, when restricting $g$, with a model metric as defined above.	
 \end{defi}

We recall from Section \ref{sec:EquivMan}, Corollary \ref{cor:UniqueF}, that if $(M,g,f)$ is a reflexive triple then $f$ is uniquely determined by $g$ so the involution does not add additional degrees of freedom in the parametrization of model triples.

    \subsection{Handling the general case}\label{subs:PathGen}
    
    Here we wish to first prove the following assertion, pertaining those manifolds (of positive scalar curvature) covered by canonical neighborhoods.
    
    	\begin{prop}\label{pro:canGEN}
    	Let $(M^3,g,f)$ be a \emph{connected} reflexive 3-manifold having positive scalar curvature, and being covered by $(C,\e)$ reflexive canonical neighborhoods. Then $M^3$ is diffeomorphic to one of the three manifolds $S^3, S^2\times S^1, \mathbb{RP}^3\# \mathbb{RP}^3$ and it can be connected, via a smooth reflexive isotopy of classes, through metrics of positive scalar curvature to a model triple.
    \end{prop}

\begin{proof}
First of all, if the manifold in question is a $C$-component or an $\e$-round component then it follows from Hamilton's work (cf. Remark \ref{rmk:Lift}) that there exists a reflexive isotopy with endpoint a model triple corresponding to a spherical space form $S^3/\Gamma$ (with round metric). However, by Theorem \ref{theorem.topological.classification.of.3-mfld.with.boundary} this is only possible if $\Gamma$ is the trivial group (all non-trivial spherical space forms must come in pairs in $M=DX$).

At this stage, we further know (cf. \cite{MT07}) that any 3-manifold having positive scalar curvature, and being entirely covered by  canonical neighborhoods that are either $\e$-necks or $(C,\e)$-caps, is diffeomorphic to one of the following four: $S^3, S^2\times S^1, \mathbb{RP}^3, \mathbb{RP}^3\# \mathbb{RP}^3$. If we further assume that the 3-manifold in question is reflexive (and connected) then, again by Theorem \ref{theorem.topological.classification.of.3-mfld.with.boundary}, the only possible options are $S^3, S^2\times S^1, \mathbb{RP}^3\# \mathbb{RP}^3$. We have already discussed about $S^3$ (see Proposition \ref{pro:can}), so let us consider the other two cases. 

We know, by our topological results, that $M=\mathbb{RP}^3\# \mathbb{RP}^3$
is only possible when $X$ is obtained by removing a ball from $\mathbb{RP}^3$ (that is to say when $a=0, b=1, c=0, d=1$ and the only spherical space quotient is determined by the group $\mathbb{Z}/2\mathbb{Z}$). Observe that $Fix(f)\subset M$ cannot be entirely covered by the core of $(C,\e)$-caps of type $A$, for in this case it would have positive sectional curvature and thus \cite{Ham82} would force $M$ to be a spherical quotient, a contradiction. If instead $Fix(f)$ were covered by the core of $(C,\e)$-caps of type $C$, then one would perform two surgeries (distant from $Fix(f)$), thereby obtaining three locally conformally flat pieces, so a reflexive pair of $\mathbb{R}\mathbb{P}^3$ and an $S^3$. Each piece can be conformally deformed to have a round metric, so the conclusion follows the conceptual scheme we have already described. Those cases being dealt with, one can follow, with minor modifications, the discussion of cases \emph{2a} and \emph{2b} we presented for $M\simeq S^3$ (Section \ref{subs:basis}) replacing type B caps by means of type C caps, both in the covering argument and in the subsequent construction of the isotopy (for the relevant definitions see \cite{Mar12}, Section 7). Actually, we note that \emph{Case 2b} can be ruled out, when $M=\mathbb{RP}^3\# \mathbb{RP}^3$, because of the separating property we always postulate for $Fix(f)$.

Concerning $M\simeq S^2\times S^1$ we have that either $X\simeq D^2\times S^1$ (in other words: $X$ is a genus one handlebody) or $X\simeq S^2\times I$ (in other words: a sphere minus two balls). Recall from Proposition A.21 in \cite{MT07} that, either way, when $M\simeq S^2\times S^1$ is covered by canonical neighborhoods then any point of $M$ actually lies in the central sphere of some $\e$-neck. In particular, by \cite{DL09} (cf. our discussion in Section \ref{subs:Necks}), any point in $Fix(f)$ lies on the central sphere of some reflexive $\e$-neck. Therefore, if $X\simeq D^2\times S^1$ we have $Fix(f)\simeq S^1\times S^1$ hence any point of $Fix(f)$ is the center of a type $T$ reflexive $\e$-neck. If we fix one such neck and perform surgery along the corresponding central sphere, then we obtain a reflexive triple that is topologically a 3-sphere, hence by Proposition \ref{pro:can} there exists a reflexive isotopy to a model triple that is just a round sphere. Hence, the conclusion comes by unwinding the previous surgery operation, namely by invoking Lemma \ref{lemma.GL.and.RFsurgery.are.isotopic} followed by Lemma \ref{lemma.equivariant.GL.construction}. 
Instead, if  $X\simeq S^2\times I$ we have that $Fix(f)$ is the disjoint union of two spheres, hence we can find a type $C$ reflexive $\e$-neck whose central sphere coincides with one connected component of $Fix(f)$. If we then perform surgery along such a sphere, we can easily conclude as above.
\end{proof}

    We now discuss how to modify the induction arguments presented when dealing with the case of a general manifold $X^3$.
    
    	\begin{lem}\label{lemma.equivariant.connected.sums.canonical.general}
    	Assume $(M^3,g,f)$ is a connected reflexive manifold. Suppose also that $(M^3,g,f)$ is obtained, by performing equivariant Gromov-Lawson connected sums, from $(S_1\sqcup \cdots\sqcup S_N, g_1\sqcup\cdots\sqcup g_N,f')$. If there is a reflexive isotopy of classes connecting $(S_1\sqcup \cdots\sqcup S_N, g_1\sqcup\cdots\sqcup g_N,f')$ to $(S_1\sqcup \cdots\sqcup S_N, g_{1,1}\sqcup\cdots\sqcup g_{N,1},f')$ where the latter is a model triple,
    	then there exists a reflexive isotopy of classes connecting $(M^3,g,f)$ to a model triple.
    \end{lem}

Before presenting the proof, let us discuss an enlightening special case and then add two useful remarks which will allow for a more direct argument.

\begin{ex}
	\label{example:Inflate}
	Let us consider the following toy problem. Let $(M,g,f)$ denote a pair of round spheres connected by means of a type $D$ Gromov-Lawson double connected sum (so by \emph{two} necks). 
	We claim that this triple can be connected, via a reflexive isotopy, to 
	a  round sphere with one handle attached (as a result of a type $C$ Gromov-Lawson connected sum), denoted by $(M,\tilde{g},f)$.
	The way the isotopy is constructed is pictorially described in Figure \ref{fig:InflateSphere} below. One can first perform surgery, in a $f$-equivariant fashion, on the central spheres of the two necks of $M$ thereby obtaining two connected reflexive 3-manifolds denoted by $(M_1\simeq S^3, g_1, f_1)$ and $(M_2\simeq S^3, g_2, f_2)$. We observe that $(M_2, g_2, f_2)$ is locally conformally flat, hence can be isotoped to a (compact) round cylinder $(M_2,g_{2,1}, f_2)$ capped off by performing surgery along the two boundary spheres (this follows by the fact that both triples can be equivariantly isotoped to a round sphere using Kuiper's developing map). Now, it follows by applying the reconstruction lemma (cf. Lemma \ref{lemma.GL.and.RFsurgery.are.isotopic}, and Remark \ref{rmk:typeDreconstruction}) that $(M,g,f)$ can be connected via a reflexive isotopy to the type $D$ connected sum $(M_1\simeq S^3, g_1, f_1)\# (M_2\simeq S^3, g_2, f_2)$; on the other hand the same lemma implies that the type $D$ connected sum $(M_1\simeq S^3, g_1, f_1)\# (M_2\simeq S^3, g_{2,1}, f_2)$ is isotopic to the triple $(M, \tilde{g}, f)$. We conclude by simply concatenating the two isotopies.
\end{ex}

\begin{figure}[htbp]
	\centering
	\includegraphics[width=\textwidth]{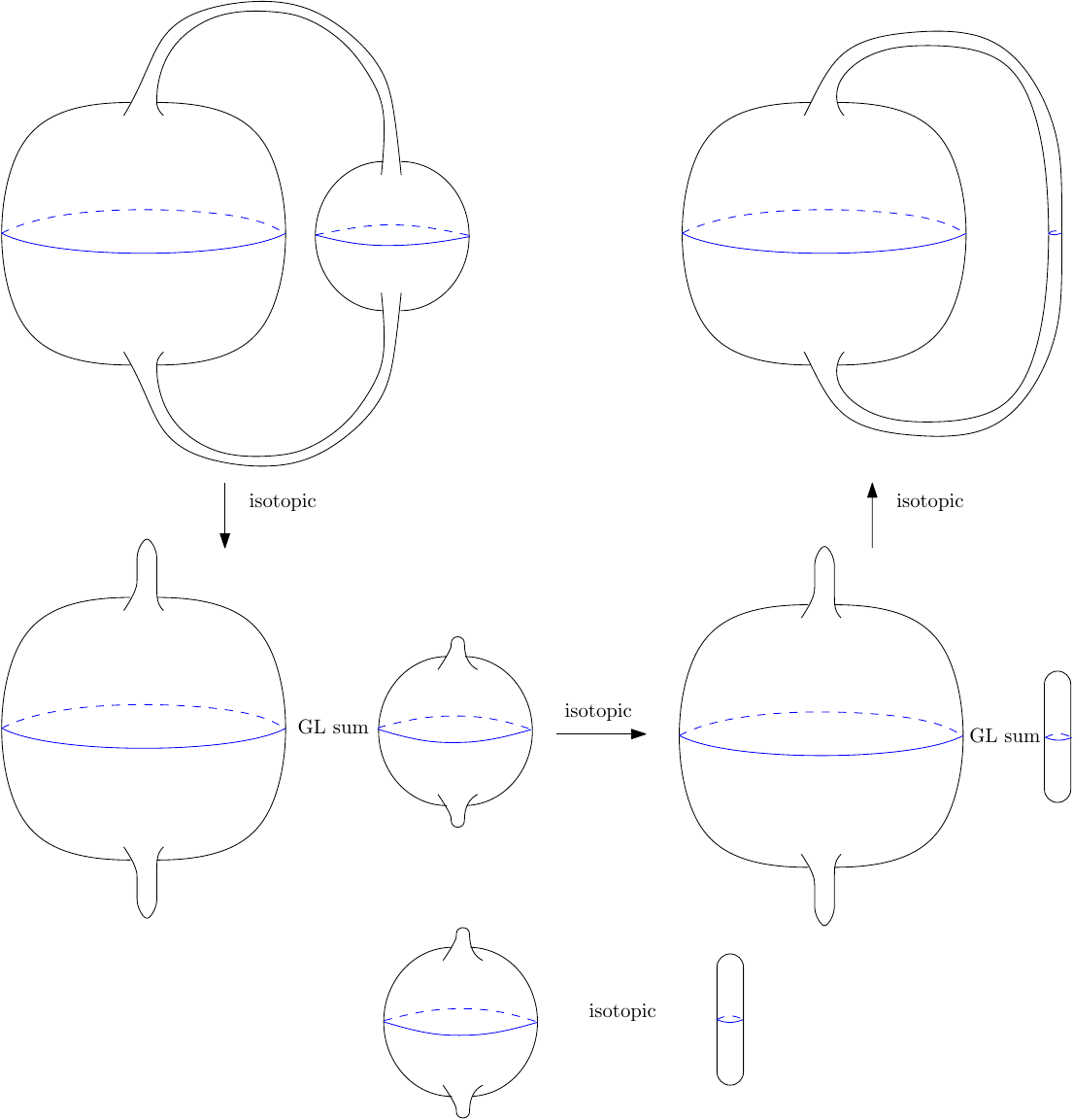}
	\caption{A standard isotopy: inflating/deflating a sphere.}
	\label{fig:InflateSphere}
\end{figure}

\begin{rmk}
	\label{rmk: UnaffArea}
	With respect to the statement of Lemma \ref{lemma.equivariant.connected.sums.canonical.general}, consider the case $N=2$, that is to say a Gromov-Lawson equivariant connected sum of two model triples. Now, let us denote by $S'_i$ the subset of the reference sphere of $S_i$ that is unaffected by the connected sums defining $(S_i, g_i, f')$, so the region of $S_i$ consisting of those points having (at least) a positive, preassigned, distance from the boundary of each ball where the Gromov-Lawson constructions are performed. Furthermore, for $i=1,2$ let $N'_i\subset N_i$ be the set of points on the union $N_i$ of type $C$ necks of $S_i$ that are unaffected by potential, subsequent type $T$ connected sums determining necks of mixed type $CT$. 
	Then, let us consider a Gromov-Lawson equivariant connected sum of $(M_1, g_1, f_1)$ and $(M_2, g_2, f_2)$. We observe that by suitably moving the basepoints for the construction, and pulling back the resulting metrics via diffeomorphisms, we can always assume that any basepoint $p_1$ belongs to $S'_1 \cup N'_1$, and similarly any basepoint $p_2$ belongs to $S'_2\cup N'_2$. This relies on Proposition \ref{lemma.equivariant.GL.construction}.
\end{rmk}

\begin{rmk}
	\label{rmk: Inflate}
The operation of moving basepoints, however, does not allow to `resolve' the following situation. With the same notation as above, we might have a type $T$ connected sum of the two model triples, with basepoints located (either in one or in both cases) on the unaffected area of type $C$ necks, namely on $N'_1\subset N_1$ and $N'_2\subset N_2$. We wish to prove that such a reflexive 3-manifold, denoted by $(M,g,f)$ and which is \emph{not} a model triple can be connected, via a reflexive isotopy, to a model triple. The idea is that, following the argument presented in Example \ref{example:Inflate} above, we can first construct a reflexive isotopy that \emph{inflates} the necks $N_1, N_2$ to round spheres, and then (using Kuiper's develeoping map, cf. Proposition \ref{pro:ConfDevMap}) we can isotope the type $T$ connected sums of those two spheres (formerly $N_1, N_2$) into a single round sphere. By repeating this operation finitely many times, one can then always reduce to the case when a Gromov-Lawson equivariant connected sum of two model triples has basepoints in the unaffected region of the reference spheres (what we denoted by $S'_i\subset S_i$).
\end{rmk}

\begin{figure}[htbp]
	\centering
	\includegraphics[width=\textwidth]{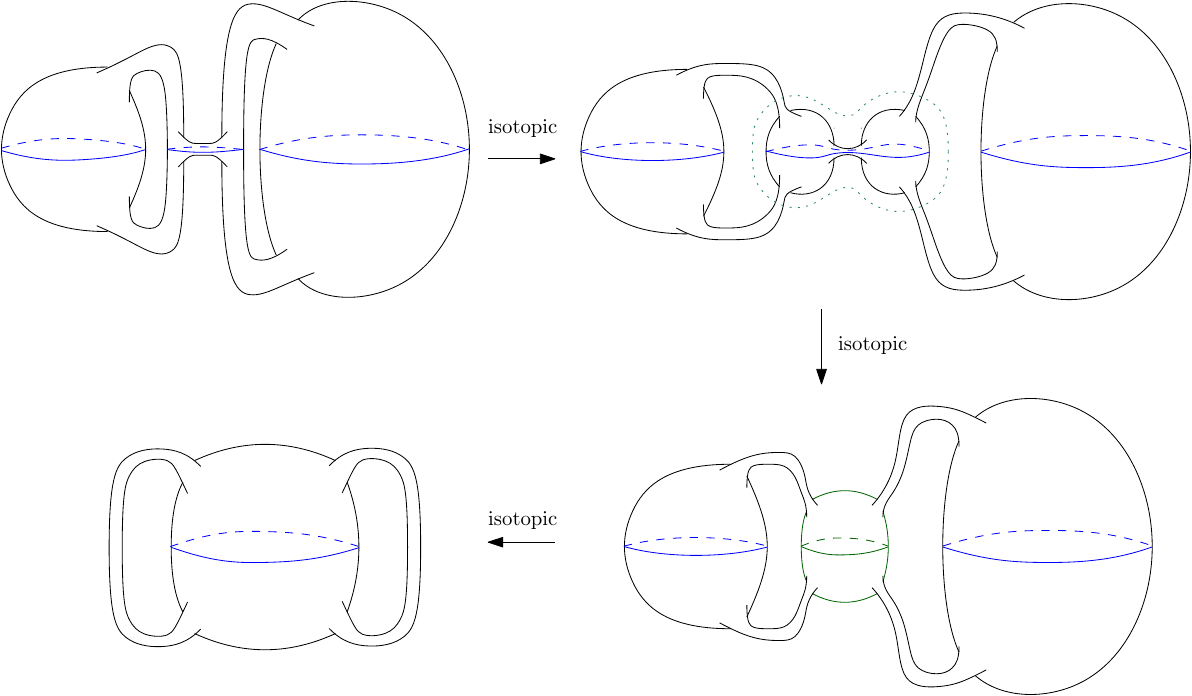}
	\caption{Reduction to model metrics.}
	\label{fig:Simplify}
\end{figure}

We now build on these observations to prove Lemma \ref{lemma.equivariant.connected.sums.canonical.general}.

\begin{proof}
	
We let $k\in\mathbb{N}$	denote the number of equivariant Gromov-Lawson operations connecting any possible pair of \emph{distinct} reflexive 3-manifolds in our connected sum, namely $(S_i, g_i, f_i)$ to $(S_j, g_j, f_j)$ for $i,j\in\left\{1,\ldots, N,\right\}$ and $i\neq j$. (We notice that with our conventions $k=k_T+k_C+k_D/ 2$ where $k_T, k_C, k_D$ denote the numbers of connecting necks of type $T, C, D$ respectively).

Thanks to Remark \ref{rmk: UnaffArea} and \ref{rmk: Inflate} above, we can assume, without loss of generality, that the basepoints for each connecting neck (of any type) are taken in the unaffected domain $S'_i$ of the reference sphere of $S_i$, for any $i\in\left\{1,\ldots, N\right\}$.

Now we proceed by induction on $k\geq 0$ and let $\mathcal{P}(k)$ denote the proposition we wish to prove. If $k=0$ there is nothing to prove, for there are no connected sums performed and thus necessarily $N=1$ so we already have a model triple. We further need to discuss, as base case, $k=1$ since the inductive scheme might employ the assertion $\mathcal{P}(k-2)$ to obtain $\mathcal{P}(k)$. So, when $k=1$ there is only one connected sum. Thus necessarily $N=2$. If the Gromov-Lawson connected sum is of type $T$ or of type $C$, we simply observe that the two reference spheres of $S_1$ and $S_2$, together with that connecting neck, can be isotoped to a round sphere (again using the developing map, see Proposition \ref{pro:ConfDevMap}, so that the conclusion comes by simply invoking Lemma \ref{lemma.equivariant.GL.construction}). If instead the equivariant Gromov-Lawson connected sum in question is of type $D$, we need more care.
Since $(M,g,f)$ is assumed to be connected, we can assume (without loss of generality) that $(S_1, g_1, f'_{|S_1})$ is a \emph{connected} reflexive manifold. There are two cases: either $(S_2, g_2, f'_{|S_2})$ is also a connected reflexive manifold, or it shall instead consist of exactly two connected components. Let us start discussing the first case.
 To fix the notation, say that $S_2$ is obtained by attaching to the reference sphere the other pieces following the procedure described in Section \ref{subs:ModelMet}.
So, we consider all connected sums that are performed to produce $S_2$ starting from its reference sphere only and transfer all pieces, \emph{except the necks associated to type $T$ connected sums}, to $S_1$ moving the basepoints, in an $f$-equivariant way, through the two necks connecting $S_1$ and $S_2$. This operation is well-defined by virtue of Lemma \ref{lemma.equivariant.GL.construction}. Thus, we have reduced to the case when $S_2$ is just a reference sphere with finitely many type $T$ necks attached along its equator. If we \emph{deflate} it, reverting the isotopy constructed in Example \ref{example:Inflate} we obtain a mixed neck to type $CT$ (possibly just a neck of type $C$) and we have that $(M,g,f)$ can be connected, via a reflexive isotopy, to a model triple.  
In the case when $(S_2, g_2, f'_{|S_2})$ is a disconnected reflexive manifold, we can employ Lemma \ref{lemma.equivariant.GL.construction} to isotope, through metrics of positive scalar curvature, $(M,g,f)$ to the type $D$ connected sum of a (connected) model triple and two round spheres: using the developing map we can then merge the reference sphere of the model triple and such two spheres, through an isotopy, to a single sphere; it follows, again by Lemma \ref{lemma.equivariant.GL.construction}, that $(M,g,f)$ ca  be isotoped to a model triple.
Thereby we have completed the discussion of the base of the induction and can move to the inductive step.

So, let us assume the implication is true for $k=0,1,\ldots, a-1$ and let us prove it for $k=a\geq 2$. If $(M,g,f)$ is obtained by performing \emph{at least one connected sum of type $T$ or type $C$} then we argue in the same fashion, as we shall now indicate. If the central 2-sphere of the neck in question is not separating then by the induction hypothesis and Lemma \ref{lemma.equivariant.GL.construction} we have that $(M,g,f)$ is isotopic to a model triple with one extra equivariant neck attached, corresponding to either a type $T$ connected sum or a type $C$ connected sum: either way, by the very definition of model metrics, we have that $(M,g,f)$ is itself a model triple. If instead the central 2-sphere of the neck in question is separating, then we have that, by neglecting the neck in question, $M$ gives rise to two equivariant triples, to each of which the inductive hypothesis applies: hence $(M,g,f)$ is connected, via a reflexive isotopy, to the type $T$ or type $C$ connected sum of two model triples, at which stage one can conclude arguing as we did above for the $k=1$ case. Finally, we need to consider the case when $(M,g,f)$ is obtained from the building blocks by only performing equivariant connected sums of type $D$. Here we have two subcases. If there exists a pair indices $i\neq j$ such that $S_i, S_j$ are connected by means of at least two type $D$ connected sums (i.~e. at least four necks) we can use the inductive hypothesis to equivariantly isotope $(M,g,f)$ to a model triple with two $f$-associated handles attached, so again a model triple.

If instead any pair $S_i, S_j$ is connected by exactly one pair of necks then we consider $S_1, S_2$, and assume (possibly renaming, without loss of generality) that $(S_1, g_1, f'_{|S_1})$ is a connected reflexive triple. We then
 repeat the argument for the type $D$ connected sum in the $k=1$ analysis and isotope that pair to a model triple; thus, we can conclude thanks to the inductive hypothesis, applied for $k-1$.
\end{proof}	
    
  The second inductive argument, presented in Section \ref{subs:backw}, can be transplanted (with purely notational changes) to the case of general $X^3$. Hence, Proposition \ref{pro:canGEN} and Lemma \ref{lemma.equivariant.connected.sums.canonical.general} allow to complete the proof of Proposition \ref{pro:Conclusive}.

 \subsection{Conclusions: path-connectedness theorems}\label{subs:pcmet}

We will now exploit the equivariant isotopy obtained above to complete the proof of Theorem \ref{thm:B}. Given $X^3$ as in that statement, we let $\mathcal{O}$ be the associated subset of model metrics. We start with two remarks. First, observe that if $X$ can be written as a connected sum of pieces as in the statement of Theorem \ref{thm:A} in two different ways, i.~e. if one has an orientation-preserving diffeomorphism
\begin{equation}\label{eq:CompDiffType}
X_1\#\ldots \# X_k \simeq X'_1\#\ldots \# X'_\ell
\end{equation}
then in fact $k=\ell$ and, possibly by re-labeling the summands, there are orientation-preserving diffeomorphisms $X_i\simeq X'_i$ for $i=1,2,\ldots k$. This assertion follows as a special case of the uniqueness of prime decomposition for compact manifolds with boundary as given in Theorem 3.21 of \cite{Hem04}. Such a conclusion is, in our special setting (cf. Theorem \ref{thm:A}), intuitively clear since
one can first look at the boundary components (thereby obtaining, for any $\g \geq 0$, a bijective correspondence between the set of handlebodies of genus $\g$ appearing on either side of \eqref{eq:CompDiffType}), then compare e.~g. the fundamental groups of the left-hand side and of the right-hand side to show that there is the same number of handles ($S^2\times S^1$), and finally conclude that there must be the same spherical space forms. In any event, at that stage a classical theorem by de Rham \cite{deR50} ensures that $X_i$ is actually isometric to $X'_i$ whenever $X_i$ is a spherical space form.

This first remark being given, we then note that, as a result of the way we constructed the model metrics in Subsection \ref{subs:ModelMet} (and also keeping in mind Remark \ref{rem:Ordering} therein), for any given $X$ the set of necks (or: handles) of any given type (type $C$, type $D$, type $T$, type $CT$) for model metrics on $DX$ is uniquely determined by the diffeomorphism type of $X$. For instance, the number of necks of type $CT$ (which, by the way, are the only iterated necks that appear in the model metrics) equals the number of \emph{strictly positive} genus handlebodies of $X$ minus one, and similar simple criteria exist for the other types of necks.

Therefore, once we fix the background manifold $X$ any two model metrics on $M=DX$ will only differ by the data that appear in the Gromov-Lawson connected sum constructions, specifically by the basepoints and the definining orthonormal bases of all summands. Hence, 
appealing to Lemma \ref{lemma.equivariant.GL.construction} and emplyoing, for notational convenience, the letter $\mathcal{O}$ to refer to the subset of model metrics \emph{restricted to $X$} (i.~e. with slight abuse of language we are consider the `upper half' of a model manifold $M=DX$)
we conclude in particular that, for any fixed $X$, the quotient $\mathcal{O}_{\mathcal{D}}:=\mathcal{O}/\mathcal{D}$ is path-connected.

That said, we shall now proceed to the proof of Theorem \ref{thm:B}.

\begin{proof} 
	First of all, we observe that the construction we presented in Appendix \ref{sec:push-in} (see, in particular, Lemma \ref{lem:pushMet}) allows to define a continuous map $\mathcal{J}: \mathcal{O}\times [0,1]\to\mathcal{M}$ (where, as in the statement of our theorem, $\mathcal{M}=\mathcal{M}_{R>0, H>0}$). In particular, this map (evaluated at time $t=1$) descends to a well-defined continuous map at the level of quotients $\mathcal{J}^{(1)}_{\mathcal{D}}:\mathcal{O}_{\mathcal{D}}\to \mathcal {M}/\mathcal{D}$, whose image shall be denoted by $\widetilde{\mathcal{O}_{\mathcal{D}}}$. These are equivalence classes of \emph{deformed} model metrics, still with positive scalar curvature but also with (strictly) mean-convex boundary.
	Observe that since $\mathcal{O}_{\mathcal{D}}$ is path-connected then $\widetilde{\mathcal{O}_{\mathcal{D}}}$ shall be path-connected as well.

	Given $(X^3, g_0)$ a manifold of positive scalar curvature and (strictly) mean-convex boundary, we can
combine Proposition \ref{thm:tgdef} and Proposition \ref{pro:Conclusive} to derive a continuous path of metrics connecting, in $\mathcal{M}_{R>0, H\geq 0}$, $\varphi^{\ast}g_0$ (for a suitable diffeomorphism $\varphi$) to a model metric $g_1\in\mathcal{O}$; if we then concatenate such a path with the path $\mu\mapsto \mathcal{J}(g_1, t)$ we obtain a path, henceforth denoted $(\alpha_{\mu})$ connecting $\varphi^{\ast}g_0$ to a deformed model metric.

Hence, if we directly apply part ii) of
Lemma \ref{lem:pushFam} (with the path $(\alpha_{\mu})$ as input), we can then obtain a new 
continuous path of metrics $(\tilde{\alpha}_{\mu})$ connecting, now in $\mathcal{M}_{R>0, H>0}$, the same initial metric $\tilde{\alpha}_0=\alpha_0$ to the same final metric $\tilde{\alpha}_{1}=\alpha_1$. It follows that there exists an
isotopy of classes starting at $[g_0]$ and ending at an element of $\widetilde{\mathcal{O}_{\mathcal{D}}}$, which is $\mathcal{J}^{(1)}_{\mathcal{D}}([g_1])$. As a result, given the arbitrariness of $[g_0]\in\mathcal{M}/\mathcal{D}$ and the aforementioned path-connectedness of $\widetilde{\mathcal{O}_{\mathcal{D}}}$, we conclude that $\mathcal{M}/\mathcal{D}$ is also path-connected, which is what we had to prove.
\end{proof}
 
More generally, we have the following statement:

\begin{thm}\label{thm:PathConnOther}
Let $X^3\not\simeq S^1\times S^1\times I $ be a connected, orientable, compact manifold with boundary. Then each of the following spaces, when not empty, is path-connected
\begin{enumerate}
	\item [i)]{$\mathcal{M}_{R>0, H>0}/\mathcal{D}$;}
	\item [ii)]{$\mathcal{M}_{R>0, H\geq 0}/\mathcal{D}$;}
	\item [iii)]{$\mathcal{M}_{R\geq 0, H>0}/\mathcal{D}$;}
	\item [iv)]{$\mathcal{M}_{R\geq 0, H\geq 0}/\mathcal{D}$.}
\end{enumerate}	
In the special case when $X^3\simeq D^3$ then the corresponding four spaces of metrics are also path-connected.
\end{thm}

We recall that the borderline case when $X^3\not\simeq S^1\times S^1\times I $ has already been characterized in Corollary \ref{cor:Equivalence}.

\begin{proof}
The path-connectedness of the space \emph{i)} is precisely the content of Theorem \ref{thm:B}. The conclusion for \emph{ii)} and \emph{iii)} follow from Theorem \ref{thm:B} by means of Lemma \ref{lem:pushMet} and Lemma \ref{lem:EigenfDef}, respectively. The conclusion for \emph{iv)} is obtained by reduction to one of the previous cases by means of on Proposition \ref{pro:Rigidity}.
The final statement, concerning the three-dimensional disk, is derived from the previous one by simply invoking the theorem of J. Cerf \cite{Cer68} ensuring the path-connectedness of $\mathcal{D}_{+}(D^3)$.
\end{proof}

\section{Motion through PSC metrics with minimal boundary}\label{sec:MinPaths}

We shall now turn our attention to spaces of metrics with minimal boundary. In particular, let us denote by $\mathcal{H}_{R>0, H=0}$ (respectively: $\mathcal{H}_{R\geq 0, H=0}$) the set of Riemannian metrics on $X$ having positive (respectively: non-negative) scalar curvature and minimal boundary.
The main statement we prove in this section is as follows:

\begin{thm}\label{thm:B'}
		Let $X^3\not\simeq S^1\times S^1\times I $ be a connected, orientable, compact manifold with boundary. Then each of the following moduli spaces, when not empty, is path-connected
		\begin{enumerate}
			\item [i)]{$\mathcal{H}_{R>0, H=0}/\mathcal{D}$;}
			\item [ii)]{$\mathcal{H}_{R\geq 0, H=0}/\mathcal{D}$.}
		\end{enumerate}	
		In the special case when $X^3\simeq D^3$ then the corresponding two spaces of metrics are also path-connected.
	\end{thm}

In fact, this result is obtained by combining Proposition \ref{pro:Conclusive} above (which provides a path through totally geodesic boundaries) and the deformation result below, which is a refinement of the argument we presented in Section \ref{sec:elldef}.

\begin{prop}\label{thm:TotGeod}
	Let $X^3$ be a connected, orientable, compact manifold with boundary, endowed with a Riemannian metric $g\in \mathcal{H}_{R>0, H=0}$. Then there exists a continuous path of smooth metrics on $X^3$, $\mu\in [0,1]\mapsto g_{\mu}\in \mathcal{H}_{R>0, H=0}$, such that $g_0=g$ and $g_1$ has  totally geodesic boundary.  $(X^3, g_1)$ doubles to a smooth Riemannian manifold of positive scalar curvature.
\end{prop}

\subsection{Conformal deformations with Robin boundary conditions}\label{subs:Ancy}

In order to perform the doubling construction (starting from a metric in $\mathcal{H}_{R\geq 0, H=0}$), we may first need to deform our background metric to strictly positive scalar curvature (whenever such a deformation is allowed).

\begin{lem}\label{pro:RigidityMin}
	Let $(X^3,g)$ be a connected, orientable, compact Riemannian manifold, such that the scalar curvature of $g$ is everywhere zero, and each boundary component is a minimal surface with respect to $g$. Suppose that 
	\emph{either}  $Ric_g$ is not identically zero, \emph{or} $\partial X$ is not totally geodesic with respect to $g$, then there is a small (isotopic) smooth perturbation $g'$ of $g$, such that the scalar curvature of $g'$ is strictly positive and every connected component of $\partial X$ is minimal.
	If instead $Ric_g$ is everywhere zero, and all the components of $\partial X$ are totally geodesic, then there exists a compact interval $I\subset \R$ such that $(X,g)$ is isometric to $S^1\times S^1\times I$ equipped with a flat metric.
\end{lem}

\begin{proof}
	  When $(X^3,g)$ is not Ricci flat,  we can argue as in the proof of Proposition \ref{pro:Rigidity}, case i).
	Therefore, assume that $Ric_g=0$ in $X$, and some boundary component $\Sigma$ of $X$ is (minimal but) not totally geodesic. By virtue of the variational characterization of the first eigenvalue of the Jacobi operator, it is seen at once that $\Sigma$ is unstable, thus we consider the normal deformations  defined by the associated first eigenfunction (taken positive and normalized, as usual) which determine subdomains $X_t\subset X$ whose boundary component in question has mean curvature $H_t>0$. Note that the boundary components of $\partial X$ that are totally geodesic in metric $g$ are not modified.
	 On $X_t$, we solve the elliptic eigenvalue problem (for the conformal Laplacian with Robin boundary condition):
	\[\begin{cases}\Delta_g \phi_t= -\lambda_1(t)\phi_t\quad  &\text{in }X_t,\\ \partial_{\nu_t} \phi_t=-\frac{1}{4}H_t \phi_t\quad &\text{on }\partial X_t,\end{cases}\]
	where $\lambda_1(t)$ is the first eigenvalue. Multiplying both sides by $\phi_t$ and integrating over $X_t$, we have that
	\[\lambda_1(t)\int_{X_t}\phi_t^2 dVol_g=\frac{1}{4}\int_{\partial X_t}H_t \phi_t^2 dS_g+\int_{X_t}|\nabla_g \phi_t|^2 dVol_g,
	\] which
	implies that $\lambda_1(t)>0$ for $t\in (0,\epsilon)$. In particular, the Riemannian manifold $(X_t,\phi_t^4 g_t)$ has strictly positive scalar curvature and minimal boundary. We then define the corresponding metrics on $X$ by means of the obvious diffeomorphism $\varphi_t: X\rightarrow X_t$.	
\end{proof}	

\begin{cor}\label{cor:MinEquiv}
	
	Let $X^3$ be a connected, orientable, compact manifold with boundary. Then 
	\[
	\mathcal{H}_{R>0, H=0}\neq\emptyset \ \ \Longleftrightarrow \ \ \mathcal{H}_{R\geq 0, H= 0}\neq\emptyset
	\]
	unless $X^3\simeq S^1\times S^1\times I$ (in which case the space $\mathcal{H}_{R\geq 0, H= 0}$ only contains flat metrics, making the boundary totally geodesic).
\end{cor}	

\begin{cor}\label{cor:TopCharMin}
	Let $X^3$ be a connected, orientable, compact manifold with boundary, such that $\mathcal{H}_{R\geq 0, H=0}\neq\emptyset$.
	Then either $X\simeq S^1\times S^1\times I$ or there exist integers $A,B,C\ge 0$, such that $X$ is diffeomorphic to a connected sum of the form
	\[ P_{\g_1}\#\cdots\# P_{\g_A}\# S^3/\Gamma_1\#\cdots\# S^3/\Gamma_B\# \left(\#_{i=1}^C (S^2\times S^1)\right).\]
	Here $P_{\g_i}$, $i\le A$, are genus $\g_i\geq 0$ handlebodies, and $\Gamma_i$, $i\le B$, are finite subgroups of $SO(4)$ acting freely on $S^3$. Viceversa, any such manifold supports Riemannian metrics of positive scalar curvature and minimal boundary.
\end{cor}

\subsection{A modification of the Gromov-Lawson doubling scheme}\label{subs:FinMin}

We shall now present the proof of Proposition \ref{thm:TotGeod}.

\begin{proof}
	Let $\epsilon_0>0$ be small enough that $\partial X\subset X$ has a tubular neighborhood of width $5\e_0$ and define, for $t\in (0,5\e_0)$ the set
	 $X_t:=\{x\in X: d_g(x,\partial X)\ge t\}$. For any such domain, we modify the metric $g$ by solving the principal eigenvalue problem:
	\begin{equation}\label{elliptic.eigenvalue.problem.distance.retraction}
	\begin{cases}\Delta_g \phi_t-\frac{1}{8}R_g \phi_t= -\lambda_1(t)\phi_t\quad  &\text{in }X_t,\\
	 \partial_{\nu_t} \phi_t=-\frac{1}{4}H_t \phi_t\quad &\text{on }\partial X_t,\end{cases}
	\end{equation}
	where $H_t$ is the mean curvature of $\partial X_t$ with respect to the outward-pointing normal vector field. It follows from Lemma \ref{lem:contdepdata} that the couple $(\lambda_1(t),\phi_t)$ varies smoothly with respect to the parameter $t$. In particular, since $t\mapsto\lambda_1(t)$ is continuous, and $\lambda_1(0)>0$ (since we assumed $g\in\mathcal{H}_{R>0, H=0}$), we conclude that, possibly by choosing $\epsilon_0$ smaller, for $t\in [0,5\epsilon_0]$ one has indeed $\lambda_1(t)>0$. 
	Denote $X'=X_{2\epsilon_0}$ and consider a smooth, positive extension to $X$ of  $\phi_{2\epsilon_0}$, which we do not rename. For $\epsilon\in (0,\epsilon_0/2)$ sufficiently small, the metric $g_1=\phi_{2\epsilon_0}^4 g$, restricted to $X_{2(\epsilon_0-\epsilon)}$, has positive scalar curvature. Now, with respect to this metric, we consider the Gromov-Lawson doubling construction of $X'$ in $\tilde{X}\times\mathbb{R}$, for $\tilde{X}=X_{2(\epsilon_0-\epsilon)}$.

	Thus, we endow $M\simeq DX$, with a non-smooth metric $g_1^M=g_1^M(\epsilon)$, by identifying it with 
	\[T_\epsilon (X')=\{(x,h)\in \tilde{X}\times \R: d^{\tilde{X}\times \R}((x,h),X'\times \{0\})=\epsilon\}.\]
		Let $R^{T_\epsilon (\partial X')}$ be the scalar curvature of $T_\epsilon (\partial X')$ in $\tilde{X}\times \R$, and $H^{Z_{\theta}}$ be the mean curvature of the level set $Z_{\theta}$ in $T_\epsilon (\partial X')$.
	It follows from the calculations we presented in the proof of Lemma \ref{lem:singpath}, see in particular equations \eqref{eq:EstScal} and \eqref{eq:eigenv}, that
	\[R^{T_\epsilon (\partial X')}=R^{\tilde{X}\times \R}+\cos(\theta)O(1),\quad H^{Z_{\theta}}=O(\epsilon).\]
	In particular, there exists $\theta_0>0$ independent of $\epsilon$, such that in the region defined by $\theta \in [\frac{\pi}{2}-\theta_0,\frac{\pi}{2}]$, we have $R^{T_\epsilon (\partial X')}>0$ uniformly in $\epsilon$.

\begin{figure}[htbp]
	\centering
	\includegraphics[scale=1.26]{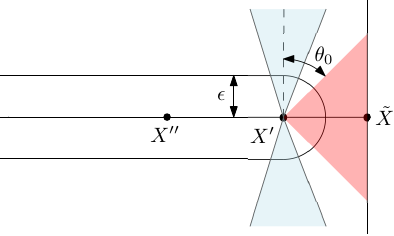}
	\caption{The Gromov-Lawson doubling construction, adapted to the minimal case.}
	\label{fig:Doubling}
\end{figure}

As in Section \ref{sec:elldef}, we employ the notation  $S_u^{\pm}$ to denote the level sets of the coordinate $u$, and we let $M_u^{\pm}$ be the connected component with boundary $S_u^{\pm}$ that contains $X_{\pm}''$. For notational simplicity, we will write $u,M_u,S_u$ in lieu of $u_+,M_u^+, S_u^+$, respectively. 

 We then apply Miao's smoothing construction described in Section \ref{subs:Miao}, along the interface $S_0$, in the region where $u\in [-\delta,\delta]$. Here $\delta=\delta(\epsilon)<\epsilon^2$ is to be chosen later, and at this stage we only require that $2\delta<\delta^{\ast}:=\epsilon \theta_0$, so that we have a positive lower bound for the scalar curvature of $M_u\setminus S_0$ when $u\in [-\delta^{\ast},0)\cup (0,\delta^{\ast}]$. Precisely, for the metric
\[g_1^M=du\otimes du+a(u),\]
and for fixed cut-off functions as in Section \ref{subs:Miao}, define  $a_\delta(u)$ by means of equation \eqref{def:1}. Then, 
\[g^M_{1,\delta} = du\otimes du+a_\delta(u)\]
is a smooth Riemannian metric that agrees with $g^M$ outside the region defined by $|u|<\delta/2$. Moreover, following the proof of Proposition \ref{pro:miao}, there exists a constant $C>0$ (independent of the choices of the parameters $\e$ and $\delta$) and a positive function $\omega$ such that the scalar curvature $R_{g^M_{1,\delta}}$ and the mean curvature $H^{S_u}_{g^M_{1,\delta}}$ satisfy the following conditions:
\[
\begin{cases}
R_{g^M_{1,\delta}}> -\omega(\e) &  \ \text{for} \ u\in [-\delta,\delta] \\
R_{g^M_{1,\delta}}> C  & \ \text{for} \ u\in [-\e_0,-\delta)\cup(\delta,\delta^*)\\
R_{g^M_{1,\delta}}>-C^{-1}  & \ \text{for} \ u\in [\delta^*,\pi\epsilon/2], 
\end{cases},  
\]
and 
\[
|H^{S_u}_{g^M_{1,\delta}}|<C^{-1} \epsilon, \ \ \text{for} \ u\in [-\epsilon_0, \pi\e/2).
\]

That being said, we consider the concatenation of five continuous paths of smooth metrics, within $\mathcal{H}_{R>0, H=0}$, defined as follows:
\begin{enumerate}
    \item [\emph{i)}] for $\mu\in [0,1]$, the metrics $(1-\mu+\mu \phi_0)^4 g$, where $u_0$ is the first eigenfunction defined by \eqref{elliptic.eigenvalue.problem.distance.retraction} on $X=X_0$, that determine an isotopy from $(X,g)$ to $(X, \phi_0^4 g)$;
    \item [\emph{ii)}] for $t\in [0,2\epsilon_0]$, the metrics $\phi_t^4 g$, pulled back through the diffeomorphisms $\Psi_t: X\rightarrow X_t$, where $\phi_t$ is the first eigenfunction of \eqref{elliptic.eigenvalue.problem.distance.retraction} on $(X_t,g)$;
    \item [\emph{iii)}] for $\mu\in [0,1]$, the metrics $(1-\mu+\mu(\psi_{2\epsilon_0}))^4 g_1$, where $\psi_{2\epsilon_0}$ is the first Neumann eigenfunction of the conformal Laplace operator on $X'\cong X_{2\epsilon_0}$, pulled back through the diffeomorphism $\Psi_{2\epsilon_0}: X\rightarrow X_{2\epsilon_0}$;
    \item [\emph{iv)}] for $t\in [2\epsilon_0, 2\epsilon_0+2\epsilon]$, the metrics $\psi^4_{t}g_1$, where $\psi_{t}$ is the first Robin eigenfunction of the conformal Laplace operator on $M_{2\epsilon_0-t}^+\cong X_{t}$, pulled back through the diffeomorphism $\Psi_{t}: X\rightarrow X_{t}$;
    \item [\emph{v)}] for $u\in [-2\epsilon, \pi\epsilon/2]$, the pull back of the metrics $(M_u^+,(\psi^{(\delta)}_u)^4 g^M_{1,\delta})$, where $\psi^{(\delta)}_u$ is the first Robin eigenfunction of the conformal Laplace operator on $M_u^+$, through the diffeomorphisms identifying $X'$ with $M_u^+$ (see Section \ref{subs:proof1}) and then via $\Psi_{2\e_0}:X\to X_{2\e_0}$. 
\end{enumerate}

Concerning case v), and similarly for case iv) with $g_1$ in lieu of $g^M_{1,\delta}$, by first Robin eigenvalue we mean the principal eigenvalue defined by
\begin{equation}\label{elliptic.eigenvalue.problem.tubular}
\begin{cases} \Delta_{g^M_{1,\delta}} \psi^{(\delta)}_u -\frac{1}{8}R_{g^M_{1,\delta}} \psi^{(\delta)}_u = -\lambda_1(u) \psi^{(\delta)}_u \quad &\text{in } M_u^+,\\ \partial_{\nu_u} \psi^{(\delta)}_u= -\frac{1}{4} H^{S_u}_{g^M_{1,\delta}} \psi^{(\delta)}_u \quad &\text{on }\partial M_u^+.
\end{cases}
\end{equation}
It follows from the bounds above, concerning $R_{g^M_{1,\delta}}$ and $H^{S_u}_{g^M_{1,\delta}}$, that one can choose $\epsilon$ small enough and $\delta=\delta(\epsilon)$ so that all the hypotheses of Lemma \ref{lem:poseigenv} are satisfied, and as a result $\lambda_1(u)$ admits a uniform, positive lower bound as we vary the parameter $u$. We note that, strictly speaking, in order to accomodate hypothesis 2. of Lemma \ref{lem:poseigenv} for a uniform constant $C_0>0$ (say $C_0=1$) one may have to consider for $f_{\omega}$, a smooth truncation $\chi(R^M_{\delta})$ in lieu of $R^M_{\delta}$, where $\chi:\mathbb{R}\to\mathbb{R}$ is a smooth non-decreasing function such that $\chi(s)=s$ for $s\leq 1$ and $\chi(s)=2$ for $s\geq 2$. This would take care of the fact that, for $\epsilon>0$ small, the scalar curvature of our family of metrics may not be uniformly bounded \emph{from above}, however such a modification does not in any way affect the validity of the argument.

 We conclude that $(X,g)$ can be deformed, through metrics in $\mathcal{H}_{R>0, H=0}$, to a positive scalar curvature metric with totally geodesic boundary, and, like we observed in Section \ref{subs:proof1}, $(M_{\pi\epsilon/2}^+,(\psi^{(\delta)}_{\pi\epsilon/2})^4 g^M_{1,\delta})$ doubles to a smooth reflexive manifold, which completes the proof.
\end{proof}

\appendix

		\section{Conformal deformations via Robin eigenfunctions}\label{sec:Neu}
	
	In this section, let us denote by $(X^n,g)$ a compact Riemannian manifold with boundary of dimension $n\geq 3$. Given suitably regular functions $f:X\to\mathbb{R}$ and $h:\partial X\to\mathbb{R}$, as specified below, we shall be concerned with an elliptic eigenvalue problem for a Schr\"odinger operator of the form
	\[
	L\phi=\Delta_g\phi-f \phi
	\]
	under a Robin boundary condition, that is to say a boundary condition of the form
	\begin{equation}\label{eq:Robin}
	\frac{\partial \phi}{\partial \nu}=-h\phi
	\end{equation}
	where $\nu$ is the outward-pointing, unit normal vector field along $\partial X$. 
	
	\begin{rmk}\label{rmk:shortcut}
	  For all purposes of the main body of the present paper, with the exception of the applications presented in Section \ref{sec:MinPaths}, we will only need to consider Neumann boundary conditions, i.~e. the case when $h=0$. In that case, the trace inequality given below, Lemma \ref{lem:TraceIneq} is not needed, and both the statement and the proof of \ref{lem:poseigenv} can be partly simplified.
	\end{rmk}
	
	If, say, the data $f$ and $h$ are continuous, it is well-known that there exists a sequence of eigenvalues $\left\{\lambda_i\right\}$, and associated eigenfunctions $\left\{\phi_i\right\}$ so that for every $i\geq 1$ one has
	\[
	\begin{cases}
	L\phi_i=-\lambda_i \phi_i \ & \text{on} \ X \\
	\frac{\partial \phi_i}{\partial \nu}=-h\phi_i & \text{on} \ \partial X;
	\end{cases}
	\] 
	 the family $\left\{\phi_i\right\}$ is an Hilbertian basis for $L^2$.
	 
	 The first eigenvalue has the variational characterization
	 \[
	 \lambda_1=\inf_{\phi\in H^1 (X,g)\setminus\left\{0\right\}}\frac{\int_X(|\nabla \phi|^2+f\phi^2)dVol_g+\int_{\partial X}h\phi^2\,dS_g}{\int_{X}\phi^2 dVol_g}
	 \]
	 which easily implies that it has unit multiplicity (i.~e. it is simple) and the associated eigenfunction can be taken to be everywhere positive on $X$. In this article, we tacitly adopt the following convention: \emph{the first eigenfunction is chosen to be positive and of unitary $L^2$ norm.}
	 
	 \

	 We need to present two additional facts about the first eigenvalue for the problem described above when the function $f$ varies smoothly in a family satisfying certain uniform bounds, which are inspired by the situation arising in the regularization scheme given in Section 3 of \cite{Miao02}.
	  
	  The first result concerns the dependence of $\lambda_1$ on both the data $f, h$ and a possibly varying background metric $g$.

	  \begin{lem}[cf. Lemma A.1 in \cite{MS15}]\label{lem:contdepdata}
	  	Let $\Omega\subset \R^d$ be open, and suppose $\Omega \ni \omega\mapsto g(\omega)$ is a smooth family of metrics on a compact manifold $X$. Let $\omega\mapsto L_{\omega}$ be a smooth family of self-adjoint second order linear elliptic operators of the form
	  	\[
	  	L_{\omega}=\Delta_{g(\omega)}-f_{\omega}
	  	\] whose coefficients depend smoothly on the parameter $\omega$, and whose first eigenvalue $\lambda_1(\omega)$ has a one-dimensional eigenspace (which holds true for $\partial X\neq\emptyset$ under Robin boundary conditions, as in \eqref{eq:Robin} with $h=h_{\omega}$). Then $\lambda_1: \Omega\rightarrow \R$ is smooth, and there exists a smooth map $u:\Omega\times X\rightarrow \R$ so that $u(\omega,\cdot)$ is a normalized, positive eigenfunction of $L(p)$ with Robin eigenvalue $\lambda_1(\omega)$.
	  \end{lem}

	  The second result provides a sufficient condition for the first eigenvalue to be uniformly positive as we vary some parameter. In order to state and prove it, we first need an ancillary result, that is an $L^2$-trace inequality for a family of data and metrics satisfying certain uniform bounds. 
	  
	  \begin{lem}\label{lem:TraceIneq}
		Let $(X,g)$ be a smooth Riemannian manifold with boundary, and $\Lambda>0$ a constant. There exists $\sigma_0>0$ depending on $(X,g)$, but not on $\Lambda$, such that for every $\sigma<\sigma_0$, if $g_\omega$ is a smooth Riemannian metrics satisfying
		\begin{equation}\label{assumption.C0.small.C1.bounded}
			\|g_\omega - g\|_{C^0(X,g)}\le \sigma,\quad \|g_\omega-g\|_{C^1(X,g)}\le \Lambda,\quad \forall \omega\in \Omega,
		\end{equation}
		then there exists a constant $C_1=C_1(X,g,\Lambda)$ such that for any $\phi\in H^1(X,g_\omega)$, we have
		\begin{equation}\label{eq.L2.trace}
		\int_{\partial X} \phi^2 dS_{g_\omega}\le C_1\left( \int_{X} (|\nabla_{g_\omega} \phi|^2 + \phi^2)  dVol_{g_\omega}\right).
		\end{equation}
	\end{lem}
	
	\begin{proof}
		Denote by $r_0$ the injectivity radius of $(X,g)$. Let $\eta:[0,\infty)\rightarrow $ be a smooth, non-increasing function such that
		$0\le \eta\le 1$, $\eta(x)=1$ when $0\le x\le \frac12$ and $\eta(x)=0$ when $x\ge 1$; let then $d(p)=d_g(p,\partial X)$. We define a smooth vector field $\xi$ on $X$ by setting 
		\[\xi(p)=-(\eta\circ d)(p)\cdot \nabla_g d(p).\]
		We have that $\xi=\nu$ on $\partial X$, where $\nu$ is the outward unit normal vector of $\partial X$, and $|\xi|_g\le 1$ in $X$. Thus, there exists $\sigma_0>0$ such that when $\sigma<\sigma_0$ and $\|g_\omega-g\|_{C^0(g)}\le \sigma$, $\langle \xi,\nu_{g_\omega}\rangle_{g_\omega}\ge \frac 12$ along $\partial X$, and $|\xi|_{g_\omega}\le 2$ inside $X$. 
		
		Given any $C^1$ function $\phi$, one can apply the divergence theorem in $(X,g_\omega)$ and get
		\[\int_{\partial X} \langle \phi^2 \xi,\nu_{g_\omega}\rangle_{g_\omega} dS_{g_\omega } =\int_X \Div_{g_\omega} (\phi^2 \xi) dVol_{g_\omega}.\]

		Hence, it follows that
		\begin{align*}
			\frac12 \int_{\partial X} \phi^2 dS_{g_\omega} &\le \int_X (|\phi\langle \nabla_{g_\omega} \phi, \xi \rangle_{g_\omega}|  + \phi^2 |\Div_{g_\omega} \xi|) dVol_{g_\omega}\\
			&\le \int_X (\phi^2 + |\nabla_{g_\omega} \phi|^2 + \phi^2 |\Div_{g_\omega} \xi|) dVol_{g_\omega},
		\end{align*}
		where we have used that $|\phi\langle \nabla_{g_\omega} \phi, \xi\rangle_{g_\omega}|\le \frac14 \phi^2|\xi|^2_{g_\omega} + |\nabla_{g_\omega} \phi|^2$ and that $|\xi|_{g_\omega}\le 2$.
		
		Now we consider the divergence term. In a local coordinate system $\{x_j\}$, let $\xi=(\xi^1,\cdots, \xi^n)$. Then
		\[\Div_{g_\omega}\xi = \sum_i \partial_i \xi^i  + \sum_{i,l} \Gamma_{il}^l \xi^i.\]
		Since $\|g_\omega-g\|_{C^1(X,g)}\le \Lambda$, $|\Div_{g_\omega}\xi - \Div_g \xi|\leq\Lambda'$, where $\Lambda'$ is a constant depending only on $\Lambda$. Therefore \eqref{eq.L2.trace} holds for a uniform constant $C_1$ depending only on $X,g,\Lambda$, provided $\sigma<\sigma_0$.

	\end{proof} 
	  
	We can now proceed with the aforementioned uniform eigenvalue estimate.

\begin{lem}\label{lem:poseigenv}
	Given constants $\Lambda, \kappa, C_0>0$, a smooth Riemannian manifold with boundary $(X,g)$, there exist positive constants $\sigma_0=\sigma_0(X,g)$ and $\tau_0=\tau_0(X,g,\Lambda,\kappa,C_0)$, such that, for any smooth Riemannian metric $g_\omega$ on $X$,  functions $f_\omega\in C^\infty(X)$, $h_\omega\in C^\infty(\partial X)$, $\omega\in \Omega$, and constants $\sigma\in (0,\sigma_0), \tau\in(0,\tau_0)$ if:
		\begin{enumerate}
			\item [1.] $\|g_\omega-g\|_{C^0(X,g)}\le \sigma,  \|g_\omega-g\|_{C^1(X,g)}\le \Lambda$,
			\item [2.] $\|f_\omega\|_{L^{\frac n2 +1}(X,g_\omega)}\le C_0$,
			\item [3.] $\|(f_\omega)_{-}\|_{L^{\frac n2}(X,g_\omega)}<\tau$,  $\|h_\omega\|_{L^\infty(\partial X, g_\omega)}<\tau$, 
			\item [4.] there exists a smooth domain $X_1=X_1(\omega)\subset X$, $Vol_{g_\omega} (X_1)\le \tau Vol_{g_\omega}(X)$, and $f_\omega\ge \kappa>0$ on $X\setminus X_1$,
		\end{enumerate}
		then the first eigenvalue $\lambda_1(\omega)$ of the elliptic problem
		\[\begin{cases}
		\Delta_{g_\omega} \phi - f_\omega \phi = -\lambda_1(\omega) \phi \quad &\text{in }X,\\ \partial_{\nu_{g_\omega}} \phi= -h_\omega \phi \quad &\text{on }\partial X,
		\end{cases}\]
		satisfies an estimate of the form $\lambda_1(\omega)\ge \lambda_*>0$ for a constant $\lambda_{\ast}=\lambda_{\ast}(X, g,\Lambda, \kappa, C_0)$.

\end{lem}

\begin{proof}
    	By setting $\sigma_0$ sufficiently small as requested by Lemma \ref{lem:TraceIneq}, assumption 1. above implies that there exists a constant $C_1=C_1(X,g_0,\Lambda)$ such that \eqref{eq.L2.trace} holds. By enlarging $C_1$ if necessary, we may further assume the validity of the Sobolev inequality with a uniform constant $C_1=C_1(X,g_0,\Lambda)>1$. In other words, for any function $\phi\in H^1(X,g_\omega)$, we have:
	\[\left(\int_X |\phi|^{\frac{2n}{n-2}}dVol_{g_\omega}\right)^{\frac{n-2}{n}}\le C_1\left(\int_{X} |\nabla_{g_\omega} \phi|^2 dVol_{g_\omega} +\int_{X} \phi^2 dVol_{g_\omega} \right),\]
	and
	\[\int_{\partial X} \phi^2 dS_{g_\omega}\le C_1\int_X (|\nabla_{g_\omega} \phi|^2 + \phi^2) dVol_{g_\omega}. \]

		For any $H^1$ function $\phi$, using assumption 3. one can estimate
	\begin{multline*}\int_X (|\nabla_{g_\omega}\phi|^2 + f_{\omega} \phi^2) dVol_{g_\omega} +\int_{\partial X} h_{\omega} \phi^2 dS_{g_\omega} \\ \ge \int_X (|\nabla_{g_\omega} \phi|^2 - (f_{\omega})_{-}\phi^2) dVol_{g_\omega} 
	- \tau \int_{\partial X} \phi^2 dS_{g_\omega}
	\end{multline*}
	and by H\"older's inequality
	\begin{align*}
	\int_X (f_{\omega})_{-}\phi^2 dVol_{g_\omega} &\le \left(\int_X (f_{\omega})_{-}^{\frac n2} dVol_{g_\omega}\right)^{\frac 2n}\left(\int_X |\phi|^{\frac{2n}{n-2}}dVol_{g_\omega}\right)^{\frac{n-2}{n}}\\
		&\le \tau \left(\int_X \phi^\frac{2n}{n-2}dVol_{g_\omega}\right)^{\frac{n-2}{n}} \\
		& \le \tau C_1\left(\int_X |\nabla_{g_{\omega}} \phi|^2 dVol_{g_\omega} + \int_{X} \phi^2 dVol_{g_\omega} \right),
	\end{align*}
	which then implies
	\begin{align*}
		& \int_X (|\nabla_{g_\omega} \phi|^2 + f_{\omega} \phi^2)dVol_{g_\omega} +\int_{\partial X} h_{\omega} \phi^2 dS_{g_\omega} \\
		&\ge (1-\tau C_1) \int_X |\nabla_{g_\omega} \phi|^2 
		 -\tau C_1\int_{X}\phi^2 dVol_{g_\omega} - \tau \int_{\partial X} \phi^2 dS_{g_\omega} \\
			&\ge (1-2\tau C_1)\int_X |\nabla_{g_{\omega}} \phi|^2 
			 -2\tau C_1 \int_{X} \phi^2\\
			&\ge -2\tau C_1 \int_X \phi^2,
	\end{align*}
	provided $\tau< (2C_1)^{-1}$.
    
    From the variational characterization of the first eigenvalue, we have that 
    \[ 
    \lambda_1(\omega)\ge- 2\tau C_1 \ge -1.
    \] Let us now improve the coarse bound above in the case $\phi=\phi_1^{(\omega)}$, namely for the first eigenfunction $\phi_1^{(\omega)}>0$ associated to $\lambda_1(\omega)$. 
	We shall prove the estimate $\lambda_1(\omega)\ge \lambda_*$ with 
	\[\lambda_*=\min\{1,\frac12 C_2^{-2}\kappa\},\]
	for $C_2$ a constant depending only on $X, g, \Lambda, \kappa, C_0$. 
    
    Assume that $\lambda_1(\omega)\le 1$. By Moser's Harnack inequality, there exists a constant $C_2$ depending on $C_1, C_0$ (cf. assumption 2.)  and $(X,g_0)$, but otherwise uniform in $(g_\omega,f_\omega,h_\omega)$), such that
	\[\sup_{p\in X} \phi^{(\omega)}_1(p)\le C_2\inf_{p\in X} \phi^{(\omega)}_1 (p).\]
	As a result, we can derive the following chain of inequalities
	\begin{align*}
		&\lambda_1(\omega)=\frac{\int_X (|\nabla_{g_\omega}\phi^{(\omega)}_1|^2 + f_{\omega} (\phi^{(\omega)}_1)^2) dVol_{g_{\omega}}+ \int_{\partial X} h_{\omega} (\phi^{(\omega)}_1)^2 dS_{g_\omega}}{\|\phi\|^{2}_{L^2(X,g_{\omega})}} \\
		&\ge \frac{(1-\tau C_1)\int_X |\nabla_{g_\omega} \phi^{(\omega)}_1|^2 dVol_{g_\omega}  +\int_X (f_{\omega}-\tau C_1) (\phi^{(\omega)}_1)^2 dVol_{g_\omega}}{\|\phi\|^{2}_{L^2(X,g_{\omega})}} \\
		&\ge \frac{\inf_X (\phi^{(\omega)}_1)^2 \int_X (f_{\omega})_{+} dVol_{g_\omega}-  \sup_X (\phi^{(\omega)}_1)^2 \int_X ((f_{\omega})_{-}+\tau C_1) dVol_{g_\omega}}{\|\phi\|^{2}_{L^2(X,g_{\omega})}} \\
		&\ge \frac{\int_X (f_{\omega})_+ dVol_{g_\omega} - C_2^4 \int_X (f_\omega)_{-} dVol_{g_\omega} - \tau C_1 C_2^4 Vol_{g_\omega}(X)}{C_2^{2}Vol_{g_{\omega}}(X)}  .
	\end{align*}

On the other hand, again by H\"older's inequality, 
\[
\int_X (f_\omega)_{-} dVol_{g_\omega} \le \|(f_\omega)_{-}\|_{L^\frac n2 (X,g_{\omega})} Vol_{g_{\omega}}(X)^{\frac{n-2}{n}}\le \tau Vol_{g_{\omega}}(X)^{\frac{n-2}{n}}.
\]
Thus,
	\begin{align*}
		&\lambda_1(\omega) \\
		& \ge \frac{\kappa (Vol_{g_{\omega}}(X)-Vol_{g_{\omega}}(X_1)) - \tau C_2^4 Vol_{g_{\omega}}(X)^{\frac{n-2}{n}}   - \tau C_1C_2^4 Vol_{g_{\omega}}(X)}
		{C_2^{2} Vol_{g_{\omega}}(X)}\\
						&\ge \frac12 C_2^{-2}\kappa 
						\ge \lambda_*,
	\end{align*}
	provided we simply require (by virtue of assumption 4.)
	\[\tau\leq \frac{\kappa}{2\left(\kappa+ C_2^4 Vol_{g_{\omega}}(X)^{-\frac 2n}+C_1C_2^4\right)}.\]
	Keeping in mind the fact that $g$ and $g_{\omega}$ are $C^0$-close, possibly by taking $\sigma_0$ smaller, we choose 
	\[\tau_0=\frac{1}{10}\min\left\{(3C_1^2)^{-1}, \frac{\kappa}{2\left(\kappa + C_2^4 Vol_g(X)^{-\frac 2n} +C_1C_2^4\right)} \right\}\]
	which allows to complete the proof.

\end{proof}

\section{Convex curvature conditions}\label{sec:ConvexCond}

Here, we would like to discuss why the curvature conditions appearing in our main theorems
are \emph{convex} within any given conformal class. The result we provide is the following:

\begin{lem}\label{lem:convex}
	Let $n\geq 3$, and let $(X^n,g)$ be a compact Riemannian manifold with boundary. We have that the following sets are convex:
	\begin{enumerate}
	\item [i)] {$\left\{\phi\in C_{>0}^{\infty}(X) \ : \ R_{\phi^{4/(n-2)}g}>0\right\}$}	
	\item[ii)] {$\left\{\phi\in C_{>0}^{\infty}(X) \ : \ R_{\phi^{4/(n-2)}g}\geq 0\right\}$}	\item [iii)]{$\left\{\phi\in C_{>0}^{\infty}(X) \ : \ H_{\phi^{4/(n-2)}g}>0\right\}$}	
	\item [iv)]{$\left\{\phi\in C_{>0}^{\infty}(X) \ : \ H_{\phi^{4/(n-2)}g}\geq 0\right\}$}	
	\end{enumerate}	
where $R_{\tilde{g}}$ denotes the scalar curvature of $(X,\tilde{g})$, while $H_{\tilde{g}}$ denotes the mean curvature of its boundary.
In particular, the sets defined in i) and iii) are \emph{strictly} convex.
Hence, denoted by $[g]$ the conformal class of the metric $g$,  the sets
\[
\mathcal{M}^{[g]}_{R>0, H>0}:=\left\{\tilde{g}\in [g] \ : \ R_{\tilde{g}}>0, H_{\tilde{g}}>0 \right\},
\]
\[
\mathcal{M}^{[g]}_{R>0, H\geq 0}:=\left\{\tilde{g}\in [g] \ : \ R_{\tilde{g}}>0, H_{\tilde{g}}\geq 0 \right\},
\] 
\[
\mathcal{M}^{[g]}_{R\geq 0, H> 0}:=\left\{\tilde{g}\in [g] \ : \ R_{\tilde{g}}\geq 0, H_{\tilde{g}}>0 \right\},
\] 
\[
\mathcal{M}^{[g]}_{R\geq 0, H\geq 0}:=\left\{\tilde{g}\in [g] \ : \ R_{\tilde{g}}\geq 0, H_{\tilde{g}}\geq 0 \right\},
\] 
are contractible.
\end{lem}	

\begin{proof}
If we let $\tilde{g}=\phi^{4/(n-2)}g$, we know that the scalar curvature changes according to the equation
\begin{equation}\label{eq:changeR}
R_{\tilde{g}}=\phi^{-\frac{n+2}{n-2}}(R_{g}\phi-c(n)\Delta_g \phi)
\end{equation}
while the mean curvature of the boundary obeys the law
\begin{equation}\label{eq:changeH}
H_{\tilde{g}}=\phi^{-\frac{n}{n-2}}\left(H_{g}\phi+\frac{c(n)}{2}\frac{\partial \phi}{\partial \nu}\right)
\end{equation}
where $\nu$ denotes the outward-pointing unit normal to $\partial X$, and we have conveniently introduced the constant $c(n)=4(n-1)/(n-2)$.
Therefore, it follows that for any $\lambda_1,\lambda_2\in [0,1]$ with $\lambda_1+\lambda_2=1$ one can write
\begin{multline*}
R_{(\lambda_1 \phi_1+\lambda_2 \phi_2)^{4/(n-2)}g}\\
=\lambda_1\left(\frac{\lambda_1 \phi_1+\lambda_2 \phi_2}{\phi_1}\right)^{-\frac{n+2}{n-2}}R_{\phi^{4/(n-2)}_1 g}+\lambda_2\left(\frac{\lambda_1 \phi_1+\lambda_2 \phi_2}{\phi_2}\right)^{-\frac{n+2}{n-2}}R_{\phi^{4/(n-2)}_2 g}
\end{multline*}
and similarly
\begin{multline*}
H_{(\lambda_1 \phi_1+\lambda_2 \phi_2)^{4/(n-2)}g}\\
=\lambda_1\left(\frac{\lambda_1 \phi_1+\lambda_2 \phi_2}{\phi_1}\right)^{-\frac{n}{n-2}}H_{\phi^{4/(n-2)}_1 g}+\lambda_2\left(\frac{\lambda_1 \phi_1+\lambda_2 \phi_2}{\phi_2}\right)^{-\frac{n}{n-2}}H_{\phi^{4/(n-2)}_2 g}.
\end{multline*}
Given these two formulae, all conclusions are straightforward.
\end{proof}	

\begin{rmk}
	\label{rem:ConvClosedCase}
The same conclusion as in \emph{i)} applies, as a special case, to compact manifolds without boundary. In particular, the subset $\mathcal{R}^{[g]}$ consisting of those metrics, in the conformal class of $[g]$, having positive scalar curvature,	is contractible.
\end{rmk}
	
\section{Conformal deformations via cutoff functions}\label{sec:push-in}
	
In this appendix, we explain how to deform Riemannian metrics on a compact 3-manifold in order to increase the mean curvature of the boundary  without significantly decreasing the scalar curvature in the interior. For expository convenience, we shall first present the case when the deformation only involves a given metric, and then discuss the uniform deformation of a continuous path.
	
\begin{lem}\label{lem:pushMet}
		Let $n\geq 3$, and let $(X^n,g)$ be a compact Riemannian manifold with boundary. If $(X,g)$ has positive scalar curvature and weakly mean-convex boundary, there exists a continuous path of smooth metrics $[0,1]\ni \mu \mapsto g_{\mu} \in\mathcal{M}_{R>0, H\geq 0}$ such that $g(0)=g$ and for any $\mu>0$ one has in fact $g_{\mu} \in\mathcal{M}_{R>0, H>0}$. Furthermore, the path in question depends continuously on the input metric $g$.
\end{lem}	

Of course, by continuity one can always accomodate the requirement that the scalar curvature of the endpoint metric of the path be larger than, say, $\inf R_g /2$.

\begin{proof}
Consider a smooth, non-increasing convex function $\psi: [0,\infty)\to [0,\infty)$ such that 
\[
\psi(t)=
\begin{cases}
1-t & \text{if} \ 0\leq t\leq 1/2; \\
0 & \text{if} \ t\geq 1, 
\end{cases}
\]
and for $\eta, \epsilon\in (0,1)$ set 
\[
\phi_{\eta,\epsilon}(t)=1+\epsilon \eta\psi\left(t/\eta\right).
\]
Let us assume, for the sake of simplicity, that $\partial X$ is connected (the case of multiple connected components is handled working one connected component at a time).
If we pick $\eta$ smaller than the size of a tubular neighborhood of $\partial X$ in $(X,g)$ and let $z$ denote the distance function from $\partial X$, one can consider a conformal metric of the form $\phi^{4/(n-2)}_{\eta,\epsilon}(z)g$: we have that (by equations \eqref{eq:changeR} and \eqref{eq:changeH}) one can choose
\[
\epsilon< C\eta \frac{\inf R_g}{2},
\] 
where $C$ only depends on $n$ and the function $\psi$, so that this metric has scalar curvature bounded from below by $\inf R_g /2$ and boundary mean curvature bounded from below by $\epsilon/2$. The desired path $g_{\mu}$ is then obtained by simply varying the parameter in the above construction, namely setting for $\mu\in [0,1]$
\[
g_{\mu}=\phi^{4/(n-2)}_{\eta,\mu\epsilon}(z)g.
\]
\end{proof}
	
\begin{lem}\label{lem:pushFam}
	In the setting above, let $[0,1]\ni\mu\mapsto \alpha_{\mu}\in \mathcal{M}_{R>0,H\geq 0}$ be a continuous path of smooth metrics on $X$.
	\begin{enumerate}
	    \item [i)] if $\partial X$ is strictly mean-convex with respect to $\alpha_0$,  then there exists a continuous path $[0,1]\ni\mu\mapsto \tilde{\alpha}_{\mu}\in \mathcal{M}_{R>0,H>0}$ such that $\tilde{\alpha}_0=\alpha_0$ and $\tilde{\alpha}_1$ can be chosen depending on $\alpha_1$ only;
	    \item [ii)] if $\partial X$ is strictly mean-convex with respect to $\alpha_0$ and with respect to $\alpha_1$,  then there exists a continuous path $[0,1]\ni\mu\mapsto \tilde{\alpha}_{\mu}\in \mathcal{M}_{R>0,H>0}$ such that $\tilde{\alpha}_0=\alpha_0$ and $\tilde{\alpha}_1=\alpha_1$. 
	\end{enumerate} 
	\end{lem}		

\begin{proof}
We simply consider two parametric versions of the construction presented above for Lemma \ref{lem:pushMet}. Concerning part i), for $\mu\in [0,1]$ we shall set
\[
\tilde{\alpha}_{\mu}=\phi^{4/(n-2)}_{\eta_{\mu},\epsilon_{\mu}}(z_{\alpha_{\mu}})\alpha_{\mu}.
\]
The size $\eta_{\mu}$ of the tubular neighborhood of $\partial X$ in metric $\alpha_{\mu}$ can be bound uniformly, by compactness of $[0,1]$ and continuous dependence of the injectivity radius of $\partial X$ with respect to $C^2$ variations of the metric (see Ehrlich \cite{Ehr74}), and one takes
\[
\epsilon_{\mu}=C\eta_{\mu} \mu \frac{\inf R_{\alpha_{\mu}}}{2}.
\]
We obtain the desired path $\tilde{\alpha}_{\mu}$ by varying $\mu\in [0,1]$. 
Similarly, for part ii) we shall follow the same argument and simply take
\[
\epsilon_{\mu}=C\eta_{\mu} (\mu-\mu^2) \frac{\inf R_{\alpha_{\mu}}}{2}.
\]
\end{proof}	

There also exist dual deformation results that allow to slightly bump-up the scalar curvature at the interior at the price of decreasing the mean curvature of the boundary. For our purposes, we shall only need the following statement.

\begin{lem}\label{lem:EigenfDef}
	Let $(X^n,g)$ be a Riemannian manifold of non-negative scalar curvature and (strictly) mean-convex boundary. Then there is a continuous path of smooth metrics $[0,1]\ni\mu\mapsto g_\mu\in\mathcal{M}_{R\geq 0, H>0}$ such that $g_0=g$, and for any $\mu>0$ one has in fact $g_{\mu}\in \mathcal{M}_{R>0, H>0}$. Furthermore, the path in question depends continuously on the input metric $g$.
\end{lem}

\begin{proof}
	With the same notation as in the proof of Lemma \ref{lem:pushMet}, consider the smooth family of functions
	\[
	\check{\phi}_{\eta,\epsilon}(t)=1-\epsilon \eta\psi\left(t/\eta\right)
	\]
	which are positive for all $\e$ sufficiently small. Then, when pre-composing with the distance function $z$ from $\partial X$ we observe that  $\Delta_g (\check{\phi}_{\eta,\epsilon}(z))\leq 0$ on $X$, with strict inequality near the boundary. 
	With this conformal factor, the metrics $\check{g}_\mu=\check{\phi}_{\eta,\epsilon}^{4/n-2}(z)g$, $\mu\in [0,1]$ form a smooth family such that $R_{\check{g}_\mu}\ge 0$, $R_{\check{g}_\mu}>0$ somewhere, and $H_{\check{g}_\mu}>0$ for $\e$ sufficiently small. For each $\mu\in [0,1]$, consider the eigenvalue problem
	\[\lambda_1(\check{g}_\mu)=\frac{\int_X (c(n)|\nabla_{\check{g}_\mu}\zeta|^2+R_{\check{g}_\mu}\zeta^2) dVol_{\check{g}_\mu}}{\int_X \zeta^2dVol_{\check{g}_\mu}}.\]
	Since $R_{\check{g}_\mu}\ge 0$, and $R_{\check{g}_\mu}>0$ on an open subset of $X$, it is clear that $\lambda_1(\check{g}_\mu)>0$ for $\mu\in (0,1]$. Take the associated eigenfunction $\zeta_\mu$, under our usual $L^2$ normalization condition. By Lemma \ref{lem:contdepdata}, $\lambda_1(\check{g}_\mu)$ is a smooth function of $\mu$, $\zeta_\mu$ is a smooth family of $C^{\infty}$ functions. Then the metrics $g_{\mu}= \zeta_{\mu}^{4/n-2} \check{g}_\mu$ satisfy the properties we claimed.
\end{proof}

\bibliographystyle{amsbook}

\end{document}